\documentclass[iicol,sn-mathphys]{sn-jnl}
\usepackage{upgreek}

\usepackage[table]{xcolor}
\usepackage{stmaryrd}
\newtheorem{lemma}{Lemma}
\usepackage{graphicx}
\usepackage{caption}
\usepackage[labelformat=simple]{subcaption}

\usepackage{enumitem}

\usepackage{array}
\newcolumntype{L}[1]{>{\raggedright\let\newline\\\arraybackslash\hspace{0pt}}m{#1}}
\newcolumntype{C}[1]{>{\centering\let\newline\\\arraybackslash\hspace{0pt}}m{#1}}
\newcolumntype{R}[1]{>{\raggedleft\let\newline\\\arraybackslash\hspace{0pt}}m{#1}}

\newcommand{\sM}{\begin{array}{ccccccccc}}
		\newcommand{\eM}{\end{array}}

\newcommand{\pd}[2]{\displaystyle\frac{\displaystyle\partial #1}{\displaystyle\partial #2}}

\newcommand{\td}[2]{\frac{{\rm d} #1}{{\rm d} #2}}

\newcommand{\Grad}[1]{{\rm Grad}\left( #1 \right)}
\newcommand{\Div}[1]{{\rm Div }\left( #1 \right)}

\renewcommand{\det}[1]{{\rm det }\lb #1 \rb}

\newcommand{\lb}{\left(}
\newcommand{\rb}{\right)}
\newcommand{\lbb}{\llbracket}
\newcommand{\rbb}{\rrbracket}

\newcommand{\la}{\langle}
\newcommand{\ra}{\rangle}

\newcommand{\sv}{\lb\begin{array}{ccccccccccccccccc}}
		\newcommand{\sV}{\begin{bmatrix}}
				\newcommand{\eV}{\end{bmatrix}}
		\newcommand{\ev}{\end{array}\rb}

\newcommand{\fempty}[1]{{}}

\newcommand{\sty}[1]{\mbox{\boldmath $#1$}}

\newcommand{\styy}[1]{{\mathbb{#1}}}

\newcommand{\fa}{\sty{ a}}

\newcommand{\fc}{\sty{ c}}

\newcommand{\fe}{\sty{ e}}
\newcommand{\ff}{\sty{ f}}

\newcommand{\fl}{\sty{ l}}
\newcommand{\fm}{\sty{ m}}

\newcommand{\ft}{\sty{ t}}
\newcommand{\fu}{\sty{ u}}
\newcommand{\fv}{\sty{ v}}

\newcommand{\fx}{\sty{ x}}

\newcommand{\fA}{\sty{ A}}
\newcommand{\fB}{\sty{ B}}
\newcommand{\fC}{\sty{ C}}

\newcommand{\fE}{\sty{ E}}
\newcommand{\fF}{\sty{ F}}

\newcommand{\fG}{\sty{ G}}
\newcommand{\fH}{\sty{ H}}
\newcommand{\fI}{\sty{ I}}

\newcommand{\fM}{\sty{ M}}
\newcommand{\fN}{\sty{ N}}

\newcommand{\fP}{\sty{ P}}

\newcommand{\fS}{\sty{ S}}
\newcommand{\fT}{\sty{ T}}
\newcommand{\fU}{\sty{ U}}

\newcommand{\fX}{\sty{ X}}

\newcommand{\fzero}{\sty{ 0}}
\newcommand{\ffA}{\styy{ A}}

\newcommand{\ffC}{\styy{ C}}

\newcommand{\ffI}{\styy{ I}}

\newcommand{\ffP}{\styy{ P}}

\newcommand{\ffR}{\styy{ R}}

\newcommand{\fsigma}{\mbox{\boldmath $\sigma$}}

\newcommand{\fDelta}{\mbox{\boldmath $\Delta$}}

\newcommand{\fxi}{\mbox{\boldmath $\xi $}}

\newcommand{\feps}{\mbox{\boldmath $\varepsilon $}}

\newcommand{\fGamma}{\mbox{\boldmath $\Gamma $}}

\newcommand{\fvarphi}{\mbox{\boldmath $\varphi $}}

\newcommand{\cA}{{\cal A}}
\newcommand{\cB}{{\cal B}}

\newcommand{\cE}{{\cal E}}

\newcommand{\cL}{{\cal L}}

\newcommand{\cO}{{\cal O}}

\newcommand{\cS}{{\cal S}}

\newcommand{\cV}{{\cal V}}

\newcommand{\scrS}{\mathscr{S}}

\newcommand{\ol}[1]{\overline{#1}}

\newcommand{\ul}[1]{\underline{#1}}
\newcommand{\ull}[1]{\ul{\ul{#1}}}

\newcommand{\WT}[1]{\widetilde{#1}}

\makeatletter
\definecolor{Sblueaa}{cmyk}{1,0.6,0,0}
\definecolor{Sbluea}{cmyk}{1,0.4,0,0}
\definecolor{Sblueb}{cmyk}{0.7,0.2,0,0}
\definecolor{Sbluec}{cmyk}{0.5,0.1,0,0}
\definecolor{Sblued}{cmyk}{0.3,0.05,0,0}
\definecolor{Sbluee}{cmyk}{0.15,0.04,0,0}
\definecolor{Svbluea}{cmyk}{0.9,0.6,0,0}
\definecolor{Svblueb}{cmyk}{0.68,0.4,0,0}
\definecolor{Svbluec}{cmyk}{0.45,0.26,0,0}
\definecolor{Svblued}{cmyk}{0.27,0.12,0,0}
\definecolor{Sblacka}{cmyk}{0.5,0.2,0.2,0.85}
\definecolor{Sblackb}{cmyk}{0.35,0.14,0.14,0.6}
\definecolor{Sblackc}{cmyk}{0.25,0.1,0.1,0.43}
\definecolor{Sblackd}{cmyk}{0.15,0.06,0.06,0.26}
\definecolor{Sblacke}{rgb}{0.827451,0.8509804,0.8627451}
\definecolor{Sred100}{HTML}{EE1C25}
\definecolor{Sorange100}{HTML}{F36F23}
\definecolor{Syellow100}{HTML}{FFDD00}
\definecolor{Spetrol}{HTML}{00AAAD}
\definecolor{Sgreen100}{HTML}{8DC63F}
\definecolor{Spink100}{HTML}{EC008D}
\definecolor{Spurple100}{HTML}{812A91}
\definecolor{Syellow}{cmyk}{0,0.1,1,0}
\definecolor{Sorange}{cmyk}{0,0.7,1,0}
\definecolor{Sred}{cmyk}{0,1,1,0}
\definecolor{Spink}{cmyk}{0,1,0,0}
\definecolor{Spurple}{cmyk}{0.6,1,0,0}
\definecolor{Scyan}{cmyk}{1,0,0.4,0}
\definecolor{Sgreen}{cmyk}{0.5,0,1,0}
\definecolor{Sgreen}{cmyk}{0.5,0,1,0}

\definecolor{uniSgreen}{HTML}{93FF00}
\definecolor{uniSred}{HTML}{FF000B}
\definecolor{uniSpink}{HTML}{FF0098}
\definecolor{uniSorange}{HTML}{FF5D00}
\definecolor{uniScyan}{HTML}{00FBFF}

\colorlet{Sblackf}{Sblacke!70!white}
\colorlet{Sdgreen}{Sgreen!50!black}
\colorlet{Slred}{Sred!40!white}

\definecolor{uniSgray}{RGB}{62, 68, 76}
\colorlet{USredgray}{red!70!uniSgray}
\colorlet{uniSredgray}{red!70!uniSgray}
\colorlet{uniSgray90}{uniSgray!90!white}
\colorlet{uniSgray80}{uniSgray!80!white}
\colorlet{uniSgray70}{uniSgray!70!white}
\colorlet{uniSgray60}{uniSgray!60!white}
\colorlet{uniSgray50}{uniSgray!50!white}
\colorlet{uniSgray40}{uniSgray!40!white}
\colorlet{uniSgray30}{uniSgray!30!white}
\colorlet{uniSgray20}{uniSgray!20!white}
\colorlet{uniSgray10}{uniSgray!10!white}
\definecolor{Scyanlight}{rgb}{ 0.53333,0.87059,0.87451}
\definecolor{uniSblue}{HTML}{004191}
\colorlet{uniSblue80}{uniSblue!80!white}
\colorlet{uniSblue60}{uniSblue!60!white}
\colorlet{uniSblue40}{uniSblue!40!white}
\definecolor{uniSlightblue}{HTML}{00BEFF}
\colorlet{uniSlblue80}{uniSlightblue!80!white}
\colorlet{uniSlblue60}{uniSlightblue!60!white}
\colorlet{uniSlblue40}{uniSlightblue!40!white}
\definecolor{FFgreen}{rgb}{0.635,0.8275,0.1255}
\definecolor{FFdgreen}{rgb}{0.45,0.61,0.09}

\newcommand{\Stilde}{\raise.17ex\hbox{$\scriptstyle\sim$}}

\makeatother

\usepackage{dsfont}
\usepackage{natbib}
\newcommand{\cvp}{+}
\newcommand{\cvm}{-}
\newcommand{\cvpm}{\pm}
\usepackage[framed]{matlab-prettifier}

\newcommand{\LIST}[3]{#1,\;#2,\;#3}
\newcommand{\SINDEX}[3]{[\LIST{#1}{#2}{#3}]}
\newcommand{\RINDEX}[3]{(\LIST{#1}{#2}{#3})}

\jyear{2021}%

\theoremstyle{thmstyleone}%
\theoremstyle{thmstyletwo}%

\theoremstyle{thmstylethree}%

\begin{document}

\title[FFT-based Homogenization at Finite Strains using Composite Boxels]{FFT-based Homogenization at Finite Strains using Composite Boxels (ComBo)}

\author[1]{\fnm{Sanath} \sur{Keshav}}\email{keshav@mib.uni-stuttgart.de}

\author*[1]{\fnm{Felix} \sur{Fritzen}}\email{fritzen@simtech.uni-stuttgart.de}

\author[2]{\fnm{Matthias} \sur{Kabel}}\email{matthias.kabel@itwm.fraunhofer.de}

\affil[1]{\orgdiv{SC Simtech, Data Analytics in Engineering}, \orgname{University of Stuttgart}, \orgaddress{ Universitätsstr. 32, \postcode{70569} Stuttgart, \country{Germany}}}

\affil[2]{\orgdiv{Lightweight Design and Insulating Materials}, \orgname{Fraunhofer-Institut für Techno- und Wirtschaftsmathematik ITWM}, \orgaddress{\city{Kaiserslautern}, \postcode{67663}, \country{Germany}}}

\abstract{Computational homogenization is the gold standard for concurrent multi-scale simulations (e.g., FE2) in scale-bridging applications. Often the simulations are based on experimental and synthetic material microstructures represented by high-resolution 3D image data. The computational complexity of simulations operating on such voxel data is distinct. The inability of voxelized 3D geometries to capture smooth material interfaces accurately, along with the necessity for complexity reduction, has motivated a special local coarse-graining technique called composite voxels \cite{Kabel2015}. They condense multiple fine-scale voxels into a single voxel, whose constitutive model is derived by the laminate theory. Our contribution generalizes composite voxels towards composite boxels (ComBo) that are non-equiaxed, a feature that can pay off for materials with a preferred direction such as pseudo uni-directional fiber composites. A novel image-based normal detection algorithm is devised which (i) allows for boxels in the firsts place and (ii) reduces the error in the phase-averaged stresses by around 30\% against the orientation cf. \cite{Kabel2015} even for equi-axed voxels. Further, the use of ComBo for finite strain simulations is studied in detail. An efficient and robust implementation is proposed, featuring an essential selective back-projection algorithm preventing physically inadmissible states. Various examples show the efficiency of ComBo against the original proposal by \cite{Kabel2015} and the proposed algorithmic enhancements for nonlinear mechanical problems. The general usability is emphasized by examining various Fast Fourier Transform (FFT) based solvers including a detailed description of the new Doubly-Fine Material Grid (DFMG). All of the studied schemes benefit from the ComBo discretization.}




\maketitle




\section{Introduction}
\label{intro}
\subsection{Homogenization in engineering}
In the last decade, the quality of micro x-ray computed tomography (CT) images has steadily improved. Nowadays, standard CT devices have a spatial resolution below one $\upmu$m. They produce 3D images of up to 4096\textsuperscript{3} voxels. This permits a detailed view of the microstructure's geometry of composite materials all the way down to the limits of continuum mechanical theories and the necessity for discrete particle methods. The geometrical information itself is sufficient to detect defects in components by measuring, e.g., the pore space and the shape of the pores or by detecting the presence of micro-cracks. In order to understand the effect of such microscopic features, one has to solve PDEs on the high-resolution 3D image data. These simulations assist in the characterization of the effective mechanical behavior.

Due to the sheer size of today's CT images, the resulting computational homogenization algorithms face severe challenges related to the high computational demands. For instance, to compute effective linear elastic properties on a 4096\textsuperscript{3} CT image, the number of nodal displacement degrees of freedom amounts to approximately 206~Billion. The solution of problems of this size using conventional simulation methods such as the finite element (FE) method requires huge compute clusters \cite{Arbenz2008,Arbenz2014}.
These difficulties are commonly overcome by working with conventional FEM on a variety of smaller subsamples of moderate size \cite{Kanit:2003}. The resulting effective properties of the individual subsamples are averaged in post-processing, cf., for instance, \cite{Andrae2013a,Andrae2013b}. This method, however, does not exploit the available information gathered from the specimen.

\subsection{FFT-based solvers}
In contrast to the conventional FEM approach, the numerical homogenization method of Moulinec-Suquet \cite{MoulinecSuquet1994, MoulinecSuquet1998} operates on the voxels of a CT image directly: the set of unknowns is formed by the strains, i.e., one tensor for each voxel in the CT image. The solution algorithm works in place (i.e., matrix-free) such that, in practice, the size of the CT images to be treated is only restricted by the size of the memory and the affordable compute time. FFT-based schemes can handle arbitrary heterogeneity, phase contrast \citep[e.g.,][]{MichelMoulinecSuquet2001} and degree of compressibility of the materials. Generally, the number of iterations is independent of the grid size, depending only on the material's contrast, i.e., the maximum of the quotient of the largest and the smallest eigenvalue of the elastic tensor field, and on the geometric complexity. The solution of the linear algebraic system relies upon a Lippmann-Schwinger fixed point formulation, enabling, by use of fast Fourier transform (FFT), a fast matrix-free implementation. By changing the discretization to finite elements \cite{Willot2015} and finite differences \cite{Schneider2016} also infinite material contrast problems became solvable. Recent displacement-based implementations \cite{Leuschner2018} also allowed the use of higher order integration schemes without increasing the memory demands and with improved computational efficiency over strain-based algorithms also building on finite element technology \cite{Schneider2017,Leuschner2018}. A relation with Galerkin based methods was outlined by \cite{Vondrejc:2014} where the use of the efficient Conjugate Gradient (CG) method was mentioned. Interestingly, the CG method (as well as minres and other krylov solvers) is completely natural in FANS and FFT-$Q_1$ Hex \cite{Leuschner2018,Schneider2017}, including straightforward implementation.

\subsection{Composite voxel technique}
Despite the many advances regarding theory and algorithmic realization of FFT-based methods, the resolution of state-of-the-art computed tomography poses challenges for the numerics. A single, double-precision scalar field on a 4096\textsuperscript{3} voxel image as delivered by modern $\upmu$CT scanners already occupies 512~GB of memory. Thus, even performing linear elastic computations on such images requires either exceptionally equipped workstations or big compute clusters. This applies evermore so to inelastic computations, where additional history variables need to be stored, increasing the memory demand significantly.
	
To enable computations on conventional desktop computers or workstations that can still take into account relevant microstructural details of the fully resolved image, the composite voxel technique was developed for linear elastic, hyperelastic and inelastic problems \cite{LionelGelebart2015,Kabel2015,Kabel2016,Kabel2017}. A coarse-graining procedure serves as the initial idea, i.e., a number of smaller, typically $4^3=64$, voxels are merged into bigger voxels. Each of these so-called \textit{composite voxels} gets assigned an appropriately chosen material law which is based on a bi-phasic laminate. This laminate takes into account the exact phase volume fractions and the interface normal vector~$N$. Thereby, it can reflect the most relevant microstructural features, avoid staircase phenomena intrinsic to regular voxel discretizations and allow for a boost in accuracy in computational homogenization.

%

In many scenarios the edge lengths of the voxels are different. This can be due to the employed imaging technique, e.g., in serial sectioning by a focused ion beam and using scanning electron microscopy (FIB-SEM, \cite{Uchic:2006}). Another reason for the wish of using anisotropic voxel shapes is the presence of a preferred direction, e.g., in composites with preferred orientation such as long fiber reinforced materials studied by \cite{Fliegener:2014}. The consideration of anisotropic composite voxels -- called {\it boxels} in the present study -- leads to low quality normal orientations when using the original approach suggested by \cite{Kabel2015}. Likewise, low volume fractions can deteriorate the quality of the normal orientation.

Another issue arises depending on the material contrast and the local volume fraction in the composite voxels: If the soft phase of the two-phase laminate has a small volume fraction, the deformation within this phase is severely exaggerated. This typically leads to convergence issues of the Newton-Raphson method which evaluates the effective response of the composite voxel and is normally circumvented by limiting the volume fractions of either of constituents to be in the range of 5$\cdots$95\% and by resorting to simple  Voigt averaging otherwise, ruling out the aforementioned advantages of the composite voxels.

\subsection{Outline}
In the current study we target composite voxels for finite strain homogenization problems. In Section~\ref{sec:hom:problem} the homogenization problem is recalled. The foundations of FFT-based schemes are summarized in Section~\ref{sec:fft} including the Lippmann-Schwinger equation and an extension of the staggered grid approach of \cite{Schneider2016} towards improved local fields. The composite voxel technique is considered in detail in Section~\ref{sec:combo} including algorithmic improvements leading to increased robustness for finite strain problems.


\subsection{Notation}

The spatial average of a quantity over a domain $\cA$ with measure $A = \lvert \cA \rvert$ is defined as 
\begin{equation}
	\langle\cdot\rangle_{\cA}=\dfrac{1}{A} \int_{\cA}(\cdot) \; \mathrm{d} A \, .
\end{equation}
In the sequel bold face lowercase letters denote vectors (exceptions: material coordinate $\fX$, normal vector $\fN$, displacement $\fU$ and traction $\fT$), bold face upper case letters denote 2-tensors and blackboard bold uppercase letters (e.g., $\ffC$) denote 4-tensors. The inner product contracts all indices of two $k$-tensors while the usual linear mapping contracts the last $k$ indices of the first tensor with the following $k$-tensor. The double contraction operator $\ffA:\fB$ contracts the trailing two indices of $\ffA$ with the leading two indices of $\fB$. The outer tensor product is denoted by $\otimes$ and its symmetric version is given by $\otimes^{\rm s}$.

General vectors and matrices are denoted by single and double underlines, respectively. Note that in the sequel we focus on an orthonormal basis for the tensor fields. A vector representation for general ($\fF \leftrightarrow \ul{F}$) and symmetric 2-tensors ($\fS \leftrightarrow \ul{S}$) is given by
\begin{align}
    \ul{F} &= \sV F_{11} \\ F_{12}\\ F_{13}\\ F_{21} \\ \vdots \\ F_{33} \eV \in \ffR^9, &
    \ul{S} & = \sV S_{11} \\ S_{22} \\ S_{33} \\ \sqrt{2}S_{12} \\ \sqrt{2}S_{13} \\ \sqrt{2}S_{23} \eV \in \ffR^6. \label{eq:vector:not}
\end{align}
Note the different ordering and normalization factors of the Mandel-type notation versus the commonly used Voigt notation which is avoided due to the ambiguity of stress vs. strain like quantities in the vector representation. The notation \eqref{eq:vector:not} preserves the inner product for general 2-tensors (e.g., $\fF, \fP$) and symmetric 2-tensors (e.g., $\fS, \fE$), i.e.,
\begin{align}
    \fF \cdot \fP &= \ul{F}^{\sf T} \ul{P}, &
    \fS \cdot \fE &= \ul{S}^{\sf T} \ul{E} .
\end{align}
Likewise, 4-tensors are represented by $9\times 9$ and $6 \times 6$ (in the case of left and right sub-symmetry) matrices, respectively.

In view of spatial derivatives, the gradient ($\Grad{\cdot}$) and the divergence ($\Div{\cdot}$) with respect to the reference configuration are used. In the context of an internal surface $\scrS$ the jump operator is given as
\begin{align}
    \lbb u \rbb &= \underset{\epsilon \to 0_+}{\rm lim} \, u( \fx + \epsilon \fN ) - u(\fx - \epsilon \fN ) \qquad  (\fx \in \scrS),
\end{align}
where $\fN$ is the normal of $\scrS$ pointing from the second phase (referred to as inclusion phase) into the first phase (referred to as the matrix material).




\section{Homogenization Problem}
\label{sec:hom:problem}
We focus our attention to finite strain mechanics of hyperelastic solids.
Consider a material point $\fX$ of a body $\mathcal{B}_0$ in the reference configuration at time $t = 0$. The corresponding current position at time $t \in [0, T]$ is denoted by $\fx$ in the deformed domain $\cB(t)$. The motion of the body is given by 
\begin{align}
	\fvarphi(\fX,t) &: \, \mathcal{B} \times \ffR \to \ffR^d , &
	\fx &= \fvarphi(\fX,t),
\end{align}
where $d \in \{ 2, 3\}$ is the spatial dimension. In the sequel, the dependence on time $t$ is implicitly assumed but not reflected in the arguments for brevity. The displacement field is given by $\fU(\fX) = \fx - \fX$. The deformation gradient $\fF$ are defined as 
\begin{align}
	\fF &= \Grad{\fvarphi(\fX)} = \fI + \fH \, .
\end{align}
The volume change at a material point is given by the material Jacobian of the deformation map,
\begin{align}
	\td{v}{V} &= J = \operatorname{det}\fF > 0,
\end{align}
with ${\rm d}v$ the differential volume in the current configuration and ${\rm d}V$ denoting its counterpart in the reference configuration.
The right Cauchy-Green tensor is given by $\fC = \fF^\mathsf{T} \fF$. At a particular material point $\fX$ with an infinitesimal area $\mathrm{d}A$ with a unit normal vector $\fN$, we define the resulting traction vector~$\fT$ in terms of the first Piola-Kirchoff (PK1) stress tensor $\fP$,
\begin{equation}
	\fT = \fP \fN.
\end{equation}

In homogenization, the constitutive response on the structural (or macroscopic) scale of a body depends strongly on the underlying microstructure. In the following we consider a statistically homogeneous periodic microstructure that is described by a periodic representative volume element (RVE)
\begin{align}
    \Omega =  \left[-\frac{l_1}{2}, \frac{l_1}{2} \right] \times  \left[-\frac{l_2}{2}, \frac{l_2}{2} \right] \times  \left[-\frac{l_3}{2}, \frac{l_3}{2} \right] \, .
\end{align}
Separation of length scales is assumed, i.e., established homogenization principles apply without the need for advanced modeling techniques, e.g., based on filtering \cite{Yvonnet2014} or higher-order continuum theorieds \cite{Jaenicke2009}.


The objective of homogenization is the identification of the effective, macroscopic constitutive response
\begin{equation}
    \ol{\fP} = \left\la\fP\right\ra_\Omega \label{eq:Pbar}
\end{equation}
of micro-heterogeneous materials for given $\ol{\fF}=\la\fF\ra_\Omega$. Further, phase-averaged stresses, stress statistics and the algorithmic tangent operator can be of relevance. 
The effective quantities must satisfy the Hill-Mandel macro-homogeneity condition
\begin{align}
	\ol{\fP} \cdot \ol{\fF} &= \left\la \fP \right\ra_{\Omega} \cdot \left\la \fF \right\ra_{\Omega} = \left\la \fP(\fX) \cdot \fF(\fX)\right\ra_{\Omega} \, .
\end{align}
Periodic fluctuation boundary conditions of the form
\begin{align}
    \fU(\fX) &= (\ol{\fF}-\fI)\fX + \WT{\fU}(\fX) 
\end{align}
satisfy the Hill-Mandel condition~(see, e.g., \cite{Suquet1985}) and have proven versatile and efficient.
The tractions $\fT(\fX^\pm)$ are then anti-periodic for point-pairs $\fX^\pm \in \partial\varOmega^\pm$ on opposing faces of the RVE~$\Omega$.
The related function space for $\WT{\fU}$ is referred to as $\cV_\#$ which
is a subset of the Sobolev space of weakly differentiable functions~$H^1(\varOmega)$:
\begin{align}
    \cV_\# & = \left\lbrace \WT{\fU} \in H^1(\varOmega): \ \WT{\fU}(\fX^+) = \WT{\fU}(\fX^-)  \right\rbrace \, .
\end{align}
The homogenization problem \newcommand{\HOM}{{\bf (HOM)}}\HOM{} on $\Omega$ for prescribed deformation $\ol{\fF}$ reads:
\begin{align} 
    \text{find } \WT{\fU} & \in \cV_\# \label{eq:hom:start}\\
    \text{s. th.: }\Div{\fP} &= \fzero & & \text{in } \Omega \, .\label{eq:equilibrium} 
\end{align}



\section{FFT-based homogenization}
\label{sec:fft}


The solution of \HOM{} \eqref{eq:hom:start}-\eqref{eq:equilibrium} can either be obtained (in seldom cases) using analytical solution or by using discrete numerical techniques. The most significant methods for computational homogenization in solid mechanics are certainly the finite element method (FEM) \citep[e.g., FE\textsuperscript{2},][]{Feyel1999} and FFT-based homogenization.

The latter was proposed by Moulinec and Suquet \cite{MoulinecSuquet1994} in the early 1990s. It is based on the Lippmann-Schwinger equation in elasticity \cite{Kroener1977, ZellerDederichs1973} and avoids both, time consuming meshing needed by conforming finite element discretizations as well as the assembly of the related linear system. Therefore, the memory needed for solving the problem is significantly reduced compared with other methods. By virtue of the seminal Fast Fourier Transform algorithm \citep[FFT,][]{Cooley1965}, the compute time scales with $\cO(n \, \log (n) )$ where $n$ is the number of unknowns,  i.e., it is just slightly superlinear.

\newcommand{\Gzero}{{\fG^0}}
For finite strain problems the Lippmann-Schwinger equation reads
\begin{equation}
	\label{eq:nonlinearLSE}
	\fF = \ol{\fF} -\fGamma^0 (\fP - \ffC^0:\fF),
\end{equation}
with the Greens' operator
\begin{equation}
	\label{eq:Gamma0}
	\fGamma^0 ( \cdot ) = \Grad{\Gzero \Div{\cdot} }\, ,
\end{equation}
which is relying on the solution operator $\Gzero$ of a linear reference problem.
Lahellec, Moulinec, and Suquet \cite{Lahellec2001} proposed to solve \eqref{eq:nonlinearLSE} by the Newton-Raphson method and the linear Moulinec-Suquet fixed point solver. In contrast, Eisenlohr et al. \cite{Eisenlohr2013} suggested using the Moulinec-Suquet fixed point iteration on the nonlinear Lippmann-Schwinger equation for finite strains directly. Kabel et al. \cite{Kabel2014} carried over the idea of Vinogradov and Milton \cite{VinogradovMilton2008} and of G\'{e}l\'{e}bart and Mondon-Cancel \cite{GelebartMondonCancel2013} to combine the Newton-Raphson procedure with fast linear solvers to the geometrically nonlinear case.

The present work will also make use of the nonlinear conjugate gradient (CG) method introduced by Schneider \cite{Schneider2020} for small strains. Further, the FFT-accelerated solution of regular tri-linear hexahedral elements \citep[FFT-$Q_1$ Hex cf.,][]{Schneider2017} and in the closely related Fourier-Accelerated Nodal Solver \cite{Leuschner2018} will be used in the sequel, for which the nonlinear CG method has a perfectly natural interpretation as a preconditioner. 
An overview on alternative solution methods can be found in \cite{Schneider2021}.

Regarding the discretization we will compare the original discretization by Fourier polynomials of Moulinec-Suquet \cite{MichelMoulinecSuqeut1999} with the HEX8R discretization of Willot \cite{Willot2015} and the fully integrated HEX8 elements \cite{Schneider2017,fans2018}. In order to consistently evaluate material laws for the staggered grid discretization \cite{Schneider2016, Ospald2015}, we apply the idea of the double fine material grid (DFMG) to large strains. In contrast to the regime for small strains, which can be efficiently implemented for isotropically linear elastic materials by a suitable averaging of the shear moduli, the DFMG for large deformations requires multiple evaluations of the material routines, regardless of the complexity of the geometrically nonlinear material law. The technical details of the implementation can be found in the appendix \ref{app:staggeredGrid}.

\section{
Composites Voxels}
\label{sec:combo}
One of the main advantages but - at the same time - a major disadvantage of FFT-based techniques is the constraint to regular Cartesian grids: While they enable direct use of 3D image data, they are unable to capture the material interface accurately compared to interface-conforming discretizations. This is primarily due to the binarized nature of the microstructure, which leads to the so-called \emph{staircase approximation} of the interface. In order to capture the microscale effects sufficiently, a high grid resolution is hence needed, which--in turn--calls for high computational cost. In order to limit or even eliminate the staircase phenomenon without the need for a (distinct) grid refinement,  so-called composite voxels were previously suggested in \cite{Kabel2015, Merkert2015, Kabel2016, Kabel2017}. They enhance the existing binary discretization with special  voxels with effective material properties that depend on the phase volume fractions and the normal orientation~$\fN$ of the laminate.\\ 

Consider such a composite voxel $\Omega^e$ comprised of two phases denoted by $\Omega^e_\cvpm \subsetneq \Omega^e$, where $\cvp$ and $\cvm$ stand for inclusion and matrix phase, respectively. The fields corresponding to the two phases are represented as $(\cdot)_\cvpm$ for brevity, and either phase is assumed to be equipped with a hyperelastic strain energy density~$W_\cvpm$. Within $\Omega^e$ the material interface is approximated by a plane $\scrS^e$ leading to a rank-1 laminate defined by the interface normal vector $\fN\in \ffR^d$ in material configuration. The individual phase volume fractions are given by $c_\cvp$ and $c_\cvm$ where $c_\cvp + c_\cvm= 1$.\\

In previous works, different rules of mixture have been suggested, namely the Voigt (\textbullet\textsuperscript{V}) and Reuss (\textbullet\textsuperscript{R})\footnote{In the nonlinear kinematic regime, these correspond to the Taylor and Sachs approximation, respectively.} estimates correspond to upper and lower bounds of the mechanical response.

\newcommand{\Clam}{\ffC_{\square}^{\Vert}}
Inspired by laminate theory, \cite{Milton2002} suggests the definition of the effective elasticity tensor of the laminate $\Clam$ implicitly through
\begin{align} \label{mixingrule}
    \lb \ffP + \lambda \lb \Clam - \lambda \ffI \rb^{-1}\rb^{-1} = \nonumber \\
     \left\langle \lb \ffP + \lambda \lb \ffC - \lambda \ffI\rb^{-1}\rb^{-1}\right \rangle_{\Omega^e}\, .
\end{align}
Here, the fourth order tensor $\ffP$ depends on the normal~$\fN$ via
\begin{align}
    \mathbb{P}_{i j k l} &=  \frac{1}{2}\Big(N_{i} \delta_{j k} N_{l}+N_{i} \delta_{j l} N_{k} +N_{j} \delta_{i k} N_{l} \nonumber\\
    & \quad +N_{j} \delta_{i l} N_{k}\Big) 
    -N_{i} N_{j} N_{k} N_{l} 
\end{align}
for $i,j,k,l \in \{1, \dots, d\}$. Note that in laminate theory, the states in the $\pm$ phase are piece-wise constant. Hence, the averaging translates into
\begin{align}
    \la (\cdot) \ra_{\Omega^e} &= c_+ \, (\cdot)_+ \, + \,c_- \, (\cdot)_- .
\end{align}
The laminate mixing rule considers the interface orientation and introduces a parameter $\lambda > 0$ 
that must be larger than the leading eigenvalue of the individual stiffness tensors $\ffC_\cvpm$. Solving \eqref{mixingrule} for linear elastic materials involves $6$ inversions of $6\times 6$ matrices\footnote{Symmetry is exploited.} for each composite voxel individually, i.e., at every interface-related voxel or cubature point within an FFT-based simulation, which can lead to non-negligible overhead. Hence, a more numerically efficient and physics-motivated approach devoid of additional parameters is sketched in the following that can also handle geometric and material nonlinearities, if needed.


\subsection{Hadamard jump conditions for finite strain kinematics}
Following established laminate theory, the following assumptions hold:
\begin{itemize}
    \item The deformation gradient in $\Omega^e_+$ and $\Omega^e_-$ are related by a rank-1 jump along the normal orientation, i.e., for $\WT{\fa} \in\ffR^d$,
\begin{align}
	\llbracket\fF\rrbracket_{\scrS^e} = \fF_\cvp - \fF_\cvm = \WT{\fa} \otimes \fN \, .
	\label{eq:hadamard:F}
\end{align}
     \item The traction vector is continuous across the interface $\scrS^e$, i.e.,
\begin{align}
    	\llbracket\fT\rrbracket_{\scrS^e}= \llbracket\fP\rrbracket_{\scrS^e}  \fN &= \fzero,  &
    \fT_\cvp &= \fT_\cvm
    	\, .
	\label{eq:hadamard:T}
\end{align}
\end{itemize}
%
%
The kinematic compatibility on the interface is characterized by having a continuous deformation gradient tangential to the interface. This ensures that material point pairs remain identical. Condition \eqref{eq:hadamard:T} ensures that the interface is in static equilibrium.

In order to enforce \eqref{eq:hadamard:F}, we chose the following parameterization\footnote{(slightly different to that used, e.g., in \cite{Kabel2017})} of the deformation gradient tensors of the two material phases
\begin{align}
\label{eq:Fpm}
	\fF_\cvpm &= \fF_{\square} \pm \frac{1}{c_\cvpm} (\fa \otimes \fN),
\end{align}
where $\fa$ is related to $\WT{\fa}$ in \eqref{eq:hadamard:F} by a scaling constant and $\fF_{\square}$ is the prescribed deformation gradient on the composite voxel~$\Omega^e$ . The parameterization \eqref{eq:Fpm} preserves the volume average~$\fF_{\square}$ of the deformation gradient $\fF$ over the composite voxel $\Omega^e$ according to
\begin{align}
	\left\la \fF \right\ra_{\Omega^e} 
	= \fF_{\square} + \lb \frac{c_\cvp}{c_\cvp} - \frac{c_\cvm}{c_\cvm}\rb \fa \otimes \fN 
	 = \fF_{\square} \, .
\end{align}
In the finite strain context, not only the average of $\fF$ must be preserved, but also the total volume of the deformed material must remain constant under the rank-one perturbation:

\begin{lemma}
	The kinematic  rank-1 perturbation on the material interface $\scrS^e$ is volume preserving by construction.
\end{lemma}
\begin{proof}
	Suppose $\fA$ is an invertible $n\times n$ matrix and $\fu$, $\fv$ are column vectors of size $n$. Then the matrix determinant lemma states that
	\begin{equation} \label{matdetlemma}
		\det{\fA + \fu \fv^\mathsf{T}} = (1 + \fv^\mathsf{T} \fA^{-1} \fu) \; \det{\fA}.
	\end{equation}
	The deformation gradient~$\fF_{\square}$ is regular by definition since material inversion is not allowed. Thus, by setting $\fA \gets \fF_{\square}$ and $\fu\gets \pm \fa/c_\cvpm, \, \fv \gets \fN$ the material Jacobian $J_\cvpm$ in the two phases can be computed using the matrix determinant lemma
	\begin{align}
		J_\cvpm &=  \det{\fF_\cvpm}
		= \left( 1 \pm \frac{1}{c_\cvpm} \fa^{\mathsf{T}} \fF_{\square}^{-\mathsf{T}}\fN \right) \det{\fF_{\square}} \, . 
		\label{eq:Jpm}
	\end{align}
	Averaging $J$ over the composite voxel $\Omega^e$ we obtain
	\begin{align}
		\la J\ra_{\Omega^e} &= c_\cvp J_\cvp + c_\cvm J_\cvm  = \det{\fF_{\square}} = J_{\square}\, .
	\end{align}
\end{proof}

To obtain the gradient jump vector $\fa$ the equilibrium of the tractions on the interface cf. \eqref{eq:hadamard:T} must be granted, see also \cite{Kabel2017}:
\begin{align} \label{tractionbalance}
	\ff(\fa) & = \fT_\cvp - \fT_\cvm \overset{!}{=} \mathbf{0}.
\end{align}
Examining the relative volume in $\Omega^e_\cvpm$, additional constraints on $J_\cvpm$ emerge:
\begin{align}
    \label{eq:Jpm:volumeConstraint}
    J_\cvp & \overset{!}{>} 0 \, , & J_\cvm & \overset{!}{>} 0 \, .
\end{align}
These constraints imposed on \eqref{tractionbalance} are essential to gain robustness, see Section~\ref{subsec:backproj}.

\subsubsection{Algorithmic implementation}
In the case of large strain kinematics, the system \eqref{tractionbalance} is always nonlinear. Hence, it must be solved iteratively, e.g., by using a Newton-Raphson scheme with the Hessian
\begin{align}
	\fDelta_f = \td{\ff(\fa)}{\fa} &= \lb\pd{ \fP_\cvp}{ \fF_\cvp} \frac{\partial \fF_\cvp}{\partial \fa} - \frac{\partial \fP_\cvm}{\partial \fF_\cvm} \frac{\partial \fF_\cvm}{\partial \fa}\rb\fN.  \label{eq:Hessian:F:0}
\end{align}
The Hessian in index notation can be written in terms of the 2-tensors $\fP_\cvpm$ and $\fF_\cvpm$ as
\begin{align}
	\lb \Delta_f \rb_{ik}
	&= N_J \left( \frac{1}{c_\cvp}\mathbb{A}^{iJkL}_\cvp + \frac{1}{c_\cvm}\mathbb{A}^{iJkL}_\cvm\right) N_L,  
	 \label{eq:hessian}
\end{align}
where, $\ffA_\cvpm = \pd{\fP_\cvpm}{\fF_\cvpm} =  \dfrac{\partial^2 W_\cvpm}{\partial\fF_\cvpm \otimes \partial \fF_\cvpm}$.
The constitutive tangent operator $\ffA_\cvpm$ can directly be obtained from the hyperelastic potentials~$W_\cvpm$.
Using a vector notation for $\fF, \fP$ and the related matrix representation $\ull{A}_\cvpm$ of $\ffA_\cvpm$, the Hessian gets
\begin{align}
    \ull{\Delta}_f &= \ull{D}^{\sf T} \lb \frac{\ull{A}_\cvp}{c_\cvp} + \frac{\ull{A}_\cvm}{c_\cvm} \rb \ull{D} \label{eq:Hessian:F}
\end{align}
with the matrix $\ull{D}$ defined via
\begin{align}
	(\fa \otimes \fN) \rightarrow \ull{D} \; \underline{a} = {\small
	\begin{bmatrix}
		N_1 & 0 & 0\\
		0 & N_2 & 0 \\
		0 & 0 & N_3 \\
		 N_1 & 0 & 0\\
		0 & N_2 & 0 \\
		0 & 0 & N_3 \\
		N_1 & 0 & 0\\
		0 & N_2 & 0 \\
		0 & 0 & N_3 
	\end{bmatrix}
	\begin{bmatrix}
		a_1 \\ a_2 \\ a_3 
	\end{bmatrix}}.
\end{align}
Note the straightforward symmetric structure of 
the matrix $\ull{\Delta}_f$ cf. \eqref{eq:Hessian:F}. Within  iteration $[k]$ the Newton-Raphson update reads 
\begin{align}
	\ul{a}^{[k]} &= \ul{a}^{[k-1]} - \lb \ull{\Delta}_f^{[k-1]}\rb^{-1} \ul{f}(\ul{a}^{[k-1]}) \, .
\end{align}
The full algorithm for the naive Newton-Raphson (NR) update is given in Algorithm~\ref{alg:combo_NR}. Note that each iteration requires a single  $d\times d$ matrix inversion. Once convergence is achieved, i.e., once the tractions on the interface are in balance, the first Piola-Kirchoff stress is
\begin{align}
    \fP_\square &= c_\cvp \fP_\cvp + c_\cvm \fP_\cvm 
    \, .
    \label{eq:Peff:Seff}
\end{align}
The unconditionally symmetric effective stiffness matrix can be computed from
\begin{align}
     \ull{\delta A} &= \ull{A}_\cvp - \ull{A}_\cvm \\
    \pd{\ul{P}_\square}{\ul{F}_\square} = \ull{A}_{\square} & =  \lb c_\cvp \ull{A}_\cvp + c_\cvm \ull{A}_\cvp \rb \nonumber \\
	& \quad -  \ull{\delta A} \, \ull{D} \; 	\ull{\Delta}_f^{-1}  \ull{D}^{\mathsf{T}} \ull{\delta A} \, .
\end{align}

\begin{algorithm}
	\caption{Newton-Raphson (NR) algorithm solving for gradient jumps in composite boxels}
	\begin{algorithmic}
		\Require $ \fF_\square$
		\Ensure $\fT_\cvp = \fT_\cvm$
		\State $\fa \gets \fzero$ (or any admissible choice)\Comment{initialization}
        \State $\fF_\cvpm \gets \fF_\square \pm \dfrac{1}{c_\cvpm} \fa \otimes \fN$
        \State $\fP_\cvpm, \; \ffA_\cvpm \gets$\ {material model}$\;(\fF_\cvpm)$
        \State $\ff \gets (\fP_\cvp - \fP_\cvm)\fN$ \Comment{compute initial residual}        \State $\cE_\square \gets \vert \ff \vert$ \Comment{initial error}
        \While{$\mathcal{E}_\square > \epsilon $} \Comment{check convergence}
        \State $\fa \gets \fa -\fDelta_f^{-1} \ff$ \Comment{naive NR-update}
        \State $\fF_\cvpm \gets \fF_\square \pm \dfrac{1}{c_\cvpm} \fa \otimes \fN$ \Comment{update $\fF_\cvpm$}
        \State $\fP_\cvpm, \; \ffA_\cvpm\gets$\ {material model}$\;(\fF_\cvpm)$
        \State $\ff \gets (\fP_\cvp - \fP_\cvm)\fN$ \Comment{update residual}
        \State $\cE_\square \gets \vert \ff \vert$  \Comment{update error}
        \EndWhile
	\end{algorithmic}
	\label{alg:combo_NR}
\end{algorithm}

\subsection{Robust algorithm via selective back-projection} \label{matdetlemmasection}\label{subsec:backproj}
Unfortunately, a naive update of the gradient jump vector $\fa\gets \fa - \fDelta_f^{-1} \ff$ sometimes leads to physically unacceptable iterates: the local volume within $\Omega^\cvpm$ can get negative which is inadmissible, i.e., contradicting the constraints \eqref{eq:Jpm:volumeConstraint}. 
%
%

Robustness can be enhanced by restricting the material Jacobian~$J_\cvpm$ in both phases to be strictly positive cf.~\eqref{eq:Jpm:volumeConstraint}. Previously, this problem was addressed in \cite{Kabel2016} by limiting the overall step width of the NR iterations, if the constraint \eqref{eq:Jpm:volumeConstraint} is not met.

We choose a different approach that can improve the convergence behavior building on the matrix determinant lemma \eqref{matdetlemma}. It allows to rewrite $J_\cvpm$ as
\begin{eqnarray}
	J_\cvpm  
	&=& \left( 1 \pm \frac{1}{c_\cvpm} \fa^{\mathsf{T}} \! \cdot \! \fF_{\square}^{-\mathsf{T}}\fN \right) J_{\square} > 0  .
\end{eqnarray}
We define
\begin{align}
\beta^\cvpm_{\rm c} &=  \mp \frac{c_\cvpm}{\|\fF_{\square}^{-\mathsf{T}}\fN\|}, &
\fm_\beta &=  \frac{\fF_{\square}^{-\mathsf{T}}\fN}{\|\fF_{\square}^{-\mathsf{T}}\fN\|}
\end{align}
and, with help of $\fm_\beta$ the orthogonal projectors
\begin{align}
   \fM_\beta^\parallel &= \fm_\beta \otimes \fm_\beta ,  &
    \fM_\beta^\perp &= \fI - \fM_\beta^\parallel \, .\label{eq:M:proj}
\end{align}
The matrix determinant lemma implies that, every NR iterate $\fa$ must satisfy 
\begin{equation} \label{matrixlemma}
	\beta^\cvp_{\rm c} < \fa^{\sf T} \fm_\beta < \beta^\cvm_{\rm c} \, ,
\end{equation}
i.e., a condition that is trivial to check. Most importantly, $\fm_\beta$ is independent of the current iteration, i.e., it can be precomputed.

If starting from a \textit{previously admissible point} $\fa_0=\fa^{[k-1]}$ an inadmissible iterate $\fa_1 \gets \fa_0 + \delta \fa$ is detected,  we compute $\beta_i = \fa_i \cdot \fm_\beta$ ($i\in\{0, 1\}$). Obviously $\beta_0$ is admissible while $\beta_1$ is not. A new selectively back-projected admissible coordinate $\beta_*$ is calculated to be the mid point between the admissible $\beta_0$ and the critical value $\beta_{\rm c} \in \{ \beta_\cvp, \beta_\cvm\}$ (the value which is exceeded by $\beta_1$ is taken, see Algorithm~\ref{alg:backprojection}):
\begin{align}
	\beta_* & = \frac{\beta_0 +\beta_{\rm c}}{2}.
\end{align}
Other than the approach suggested in \cite{Kabel2016} our selective back-projection yields a new (now admissible) iterate by adjusting only the part of $\fa$ contributing to $J_\cvpm$:
\begin{align}
	\fa_* &= \fa_0 + \beta_* \; \fm_\beta = \fa_1 + (\beta_* - \beta_1 ) \; \fm_\beta.
\end{align}
It leaves the part of $\fa$ not interacting with the volume constraint -- that is the part of $\fa_1$ that is orthogonal to $\fm_\beta$ -- unaltered. Only the part of $\fa_1$ that is co-linear with $\fm_\beta$ is modified such that a valid iterate is obtained. Using the projectors \eqref{eq:M:proj} this is equivalent to:
\begin{align}
    \fM_\beta^\perp \;\fa_* &= \fM_\beta^\perp \;\fa_1, &
    \fM_\beta^\parallel\; \fa_* &= \beta_* \;\fm_\beta \,. \label{eq:backproj:result}
\end{align}
This induces that a large part of the NR update is preserved. No additional hyperparameters, line searches and constitutive evaluations are required.

\begin{algorithm}
	\caption{Selective back-projection of inadmissible $\fa_1$ for previously admissible $\fa_0$}
	\begin{algorithmic}
		\Require inadmissible $\fa_1$; $\fa_0$; $\fm_\beta,\;\beta^\cvpm_{\rm c}$
		\Ensure $\beta^\cvp_{\rm c} < \fa_*^{\sf T} \fm_\beta < \beta^\cvm_{\rm c}$
		\State $\beta_1 \gets \fa_1^\mathsf{T}\fm_\beta$, $\beta_0 \gets \fa_0^\mathsf{T}\fm_\beta$\Comment{initialize}
		\If{$\beta_1 \leq \beta^\cvp_{\rm c}$} \Comment{get critical value $\beta_{\rm c}$}
		\State $\beta_{\rm c}=\beta^\cvp_{\rm c}$
		\ElsIf{$\beta_1 \geq \beta^\cvm_{\rm c}$}
		\State $\beta_{\rm c}=\beta^\cvm_{\rm c}$
		\EndIf
		\State $\beta_* \gets (\beta_0+\beta_{\rm c})/2$ \Comment{get admissible $\beta$ value}
		\State $\fa_* \gets \fM_\beta^\perp \fa_1 + \beta_* \; \fm_\beta$ \Comment{update $\fa$ accordingly}
	\end{algorithmic}
	\label{alg:backprojection}
\end{algorithm}

The increased robustness due to the selective back-projection is of utmost relevance for actual finite strain simulations, particularly in FE\textsuperscript{2} like settings. This holds {\it a fortiori} if composite voxels with a rather small phase volume fractions go along with pronounced phase contrast and high load increments, see Section~\ref{fig:bp_polyhedron1}.

As mentioned earlier, our algorithm shares similarities with the backtracking algorithm proposed in \cite{Kabel2016}. Both the approaches use the matrix determinant lemma to check for inadmissible iterates. The backtracking of \cite{Kabel2016} is employed as a line search strategy to compute an optimal step size in the admissible range, leading to scalar backtracking of $\fa$ and a damped NR method. This induces additional line search parameters as well as additional function evaluations. Contrary to that, our selective back-projection algorithm requires no extra constitutive evaluations and ensures an admissible update unconditionally in the absence of additional hyperparameters.

\section{Identification of the normal vector for composite voxels and boxels}
\label{sec:normal}
\subsection{Existing procedures and complications}
Composite voxels and boxels (non-equiaxed voxels) require the normal vector~$\fN$ alongside the phase volume fractions $c_\cvpm$. While $c_\cvp$ and $c_\cvm$ are easy to compute from averaging over the composite voxel or boxel, identifying the normal vector~$\fN$ is not without issues. It is also evident that the quality of the normal information is critical to the performance of the composite boxels. The original proposal of Kabel and colleagues~\cite{Kabel2015} uses the direction of the line connecting the barycenter of the dominant of the two phases within the composite voxel/boxel against the center of the voxel as an approximation of the normal vector, see Figure~\ref{fig:normal:kabel}. A trivial computation shows that this is equivalent to having the normal defined via the connecting line of the barycenters:
\begin{align}
    \WT{\fN} &= \frac{ \fc_- - \fc_+ }{\vert \fc_- - \fc_+ \vert} . \label{eq:normal:kabel}
\end{align}
In the present work, the normal is assumed to point out of the inclusion phase (indexed by $\cvp$).
\begin{figure*}
	\centering
	\includegraphics[width=\textwidth]{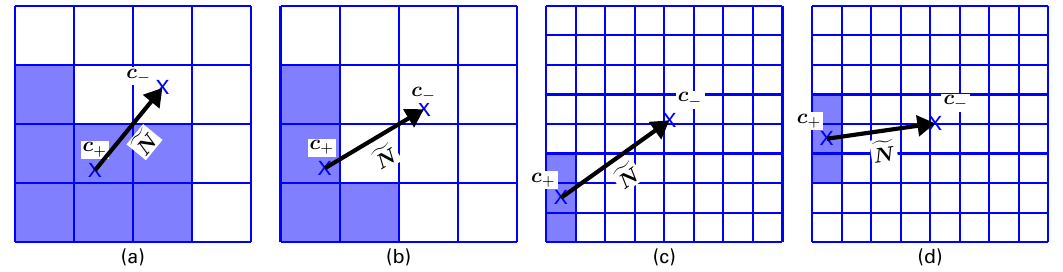}
	\caption{\protect Normal direction cf. \cite{Kabel2015}: $\WT{\fN}$ cf. \eqref{eq:normal:kabel} is the direction between the barycenters of the phases for different examples}
	\label{fig:normal:kabel}
\end{figure*}
The approach \eqref{eq:normal:kabel} has several advantages to it: The implementation is rather trivial, the computation is rapid, and it is easy to guarantee that the normal~$\fN$ points from phase~1 (denoted inclusion phase in the sequel) into phase~0 (denoted matrix phase in the sequel). However, the use of the barycenters $\fc_\pm$ for the normal detection is not without issues, as can be seen from the examples shown in Fig.~\ref{fig:normal:kabel}. At low volume fractions, the normal direction will depend only on the position of the phase rather than on the actual interface (compare (c) and (d) in Fig.~\ref{fig:normal:kabel}) and the method is not working for boxels that differ from equiaxed composite voxels by having (sometimes pronounced) aspect ratios, see Figure~\ref{fig:normal:kabel2}. This occurs if the number of voxels inside the composite boxel varies along the different edges (with fine scale voxels being equiaxed, e.g., Fig.~\ref{fig:normal:kabel2}, top). For instance, this can be useful in order to reduce the resolution in pseudo-unidirectional fiber reinforced materials; see Section~\ref{subsec:complex} for examples.


\begin{figure}
	\centering
	\includegraphics[width=\columnwidth]{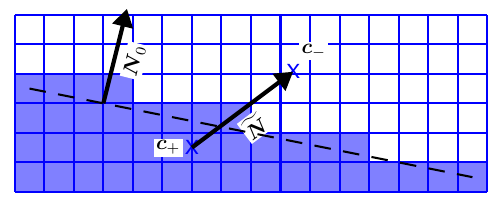}
    \caption{\protect Normal direction $\WT{\fN}$ from \eqref{eq:normal:kabel} cf. \cite{Kabel2015} for a 2D boxel (16$\times$6 fine-scale grid, equiaxed fine scale); the true normal $\fN_0$ is provided for comparison}
	\label{fig:normal:kabel2}
\end{figure}


In order to allow for an improved normal computation over \eqref{eq:normal:kabel}, the authors suggest a novel strategy that can be broken down into two steps:
\begin{enumerate}[label={\bf [N.\arabic*]},leftmargin=12mm]
	\item compute an indicator for the interface between the phases using a discrete Laplacian resulting into discrete weights on the fine grid;
	\item within composite boxels that contain a material interface, define the normal vector via a minimization problem.
\end{enumerate}
The individual steps are described below. A free python implementation is available via an open access software accessible via the Github (\cite{github:combo:normal}) including the option to easily process HDF5 files \cite{hdfgroup} (see also the documentation and tutorial given in the repository's jupyter notebook).

\subsection{Interface indication via a discrete Laplacian}
In the following, the Laplace filter commonly used in image processing is introduced. It can be employed for edge detection. For simplicity, the Laplace stencil is defined by building on a first order finite difference gradient operator along each coordinate axis. Therefore, the 3D stencil illustrated in Figure~\ref{fig:3d:stencil} is utilized.

\begin{figure}
\centering
\includegraphics[scale=1.2]{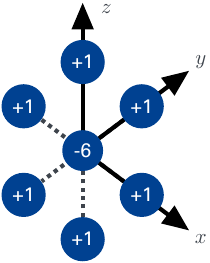}
\caption{3D stencil used for the Laplacian (here $h_1=h_2=h_3=1$ for simplicity)}
\label{fig:3d:stencil}
\end{figure}

In 1D for grid spacing $h>0$, the derivative at position $x_i$ and the second derivative (gathered from the subsequent application of the first derivative) read
\begin{align}
	\pd{f(x_{i-1})}{x} & \approx f'_{i-1} = \frac{f_{i} - f_{i-1}}{h}, \\
	\pd{f(x_{i})}{x} & \approx f'_{i} = \frac{f_{i+1} - f_i}{h}, \\
	\pd{f^2 (x_i)}{x^2} & \approx \frac{1}{h} ( f'_i - f'_{i-1} ) \nonumber\\
	&= \frac{1}{h^2} \lb f_{i-1} - 2 f_i + f_{i+1}\rb.
\end{align}
Setting $f_j=f(x_j)$ to value 1, if the inclusion phase is found at $x_j$ and to 0 otherwise, the Laplacian will be 0, if and only if all pixel values within the stencil are identical. Therefore, only point triples hosting more than one phase will lead to a nonzero Laplacian. Further, pixels with positive and negative values will be found, which represent inside and outside voxels related to the interface, respectively. For 2D and 3D boxels, the stencil is composed of uniaxial Laplace stencils along the coordinate axes, using the respective $h$ value corresponding to the fine-scale boxel dimension along the respective direction $(h_x, h_y, h_z)$, see also Fig.~\ref{fig:3d:stencil}. The Laplace stencil will be weighted by the boxel volume. Thereby, the scalar factor will get a neat physical interpretation:
\begin{align}
	r_i &= r(l_i) = \frac{l_j l_k}{l_i} \quad (i \neq j, \ i \neq k, \  j \neq k).
\end{align}
The ratio $r_i$ expresses the ratio of the area of the boxel face with normal direction $\fe_i$ divided by the boxel dimension along direction~$\fe_i$. After application of the stencil~$\cS$ to the image the absolute value is taken to gain the discrete weights
\begin{align}
	w_{ijk} &= \vert \cS \ast \chi \vert_{ijk} \qquad \nonumber \\ &\lb i \in \{ 1, n_1 \}, \ j \in \{1,n_2\}, \  k \in \{ 1,n_3\} \rb.
\end{align}
The convolution is carried out in the Fourier domain which automatically enforces periodic boundary conditions for the normal detection of inclusions crossing the edges of the computational domain. Of course, also a direct computation would be feasible at comparable compute cost. 
The weights are now processed for each composite boxel attributed with multiple material phases, i.e., with volume fraction $0 < c_\cvm < 1$. On each of these boxels, a weighted least squares problem is set up in order to identify the normal orientation. Therefore, on the boxel $\Omega_{\rm B}$ the second moment tensor 
\begin{align}
	\ull{M} = \sum_{ \{i,j,k\} \in \Omega_{\rm B}} w_{ijk} \fx_{ijk}\otimes \fx_{ijk} \in Sym_+(\ffR^3)
\end{align}
is computed, where $\fx_{ijk}$ denotes the coordinates of the voxel with discrete coordinates $(i,j,k)$\footnote{It is assumed that $\fx_{ijk}$ is relative to the barycenter of $\Omega_{\square}$.}.
Next, the eigenvector matching the smallest eigenvalue of $\ull{M}$ is used as initial normal $\fN$. In a second step, the direction of the normal $\fN$ is identified. Therefore, the vector connecting the barycenters of the material phase within the boxels $\WT{\fN}$ cf. \eqref{eq:normal:kabel} of the original approach \cite{Kabel2015} is considered to identify the proper sign of the proposed $\ul{N}$:
\begin{align}
	\fN & \gets  \left\lbrace \begin{array}{rl} \fN & \quad\text{if }\WT{\fN} \cdot \fN > 0 \\ -\fN & \quad\text{else.} \end{array} \right. \label{eq:normal:flip}
\end{align}
Thereby, the normal is guaranteed to point out of the $\cvp$ phase.

The example shown earlier in Fig.~\ref{fig:normal:kabel2} has been used with the new approach, see Fig.~\ref{fig:normal:new:1}. The interface voxels are shown in red with darker tones denoting increased weights. Note the good correlation of the analytically defined normal $\fN_0$ used to generate the discrete image and the normal reconstructed using the new approach. The colinearity of the two vectors is $\fN_0 \cdot \fN = 0.99998$ where the interface detection was effected using matching padding for the analytical normal. This compares against $\fN_0 \cdot \WT{\fN} = 0.7761$ for the normal shown in Fig.~\ref{fig:normal:kabel2} obtained from \eqref{eq:normal:kabel}.

\begin{figure}[!h]
	\centering
	\includegraphics[width=\columnwidth]{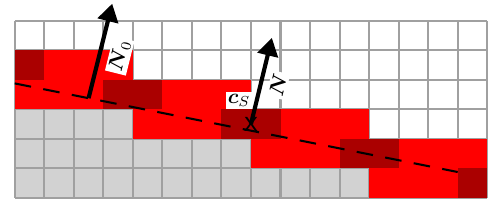}
	\caption{\protect Normal vector for the same example as in Fig.~\ref{fig:normal:kabel2} but using the new algorithm; interface voxels are red; $\fc_S$ is the barycenter of the interface}
	\label{fig:normal:new:1}
\end{figure}

Another comparison is shown in Fig.~\ref{fig:normal:new:3}. Here an input image consisting of 256\textsuperscript{3} voxels containing a centered sphere of radius $r=0.4\,L$ (with $L$ denoting the edge length of the cube) is coarsened using composite voxels of size 32\textsuperscript{3}, i.e., the information is compressed by a factor 32\textsuperscript{3}=32\,768 yielding a coarse scale resolution of just 8\textsuperscript{3}. For each composite voxel the normal is computed using either the approach from \eqref{eq:normal:kabel} \cite{Kabel2015} (Fig.~\ref{fig:normal:new:2:b}) and the new approach presented in this Section (Fig.~\ref{fig:normal:new:2:d}). In this visualization the actual facets are reconstructed from the normal and the volume fraction~$c_\cvp$. By the metric of vision, the normals of the old approach graphs are less regular. Obviously the facets reconstructed from these are also not leading to an accurate approximation of the curved surface of the sphere (e.g., gaps/overlaps). Our new procedure leads to visually more accurate normals. This is confirmed by the planar facet reconstruction that is almost entirely devoid of gaps and overlaps.


\begin{figure*}[!h]
    \begin{subfigure}[b]{0.515\columnwidth}
         \centering \captionsetup{justification=centering}
         \includegraphics[width=\textwidth]{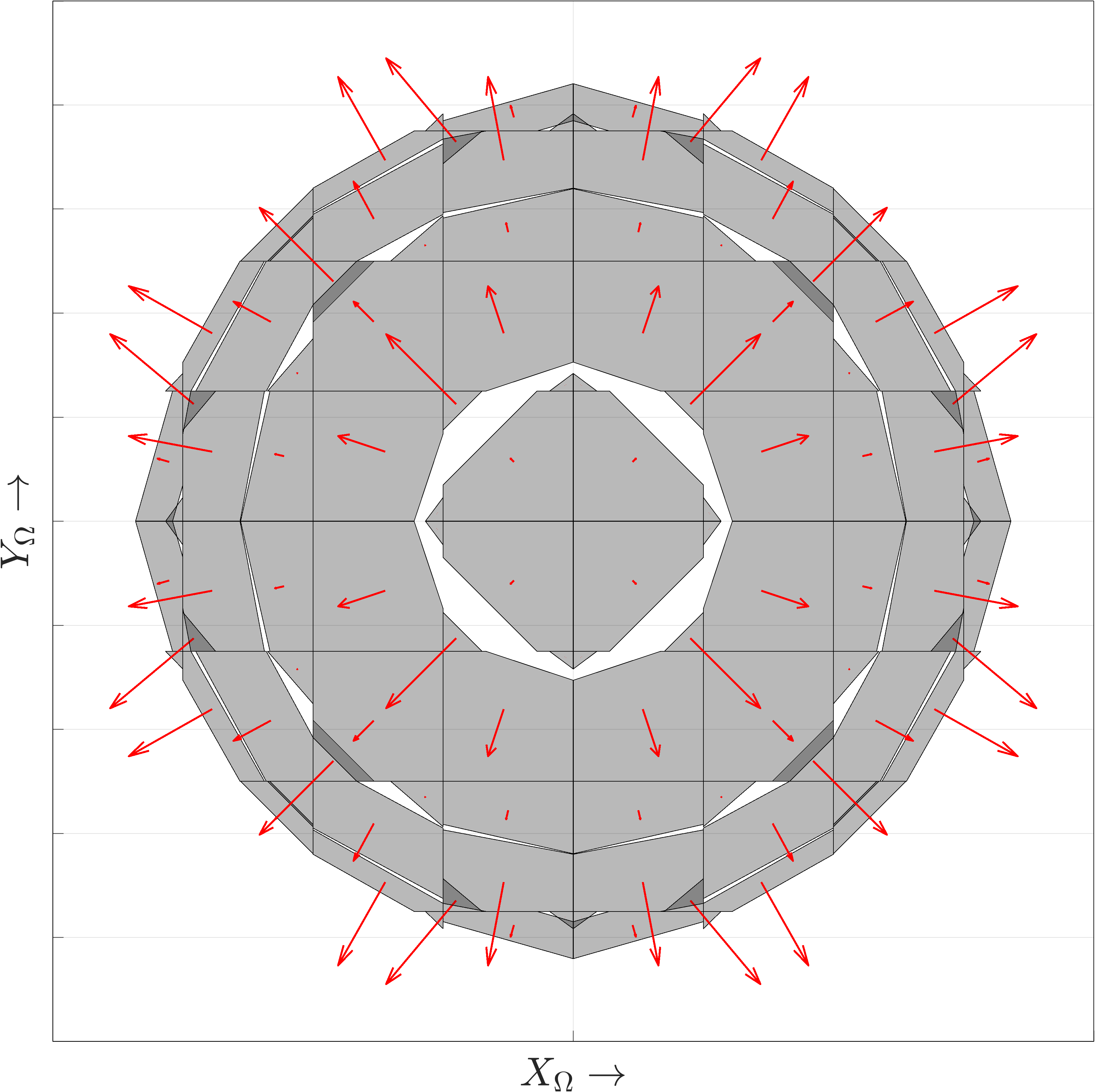}
         \caption{\protect Normals cf. \eqref{eq:normal:kabel} -\\ $8\times8\times8$}
         \label{fig:normal:new:2:b}
     \end{subfigure}
      \begin{subfigure}[b]{0.515\columnwidth}
         \centering \captionsetup{justification=centering}
         \includegraphics[width=\textwidth]{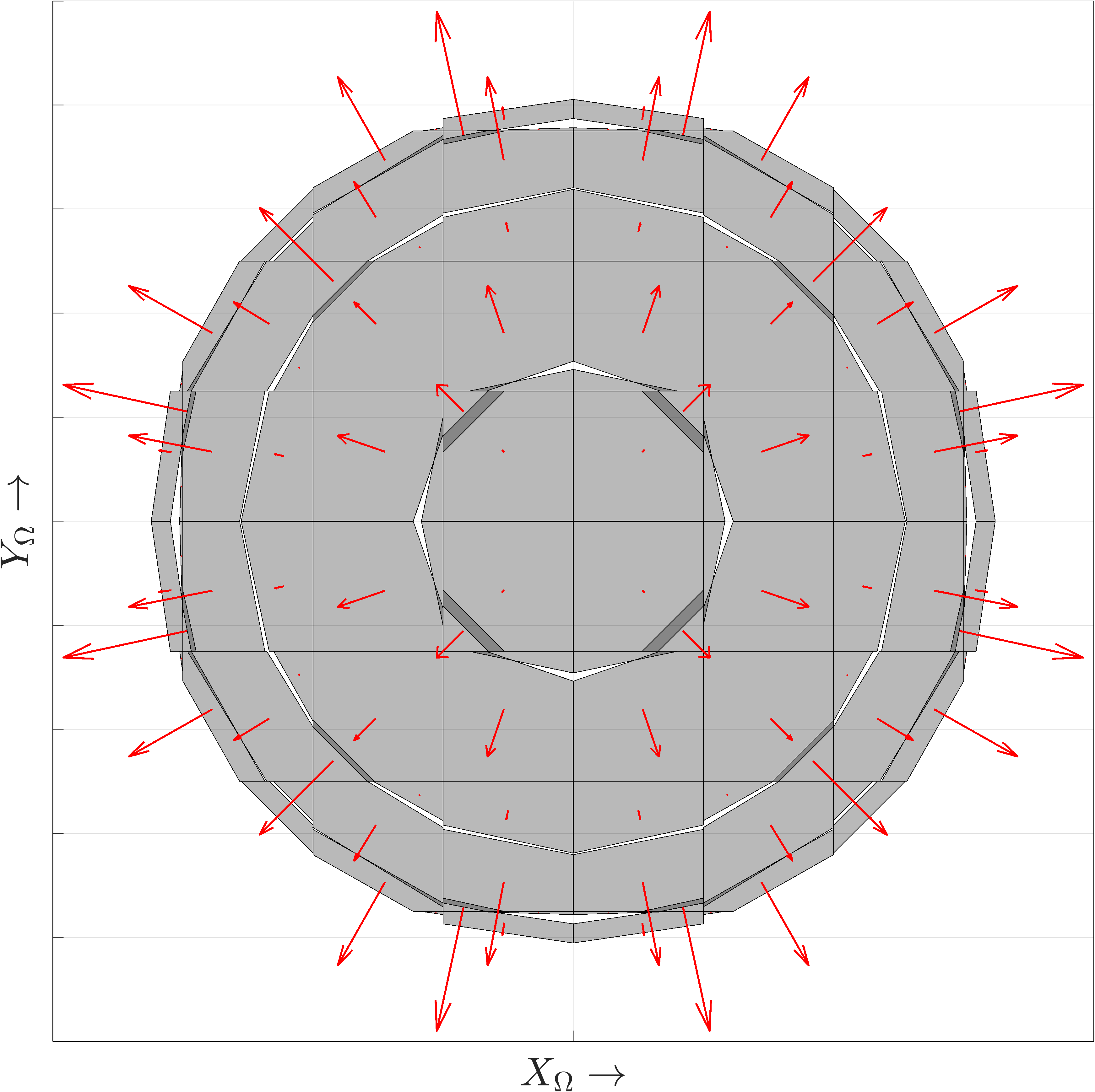}
         \caption{Normals cf. Sec~\ref{sec:normal} - \\ $8\times8\times8$ }
         \label{fig:normal:new:2:d}
     \end{subfigure}
     \begin{subfigure}[b]{0.515\columnwidth}
         \centering \captionsetup{justification=centering}
       \includegraphics[width=\textwidth]{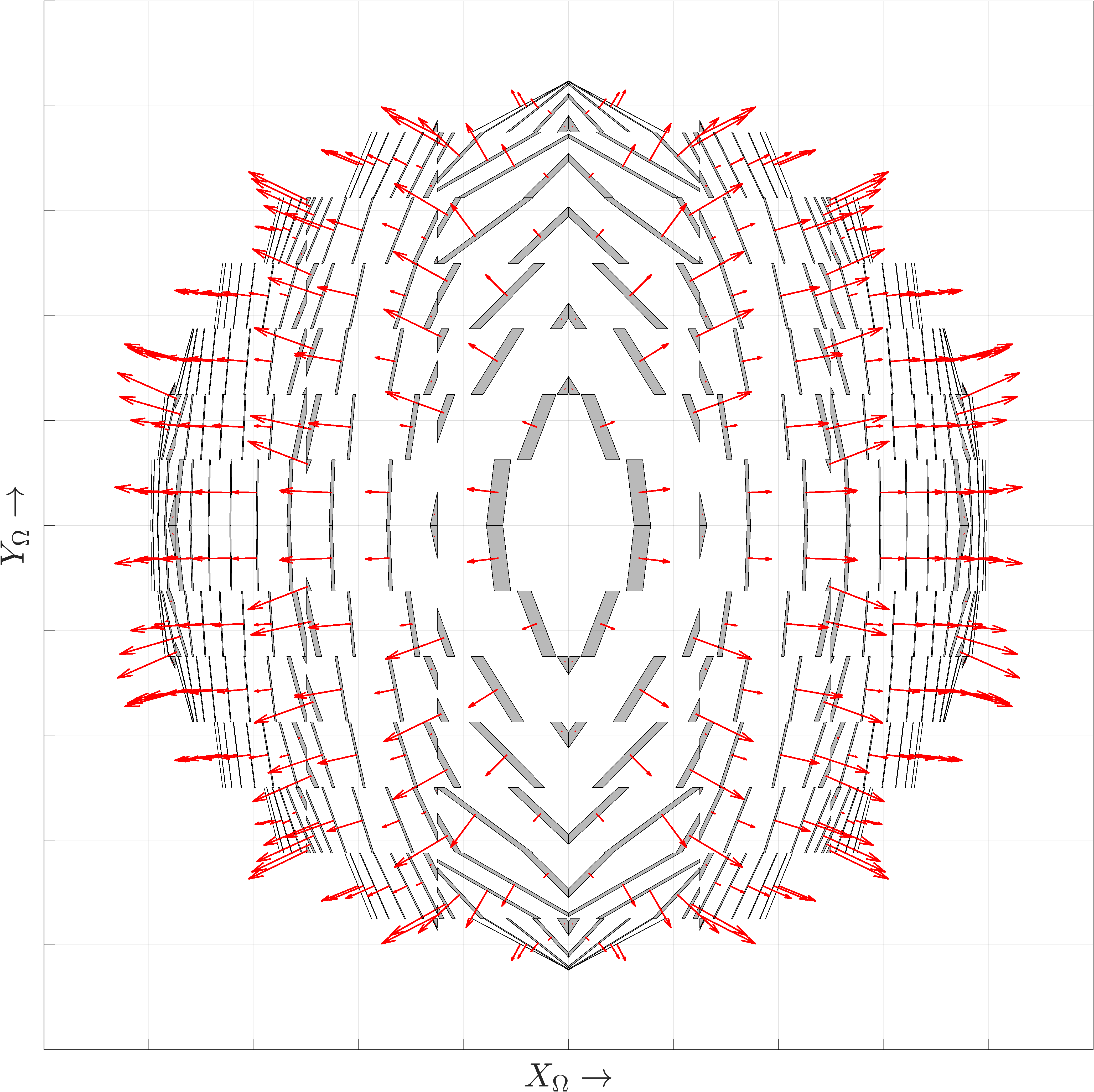}
         \caption{\protect Normals cf. \eqref{eq:normal:kabel} - \\ $8\times16\times32$}
         \label{fig:normal:new:3:a}
     \end{subfigure}
     \begin{subfigure}[b]{0.515\columnwidth}
         \centering \captionsetup{justification=centering}
         \includegraphics[width=\textwidth]{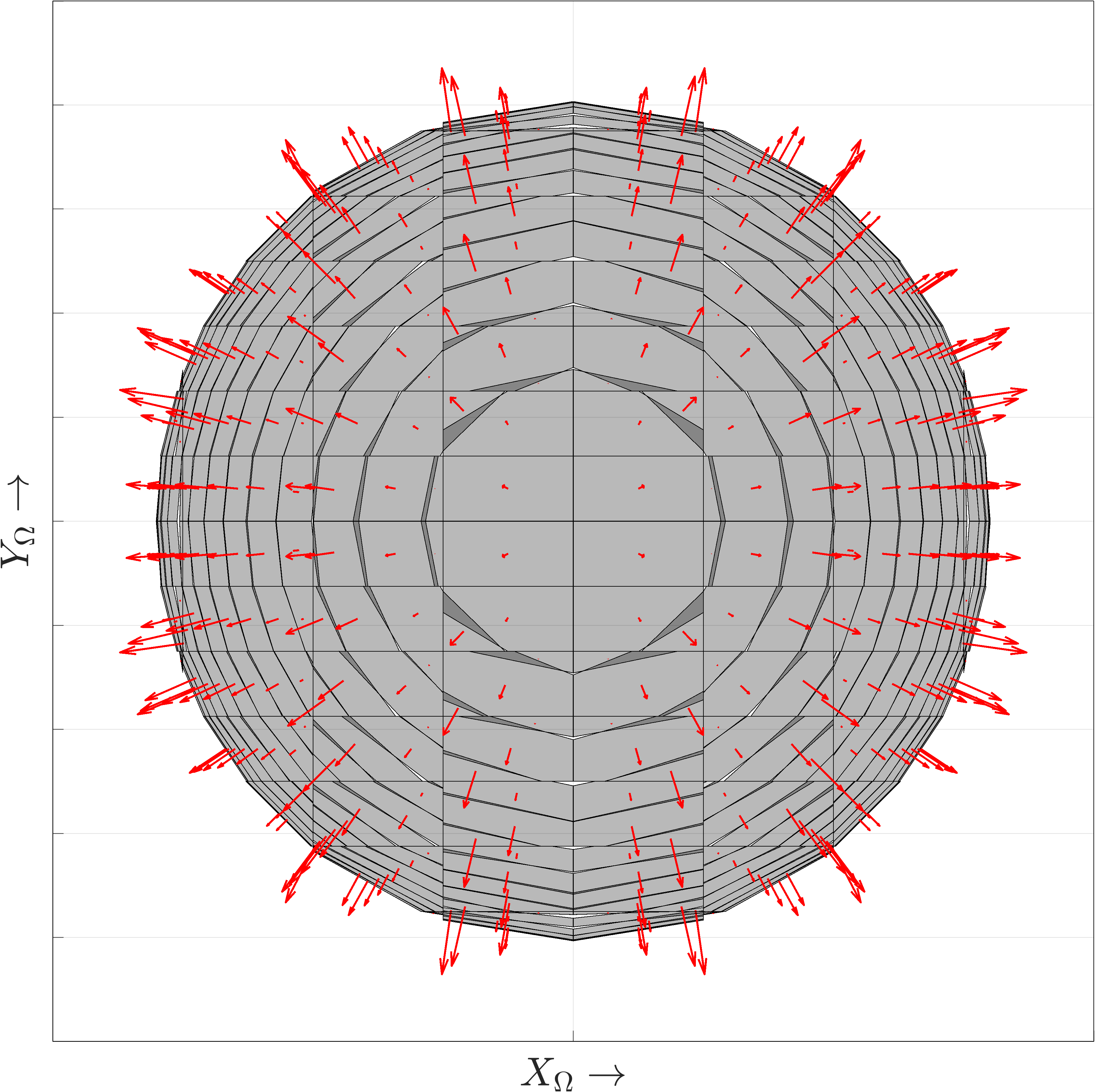}
         \caption{Normals cf. Sec~\ref{sec:normal} - \\ $8\times16\times32$}
         \label{fig:normal:new:3:b}
     \end{subfigure}
     
     \caption{\protect (Top view) Comparison of the established normal computation \cite{Kabel2015} (Figure (a), (c)) and the new approach (Figure (b), (d)) 
     with composite boxels of size $32\times32\times32$ and $32\times16\times8$ respectively. Input image of size 256\textsuperscript{3}.}
    \label{fig:normal:new:3}

\end{figure*}

Next, we have repeated the previous comparison by scaling down from 256\textsuperscript{3} to 8$\times$16$\times$32, i.e., introducing non-equiaxed composite boxels of size 32$\times$16$\times$8. The top views of the concurrent approaches for the normal computation are compared in Fig.~\ref{fig:normal:new:3}. The difference when using boxels is massive, with large gaps showing in Fig.~\ref{fig:normal:new:3:a} for the old method, while basically no gaps and minimal overlap is found in Fig.~\ref{fig:normal:new:3:b} for the new normal identification. The results from Fig.~\ref{fig:normal:new:1}-\ref{fig:normal:new:3} are consistent. They emphasize the relevance of using a dedicated scheme for normal identification. This is evermore so true in the presence of small volume fractions and/or non-equiaxed boxels.

\section{Numerical Examples}
\label{sec:examples}
\subsection{Material models and loading conditions}
In this section, we will focus on the key aspects and novelties of the proposed composite boxel approach for mechanical applications at finite strains. Without any loss of generality, we assume that the materials obey a compressible Neo-Hookean material model 
\begin{align}
    W = \dfrac{1}{2} \lambda (\ln{J})^2 - \mu \ln{J} 
 + \frac{1}{2} \mu (\mathrm{tr}(\fC) - 3).
\end{align}
The Young's modulus $E_\cvpm$ and the Poisson ratios $\nu_\cvpm$ for the inclusion and the matrix phase are denoted by the respective subscripts. We have deliberately chosen a rather high contrast by setting the synthetic parameters
\begin{align}
    E_\cvp &= 10 \text{ GPa}, \;\;\;\; \nu_\cvp = 0.3 \nonumber \\
    E_\cvm &= \;\;1 \text{ GPa}, \;\;\;\; \nu_\cvm = 0.0.
\end{align}
Note that Young's modulus has a distinct contrast of 10, exceeding that of literally all practical metal matrix composites, while the difference in the Poisson ratio is also pronounced\footnote{This is equivalent to a contrast of 25 in the bulk modulus $K$ and a contrast of 7.69 in the shear modulus $G$.}. Elevated contrast in the Poisson ratio is usually detrimental for the convergence of FFT-based schemes as it alters the colinearity of the stiffness of the contained phases, see \cite{Leuschner2018} for a convergence study.
The macroscopic deformation gradient $\overline{\fF}$ imposed on the RVE in the homogenization problem \HOM{} is chosen to be 50\% pure shear in component $\ol{F}_{xy}$
\begin{align} \label{Floading}
    \overline{\fF} = 
    \begin{pmatrix}
     1 & 1/2 & 0\\
     0 & 1 & 0 \\
     0 & 0 & 1 \\
    \end{pmatrix} \, .
\end{align}
The authors also emphasize the arbitrary selection of the material models (trying to trigger elevated material contrasts) and of the imposed kinematic loading. In view of reproducibility we have opted for a simple unambiguous model and a substantial loading considering multiscale applications.\\

\subsection{Spherical inclusion}
\label{subsec:ex:sphere}
Note that all of the following examples are truly 3D although some slice views could lead to the misinterpretation of a 2D problem.

First, we look at a benchmark problem with a spherical inclusion of radius $R = 0.4$ embedded in a 3D matrix.
We aim to identify the impact of the new normal identification (Section~\ref{sec:normal}) for both equiaxed and non-equiaxed composite boxels. The influence on the phase averages is also studied. We start from a fine-scale microstructure of $256^3$ voxels which is down-scaled to a resolution of $32^3$ voxels (downscale factor $512$). The coarsened discretization comprises 2624 ($\sim$~8\%) composite voxels. In each composite voxel, the normal orientation~$\fN$ and the local phase volume fractions~$c_\cvpm$ are computed cf. Section~\ref{sec:normal} and using the established method of \cite{Kabel2015}. The homogenization problem is then solved via Fourier-Accelerated Nodal Solvers (FANS, \cite{Leuschner2018}) using the 8-noded hexahedral finite elements. The convergence criterion is a relative tolerance of 10\textsuperscript{-10} with respect to the $l^\infty$-norm of the nodal force residual vector.

In Figure~\ref{fig:ex:sphere:1}, the converged solution is shown at the center slice $Z = 0$.
We employ a new post-processing procedure that gathers individual field data for each phase and visualizes it using the planar interface with orientation $\fN$.
 
It is obvious that this leads to a much smoother solution close to the interface, even within the individual phases. To our knowledge, a comparable visualization has not been used in previous studies. Further, we have visualized the tractions in the deformed configuration (Fig.~\ref{fig:ex:sphere:1}, (c)) which are unavailable in both non-conforming disretizations as well as in usual conforming FE discretizations.


\begin{figure*}[!h]
	\centering
         
         
         
     \hspace{-0.5cm}
     \begin{subfigure}[b]{0.39\textwidth}
         \centering
         \includegraphics[width=\textwidth]{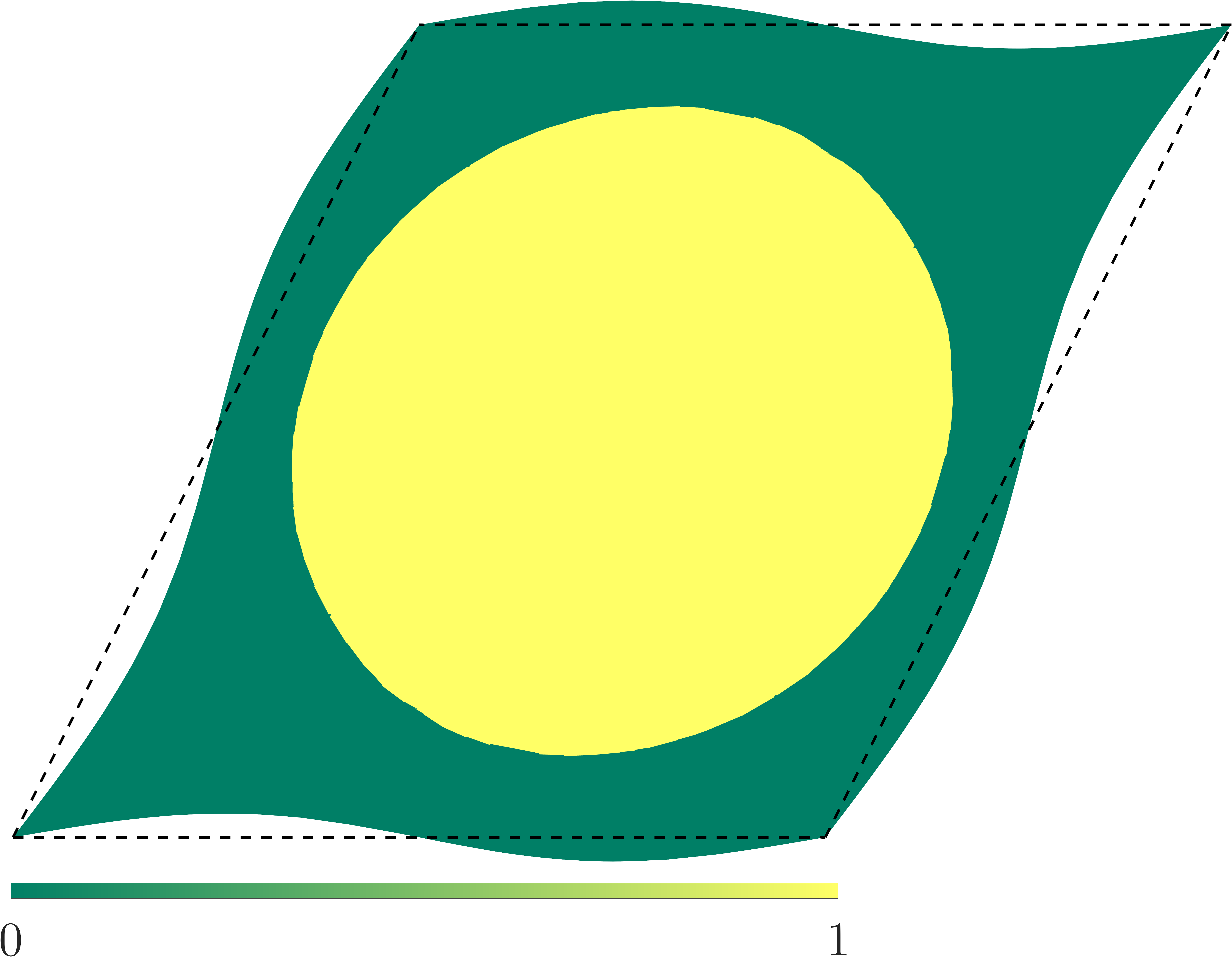}
         \caption{Deformed RVE $\Omega_t$}
         
     \end{subfigure}
     \hspace{-1.4cm} 
         \begin{subfigure}[b]{0.39\textwidth}
         \centering
         \includegraphics[width=\textwidth]{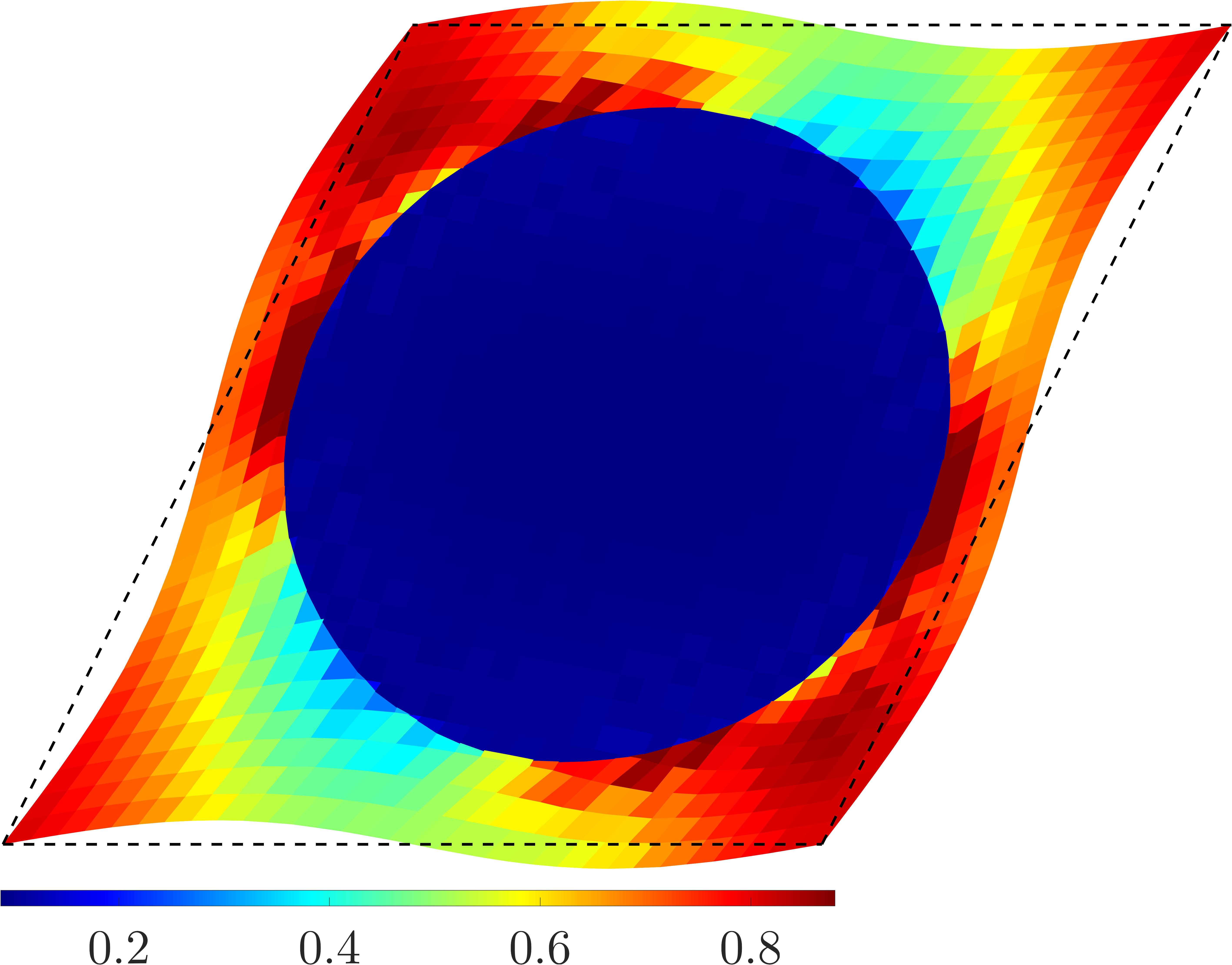}
         \caption{Almansi strain - $e_{xy}$}
         
    \end{subfigure}
    \hspace{-1.4cm} 
    \begin{subfigure}[b]{0.39\textwidth}
         \centering
         \includegraphics[width=\textwidth]{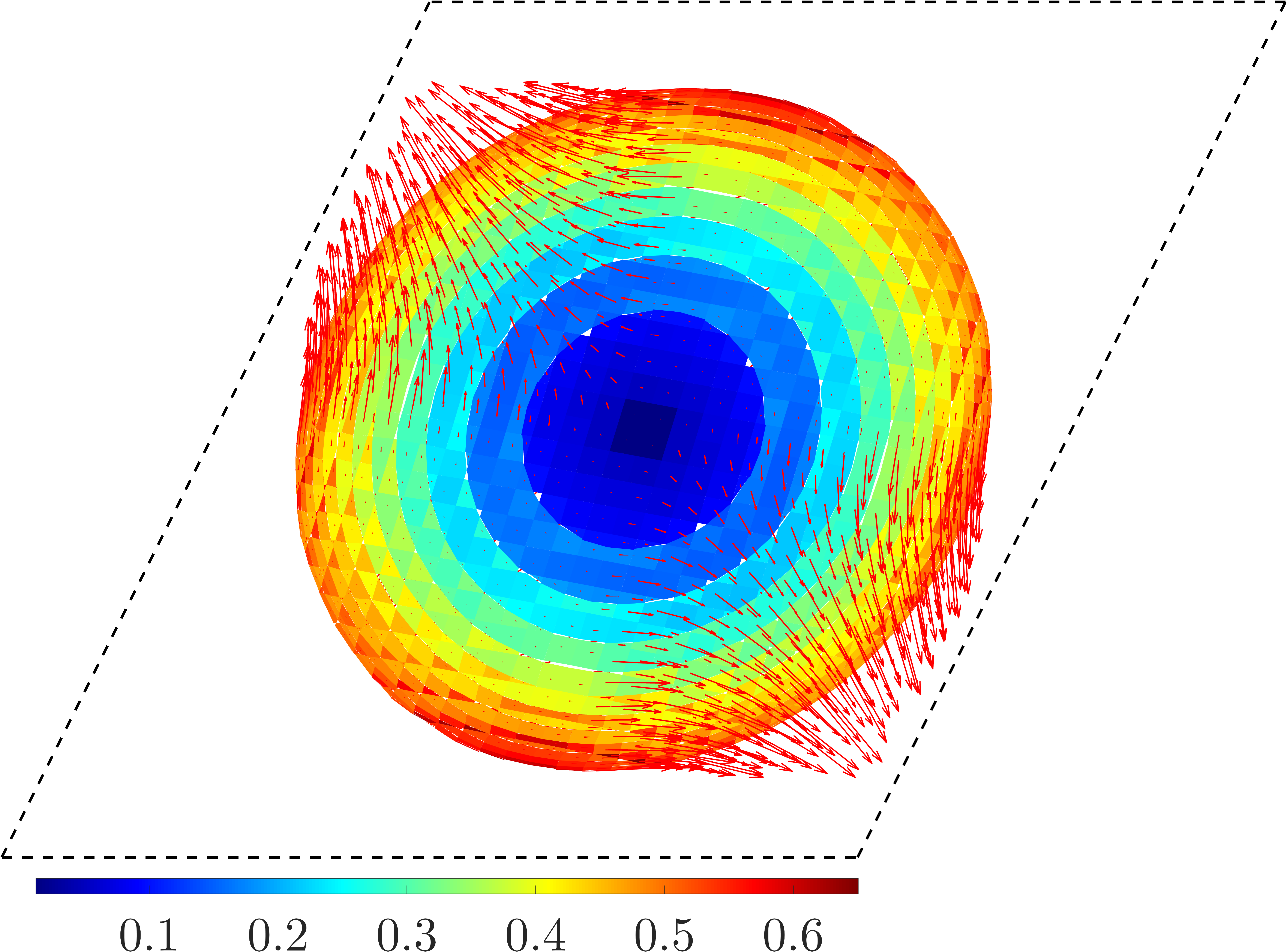}
         \caption{tractions $\ft(\fx)$ and $\| \ft\|_2$ on $\mathscr{S}_t$}
    \end{subfigure}
    
    \caption{\protect Composite Boxel (ComBo) solution using CG-FANS \cite{Leuschner2018} with full integration for the spherical inclusion example of Section~\ref{subsec:ex:sphere}: spatial configuration at slice $Z = 0$}
    \label{fig:ex:sphere:1}
\end{figure*}

Regarding the accuracy of the effective stress tensor~$\ol{\fP}$ and the phase averages $\ol{\fP}_\cvpm$, a study was performed which compares the 256\textsuperscript{3} reference solution using the procedure suggested in \cite{Willot2015} for five different ComBo discretizations using reduced integration CG-FANS \cite{Leuschner2018} (FANS HEX8R) and using the normals from \eqref{eq:normal:kabel} matching previous studies as well as the improved normals from Section~\ref{sec:normal}. The results are summarized in Table~\ref{tab:ex:sphere:P1}. Note that the average stress only reflects part of the actual accuracy of the solver as it neglects local field fluctuations, which are examined in the following examples, see Section~\ref{subsec:ex:polyhedron}. It can be noted that the errors in $\ol{\fP}$ of all composite discretizations are below 1.15\% for the normals cf. \cite{Kabel2015} and 0.78\% for the new normals suggested in Section~\ref{sec:normal}, respectively, i.e., the new normals reduce the error in all averaged stresses by approximately 30\% for equiaxed composite voxels and up to 1\,000\% for the non-equiaxed composite boxels. Remarkably, these improvements come at no additional computational expense, but they owe only to the more accurate orientation information.

\begin{table*}[!h]
\begin{center}
\begin{minipage}{\textwidth}
\caption{\protect Comparing averaged 1\textsuperscript{st} Piola-Kirchoff stresses in the spherical inclusion problem cf. Section~\ref{subsec:ex:sphere}: different resolutions, normals obtained via cf. \cite{Kabel2015} and cf. Sec.~\ref{sec:normal}, with errors against a reference solution $(256^3)$ based on \cite{Willot2015}; ComBo solutions are obtained using FANS with reduced integration \cite{Leuschner2018}; gray background highlights better result (old normals vs. new normals); 
Error defined as the relative Frobenius norm w.r.t reference solution
\label{tab:ex:sphere:P1}
}
\scriptsize
\begin{tabular*}{\textwidth}{@{\extracolsep{\fill}}crrrrrrr@{\extracolsep{\fill}}}
\toprule%
 &&\multicolumn{3}{c}{Normals cf. \cite{Kabel2015}} & \multicolumn{3}{c}{Normals cf. Sec.~\ref{sec:normal}}\\
\cmidrule(lr){3-5}
\cmidrule(lr){6-8} 
 &&\multicolumn{3}{c}{Error (\%)} & \multicolumn{3}{c}{Error (\%)} \\
\cmidrule(lr){3-5}
\cmidrule(lr){6-8} 
Resolution & Downscale & {$\overline{\fP}$} & {$\overline{\fP}_\cvm$} & {$\overline{\fP}_\cvp$} & {$\overline{\fP}$} & {$\overline{\fP}_\cvm$} & {$\overline{\fP}_\cvp$} \\
\midrule
$ 8 \times 8 \times 8$ & 32768 & 1.145 & 0.263 & 3.537 & \cellcolor{uniSgray20} 0.771 & \cellcolor{uniSgray20} 0.177 & \cellcolor{uniSgray20} 2.361 \\ [2pt]
$ 16 \times 16 \times 16$ & 4096 & 0.160 & 0.042 & 0.501 & \cellcolor{uniSgray20} 0.053 & \cellcolor{uniSgray20} 0.028 & \cellcolor{uniSgray20} 0.172\\[2pt]
$ 8 \times 16 \times 32$ & 4096 & 1.586 & 0.282 & 4.609 & \cellcolor{uniSgray20} 0.177 & \cellcolor{uniSgray20} 0.064 & \cellcolor{uniSgray20} 0.584 \\[2pt]
$ 32 \times 32 \times 32$ & 512 & 0.072 & 0.019 & 0.219 & \cellcolor{uniSgray20} 0.070 & \cellcolor{uniSgray20} 0.018 & \cellcolor{uniSgray20} 0.212\\[2pt]
$ 16 \times 32 \times 64$ & 512 & 0.801 & 0.156 & 2.335 & \cellcolor{uniSgray20} 0.100 & \cellcolor{uniSgray20} 0.019 & \cellcolor{uniSgray20} 0.262 \\
\botrule
\end{tabular*}
\end{minipage}
\end{center}
\end{table*}

\subsection{Composite boxels and local solution field quality}
\label{subsec:ex:polyhedron}
The straightforward implementation of composite boxels into various FFT-based schemes makes them truly versatile. Here, we present some of the most popular methods used in tandem with composite boxels.
Most importantly, the composite boxel method can be used to replace any call to a constitutive model, independent of the discretization method.

In figure~\ref{fig:ex:polyhedron},  we consider a random polyhedral inclusion surrounded by matrix material. The fine-scale microstructure is generated at a resolution of 256\textsuperscript{3}. The ComBo discretization is 32\textsuperscript{3} with normals gained cf. Section~\ref{sec:normal}, yielding a 512 times smaller problem. We further go on to compare the full field solutions for the $E_{XY}$ component of the Green-Lagrange strain at the center slice ($Z=0$) for various popular FFT-based solution strategies. A reference solution is computed on the original fine-scale problem (without any composite boxels).
\begin{figure*}[h!]
	\centering
     \begin{subfigure}[b]{0.24\textwidth}
         \centering
         \includegraphics[width=\textwidth]{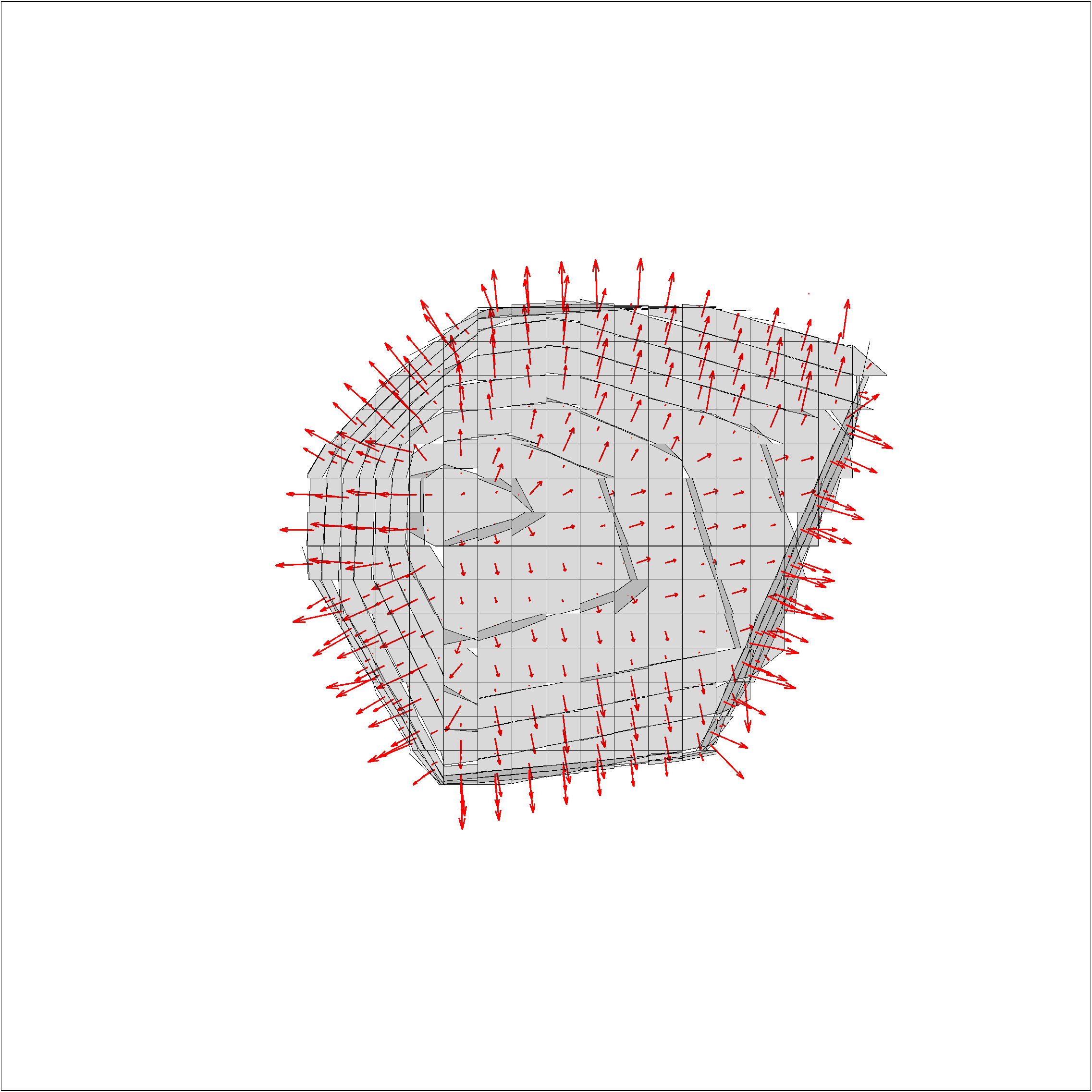}
     \end{subfigure}
	\begin{subfigure}[b]{0.24\textwidth}
         \centering
         \includegraphics[width=\textwidth]{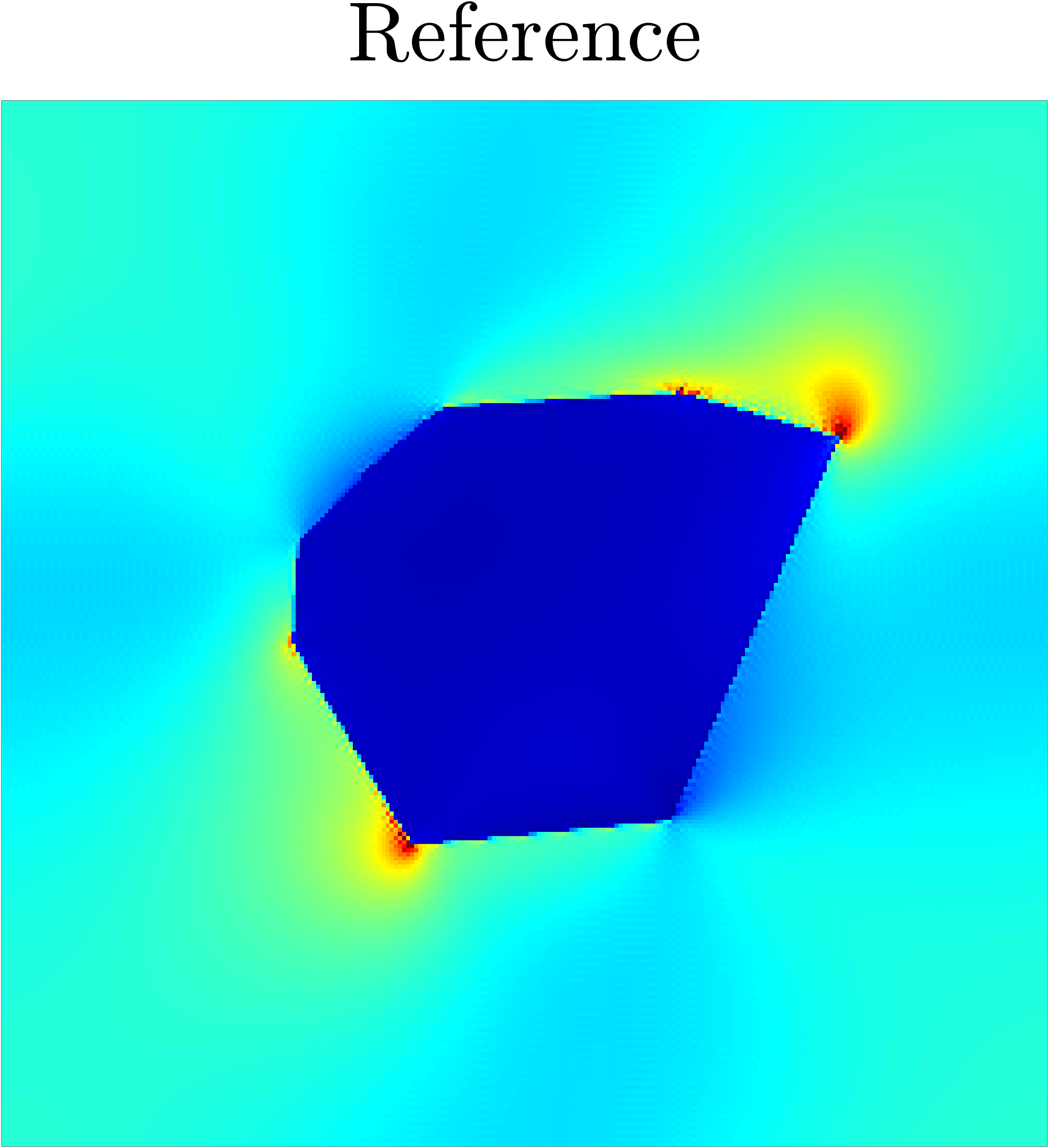}
     \end{subfigure}
	\begin{subfigure}[b]{0.24\textwidth}
         \centering
         \includegraphics[width=\textwidth]{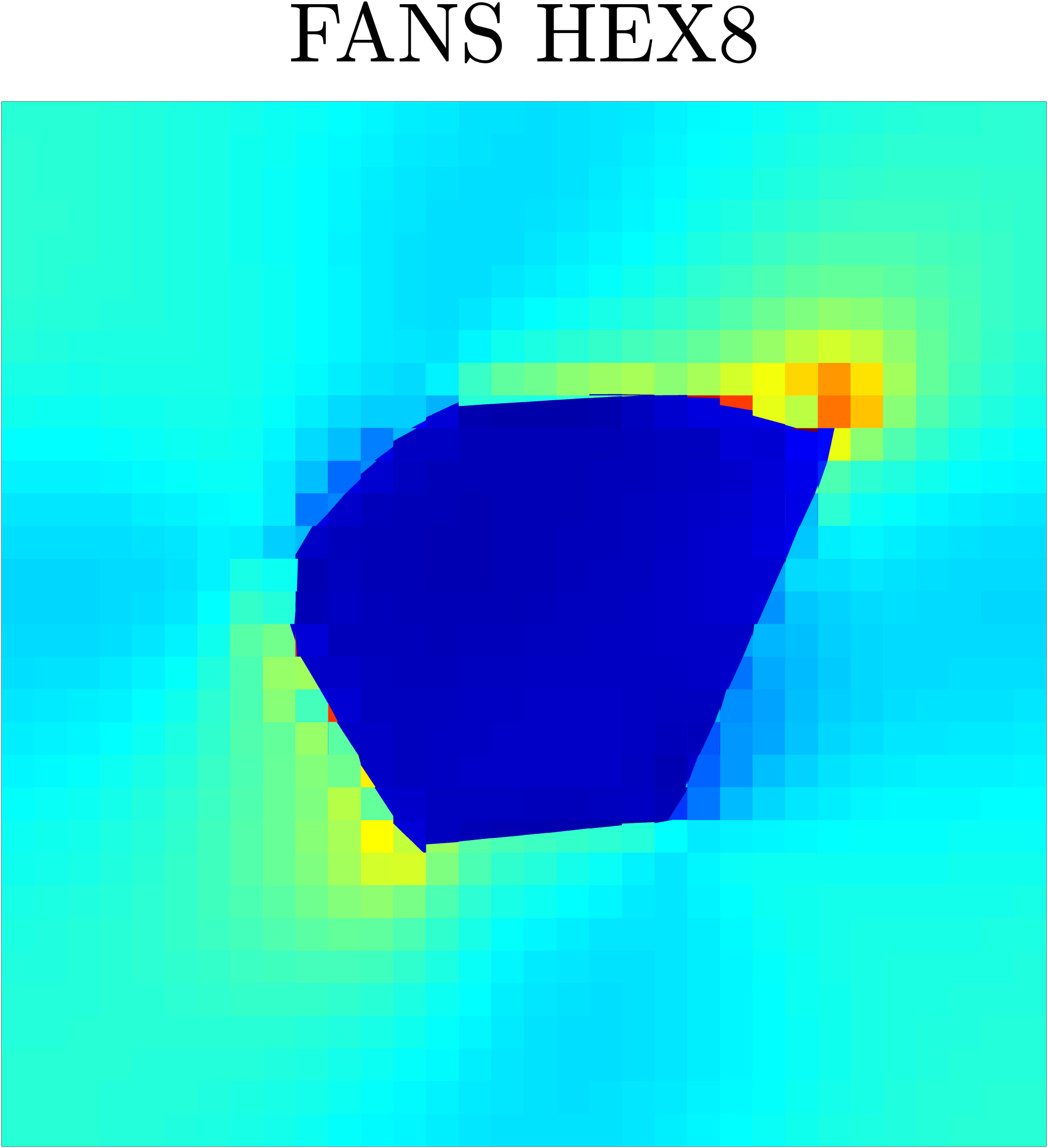}
     \end{subfigure}
     \begin{subfigure}[b]{0.24\textwidth}
         \centering
         \includegraphics[width=\textwidth]{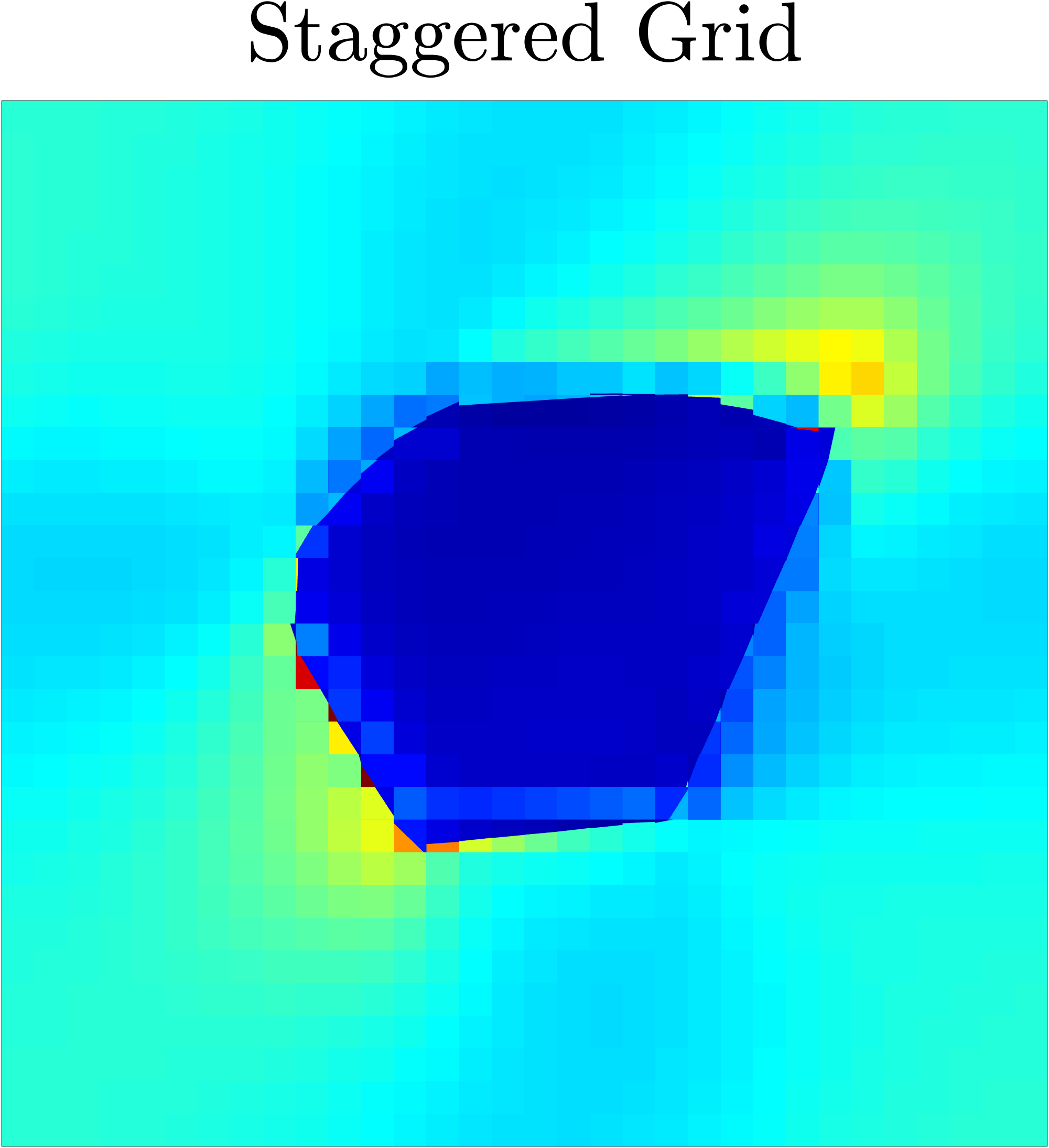}
     \end{subfigure} 
     \begin{subfigure}[b]{0.24\textwidth}
         \centering
         \includegraphics[width=\textwidth]{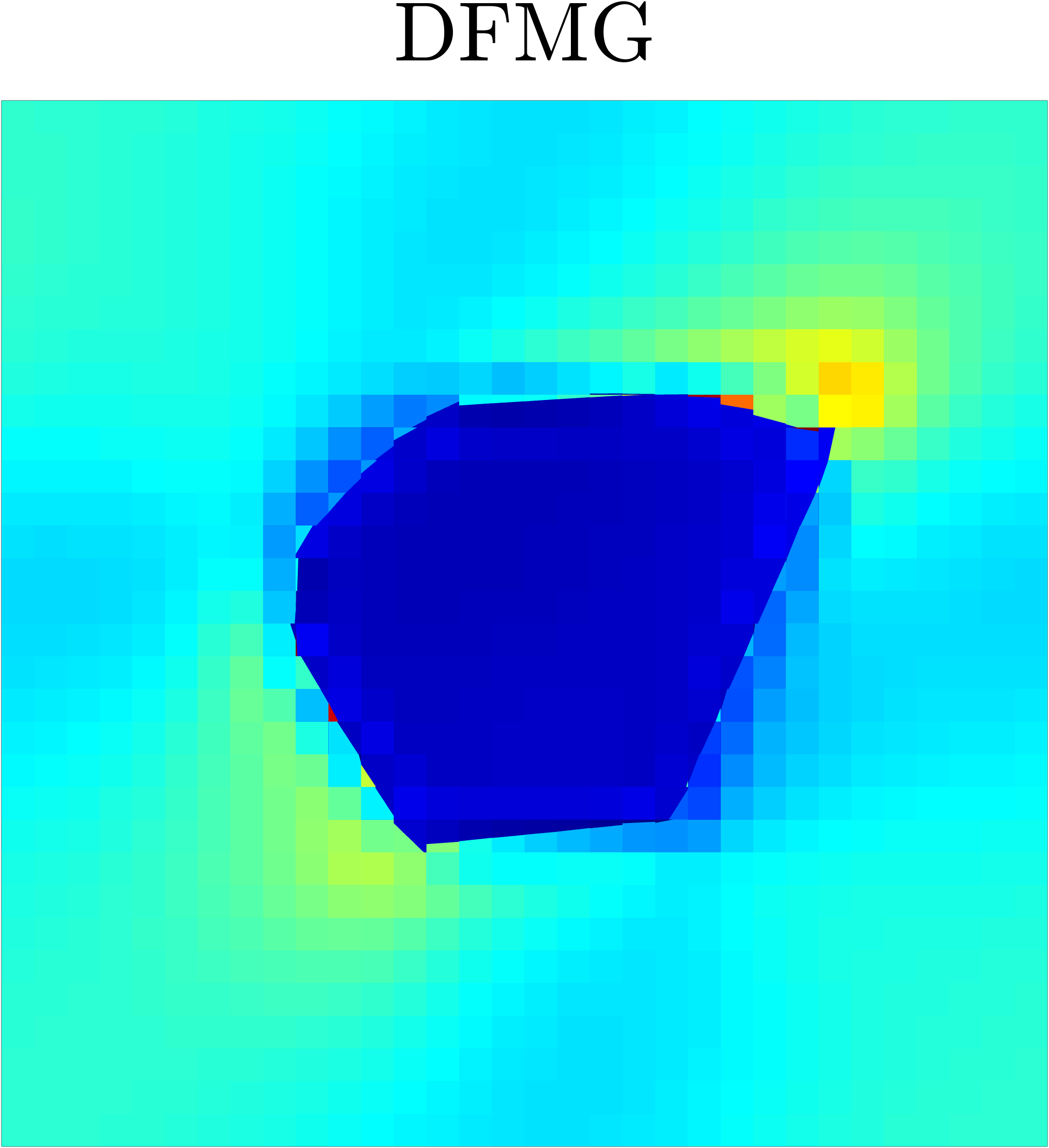}
     \end{subfigure}
     \begin{subfigure}[b]{0.24\textwidth}
         \centering
         \includegraphics[width=\textwidth]{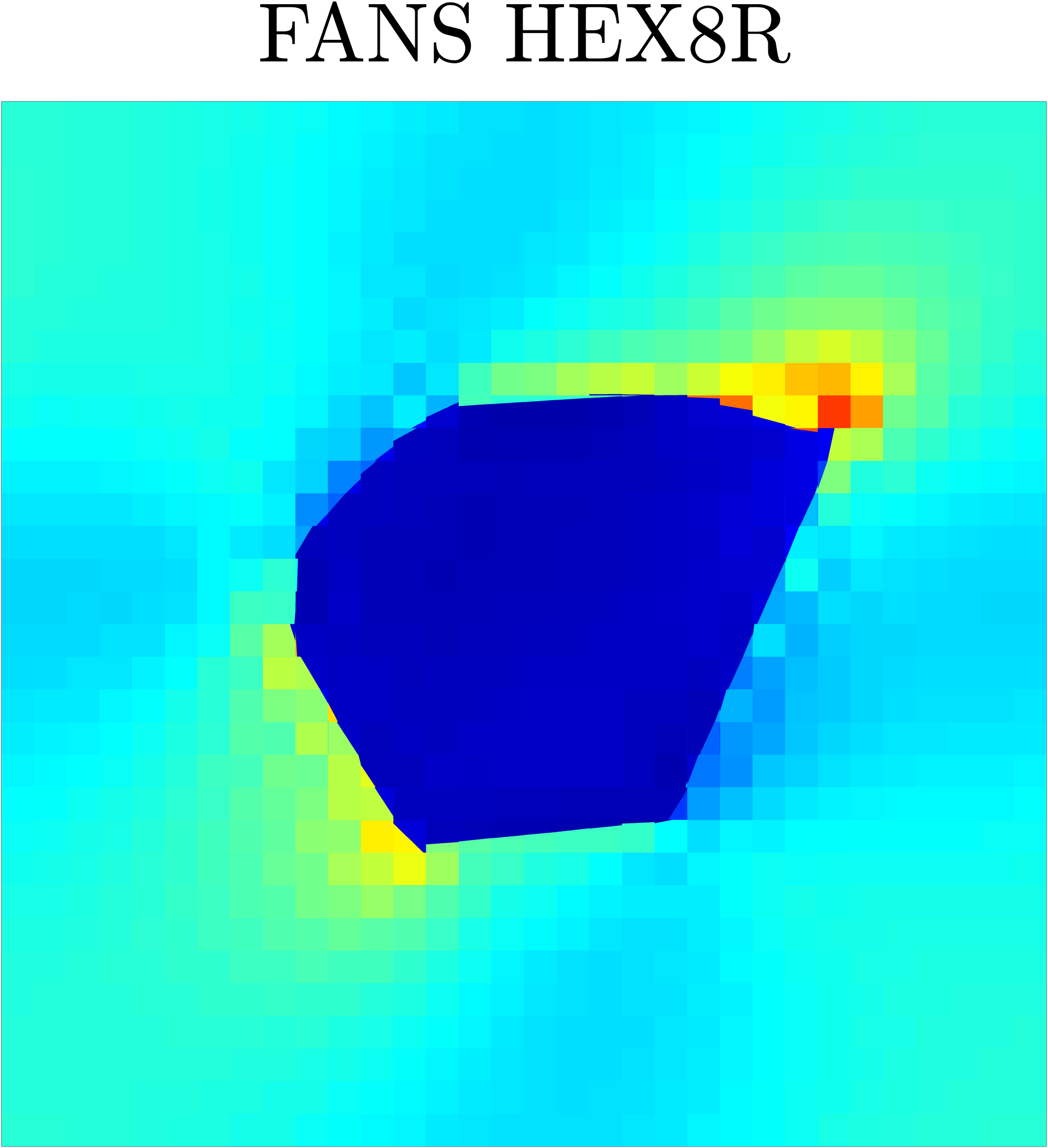}
     \end{subfigure}
         \begin{subfigure}[b]{0.24\textwidth}
         \centering
         \includegraphics[width=\textwidth]{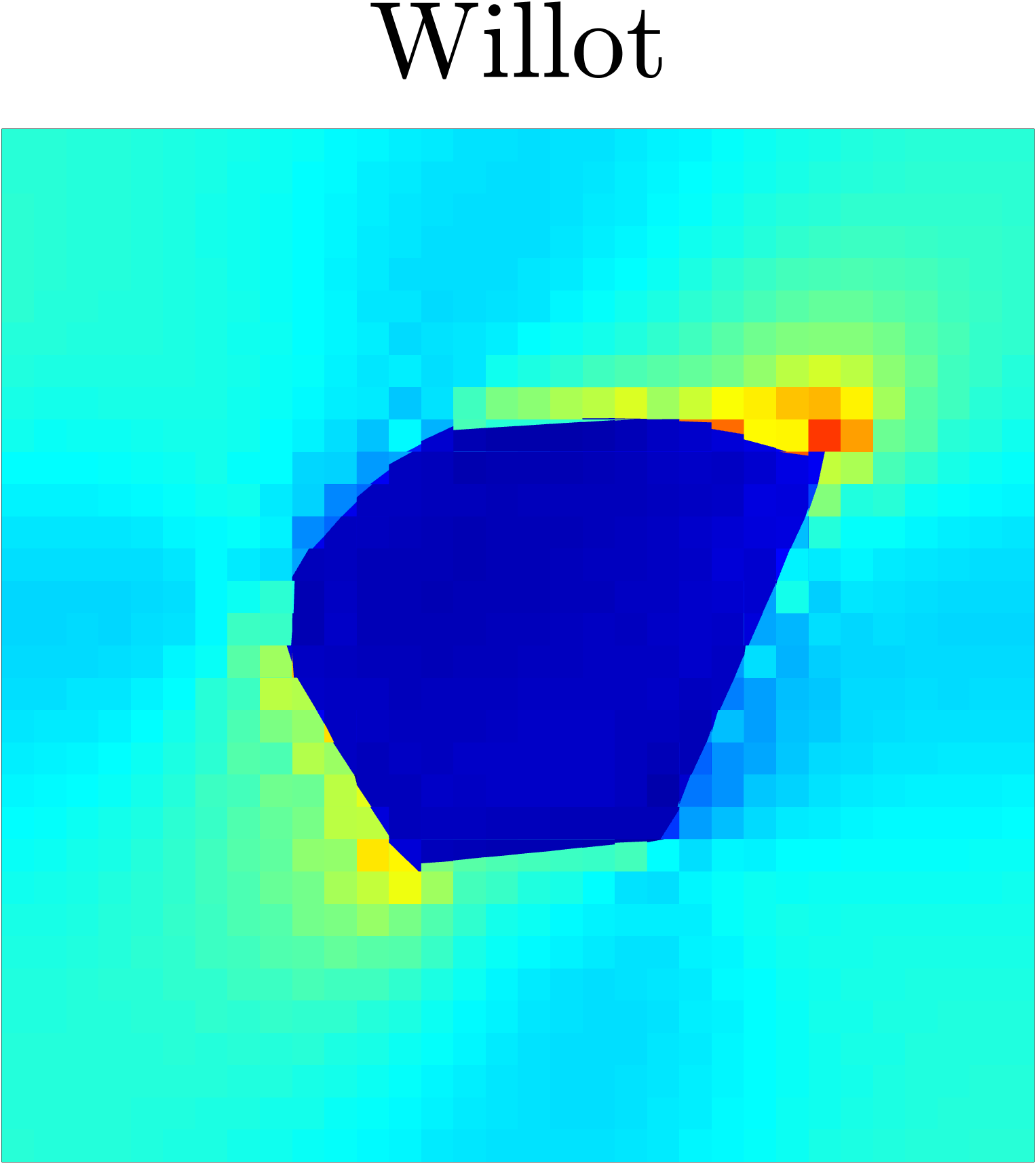}
    \end{subfigure}
    \begin{subfigure}[b]{0.24\textwidth}
         \centering
         \includegraphics[width=\textwidth]{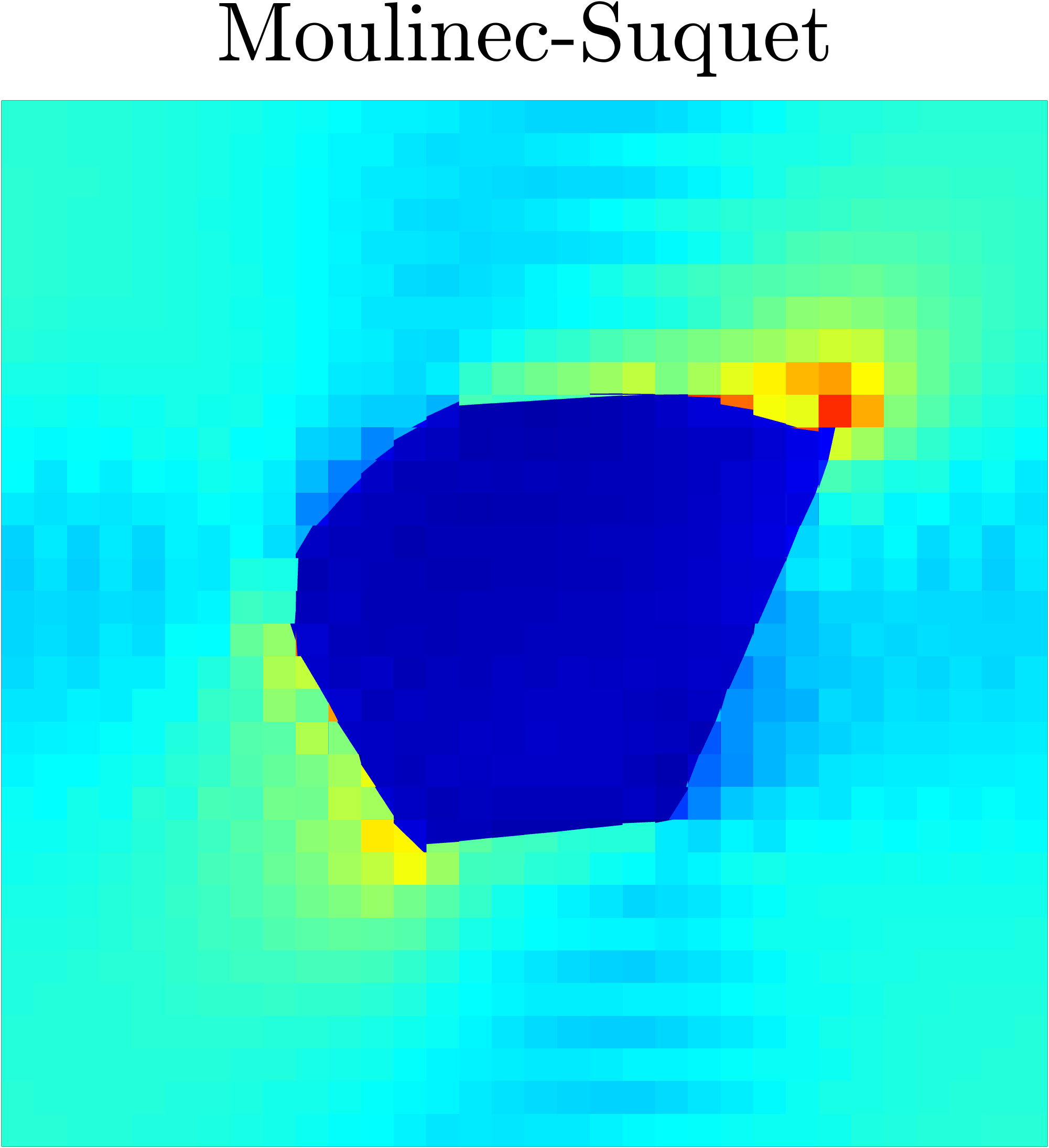}
    \end{subfigure}
    \\
     \begin{subfigure}[b]{0.5\textwidth}
         \centering
         \includegraphics[width=\textwidth]{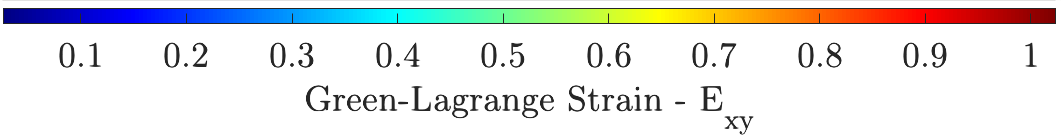}
    \end{subfigure}
    \caption{Solutions for the Green-Lagrange strain $E_{XY}$ obtained by different FFT-based schemes for the polyhedral inclusion problem (Resolution: $32^3$) and comparision against reference solution (Resolution: $512^3$) at slice $Z=0$}
    \label{fig:ex:polyhedron}
\end{figure*}

The solution using FANS with full integration (8 Gauss points per element) finite elements (HEX8) is shown next to the reference. It exhibits no staircasing and no artefacts even close to the interface. The solution quantitatively captures the behavior well in the interior of the phases and near the phase boundaries compared to the reference solution even for this coarse discretization. Stereotypical of the fully integrated HEX8 elements, the response is a tad stiffer which is reflected in the phase-wise averages, Table~\ref{tab:ex:poly:P1}.
The solution obtained by the staggered grid approach, despite its advantages in an algorithmic sense, lacks symmetry in the physical location of the strain components in each voxel. Hence, only an interpolated measure can be obtained at the boxel center for visualization. This ambiguity, unfortunately, leads to poor results close to the material interface despite the use of the ComBo discretization which is also reflected in the homogenized quantities. 
This issue can partially be overcome by the proposed double-fine material grid (DFMG, see Appendix~\ref{subsec:finite_difference_staggered}) approach at the cost of extra material law evaluations compared to the original staggered grid approach. The DFMG approach also interpolates the field quantities to the boxel center, which causes blurring/smoothing everywhere, including the material boundaries: local solution field accuracy is sacrificed, although DFMG outperforms the staggered grid approach, and it preserves existing symmetries.


FANS HEX8R, i.e., with reduced integration (1 Gauss point per element) is numerically similar to the HEX8R discretization by Willot \cite{Willot2015},  as shown by \cite{Schneider2016}. The FANS HEX8R solution suffers from hour-glassing but the amplitude of the hourglass modes is greatly diminished in the presence of composite boxels compared to when no composite boxels are in use. The convergence behavior of FANS HEX8R is very much on par with FANS with fully integrated HEX8 elements.
Finally, we also employ the original Moulinec-Suquet scheme which has been extensively studied in the literature 
and suffers from spurious oscillations, although predicting the homogenized quantities very well.

\begin{table}[h!]
\begin{minipage}{\columnwidth}
\caption{\protect Comparing averaged 1\textsuperscript{st} Piola-Kirchoff stresses in the Polyhedron problem ($32^3$) cf. Section~\ref{subsec:ex:polyhedron}: different FFT based methods: with errors against a reference solution $(256^3)$ based on \cite{Willot2015}; Error defined as the relative Frobenius norm w.r.t reference solution
}\label{tab:ex:poly:P1}
\scriptsize
\centering
\begin{tabular}{@{\extracolsep{\fill}}lrrr@{\extracolsep{\fill}}}
\cmidrule(lr){1-4}
 &\multicolumn{3}{c}{Error (\%)} \\
\cmidrule(lr){2-4}
Method &  {$\overline{\fP}$} & {$\overline{\fP}_\cvm$} & {$\overline{\fP}_\cvp$} \\
\midrule
Staggered Grid & 1.219 & 0.202 & 16.993\\[2pt]
DFMG &  0.661 & 0.103 & 9.147\\[2pt]
FANS HEX8 & 0.142 & 0.024 & 1.947 \\[2pt]
Willot & 0.065 & 0.011 & 0.903 \\[2pt]
FANS HEX8R & 0.049 & 0.008 & 0.675\\[2pt]
Moulinec-Suquet & 0.026 & 0.005 & 0.349 \\[2pt]
\botrule
\end{tabular}
\end{minipage}
\end{table}

\subsection{Selective back-projection in practice}
\label{subsec:ex:polyhedron:bp}
The selective back-projection introduced in Section~\ref{subsec:backproj} is investigated for the polyhedral microstructure of the previous Section~\ref{subsec:ex:polyhedron}. First, we demonstrate the influence it has on the convergence of the NR scheme inside of a single critical voxel, see Figure~\ref{fig:matdetlemma_ex}.

\begin{figure}[h!]
	\centering
	\includegraphics[width=\columnwidth]{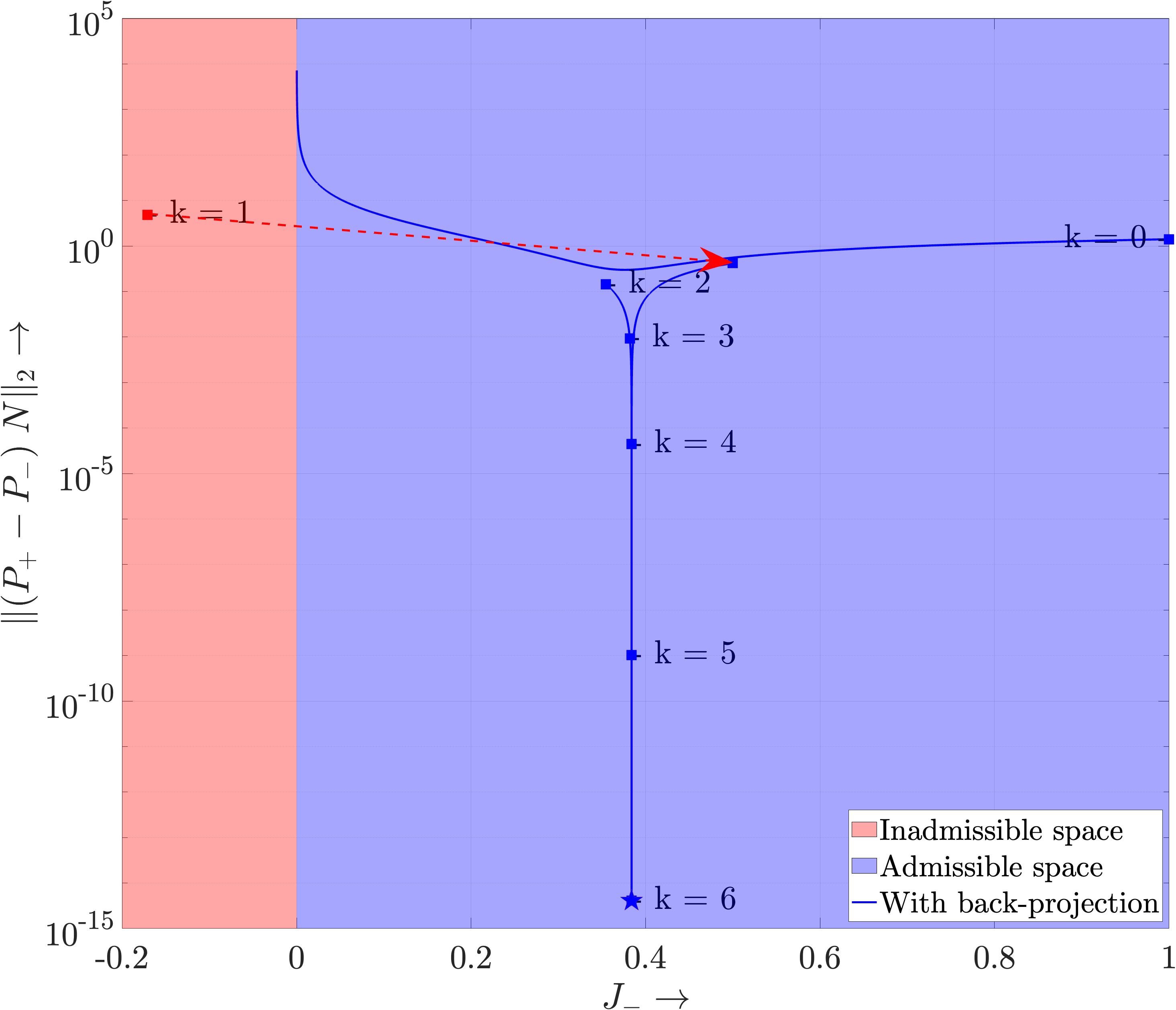}
    \caption{Convergence behavior of the Newton-Raphson algorithm with back-projection}
    \label{fig:matdetlemma_ex}
\end{figure}

The traction balance residual vs. material Jacobian of phase $\cvm$ is plotted in a particular case when back-projection is employed. In the case discussed, the volume fraction of the inclusion phase $\cvp$ is 99.625\%. The Newton-Raphson algorithm~\ref{alg:combo_NR} starts with an admissible initial guess (at $k=0$), but the naive NR-update would push the iterate $\fa^{[1]}$ to the inadmissible domain which is highlighted in red in Fig.~\ref{fig:matdetlemma_ex}. The solid blue line tracks the continuous intermediate residual between the current iterate $\fa^{[k]}$ and the subsequent iterate $\fa^{[k+1]}$ along a line search parameter. Note that if either one of $J_\cvpm$ tends to 0, the residual tends to $\infty$, which is characteristic of $\log$ based hyperelastic strain energies \cite{Doll2000}. It is also noteworthy that, although the solution lies in between $\fa^{[0]}$ and $\fa^{[1]}$ in this plot, the continuous residual does not go to zero anywhere. This is due to the projection onto the $J_\cvm$ axis (corresponding to the $\beta$ axis modulu scaling) which cannot account for incorrect components $\fM^\perp_\beta \fa^{[k]}$ -- i.e., a mere line search \textit{cannot} suffice, in general. Using the selective back-projection cf. algorithm~\ref{alg:backprojection} yields a valid iterate $\fa^{[1]}$. From there on the NR algorithm~\ref{alg:combo_NR} converges to a physically feasible solution within 6 iterations up to machine precision and in practice one could stop after just four iterations. 
Each iteration of our NR algorithm equates to a single evaluation of the constitutive law for either phase, independent of whether selective back-projection is needed.

Despite striking similarities with the algorithm proposed in \cite{Kabel2016}, we developed our scheme independently since we observed that it was needed during the simulations. While the authors of \cite{Kabel2016} state that "The projection and backtracking steps occur only rarely", we would like to emphasize that a single inadmissible iterate can break the entire simulation (e.g. leading to negative $J_\cvpm$ used in $\log$-energies). Further, we have investigated the percentage of the composite boxels that require selective back-projection as a function of the loading for the polyhedral inclusion, see Figure~\ref{fig:bp_polyhedron1}. 

\begin{figure}[h!]
	\centering
	\includegraphics[width=\columnwidth]{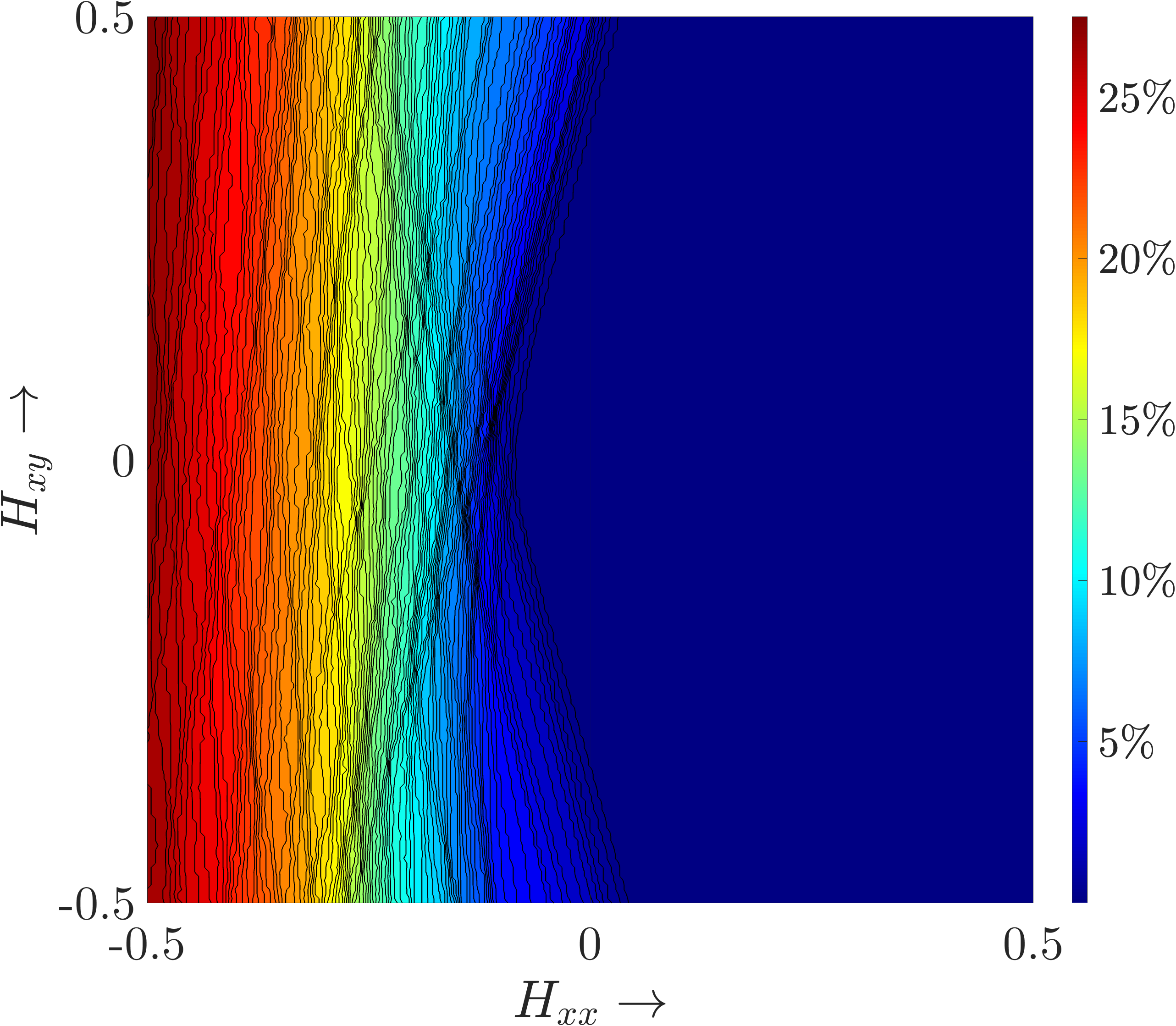}
    \caption{Percentage of boxels (Problem c.f. Sec~\ref{subsec:ex:polyhedron}) with failed naive N-R updates for mixed displacement gradient loading conditions}
    \label{fig:bp_polyhedron1}
\end{figure}

It turns out that our algorithm is needed in a majority of the composite voxels under certain loading conditions. This becomes critical in actual two-scale simulations, where spuriously high loadings are frequently observed in isolated points, especially in the compression regime.

On a side note we would like to emphasize that in our tests, the old normal direction cf. \cite{Kabel2015} further increased the number of failed iterates. Hence, using the new normal detection cf. Section \ref{sec:normal} alongside Algorithm~\ref{alg:backprojection} improves on the robustness in both regards.

\subsection{Short fiber reinforced microstructures}
\label{subsec:complex}
In this Section, we investigate fiber-reinforced composites with global fiber directional affinity to show the effectiveness of non-equiaxed composite boxels in certain use cases. We consider the two microstructures shown in Figure~\ref{fig:ex:frp1} and \ref{fig:ex:frp2} referred to as [FPR-1] and [FRP-2] respectively. 
Microstructure [FRP-1] is a fibrous microstructure with {\it almost} aligned fibers with a fiber volume fraction of $\sim 4 \%$ primarily oriented along the $X$-axis. This problem is down-scaled to varied equiaxed and non-equiaxed resolutions starting from a fine-scale image comprising 504\textsuperscript{3} voxels tabulated in Table~\ref{table:1pk_comparision-FRP1}. Similarly, the microstructure [FRP-2] hosts 150 fibers almost aligned along the $X$-axis with a fiber volume fraction of $\sim15\%$. The fibers have a length of $120\mu m$ and a diameter of $12\mu m$ and the minimum fiber distance is about $2\mu m$. This problem is down-scaled to varied equiaxed and non-equiaxed resolutions starting from a fine-scale image comprising 240\textsuperscript{3} voxels tabulated in Table~\ref{table:1pk_comparision-FRP2}.

The overall homogenized first Piola-Kirchoff stress and its phase-wise counterparts are compared towards a reference solution computed without the use of composite boxels and using the FFT solver proposed in \cite{Willot2015} while the coarse-grained models are solved using FANS HEX8R.

\begin{figure*}[!h]
    \begin{subfigure}[b]{0.45\textwidth}
         \centering
         \includegraphics[width=\textwidth]{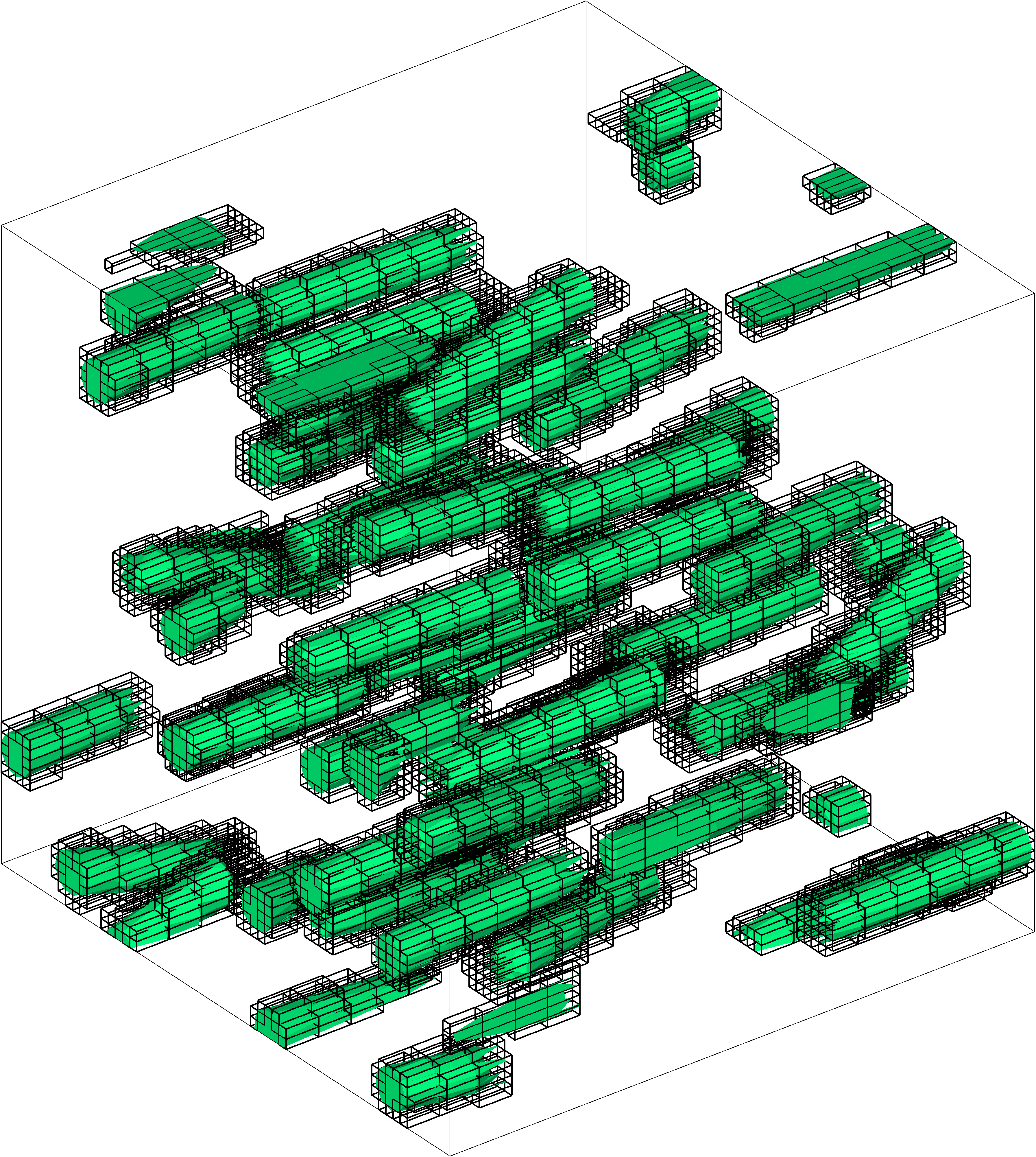}
         \caption{Fiber-reinforced composite RVE - [FRP-1]}
         \label{fig:ex:frp1}
    \end{subfigure}
    \hfill
    \begin{subfigure}[b]{0.45\textwidth}
         \centering
         \includegraphics[width=\textwidth]{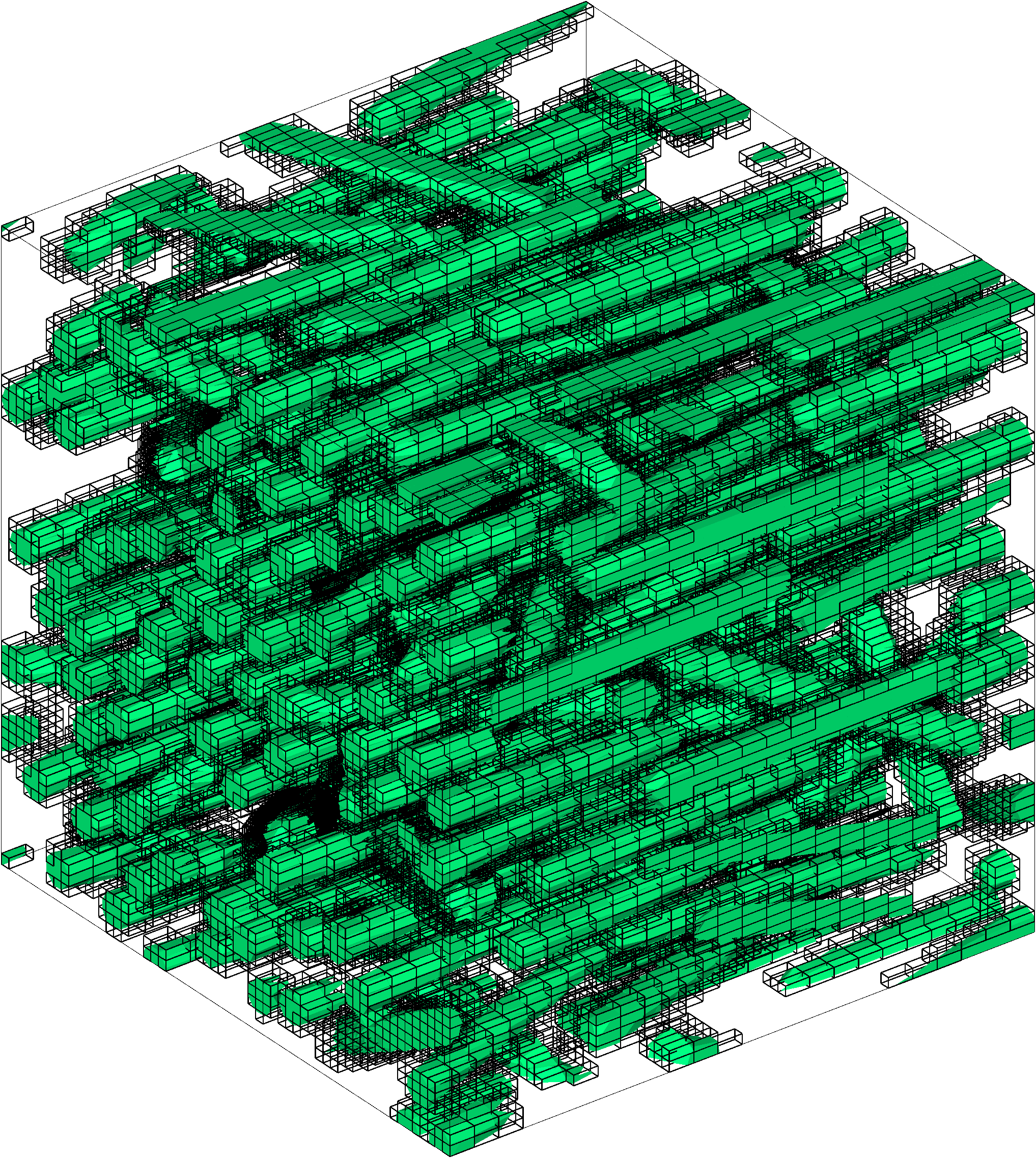}
         \caption{Fiber-reinforced composite RVE - [FRP-2]}
         \label{fig:ex:frp2}
    \end{subfigure}
    \begin{subfigure}[b]{0.45\textwidth}
         \centering
         \includegraphics[width=\textwidth]{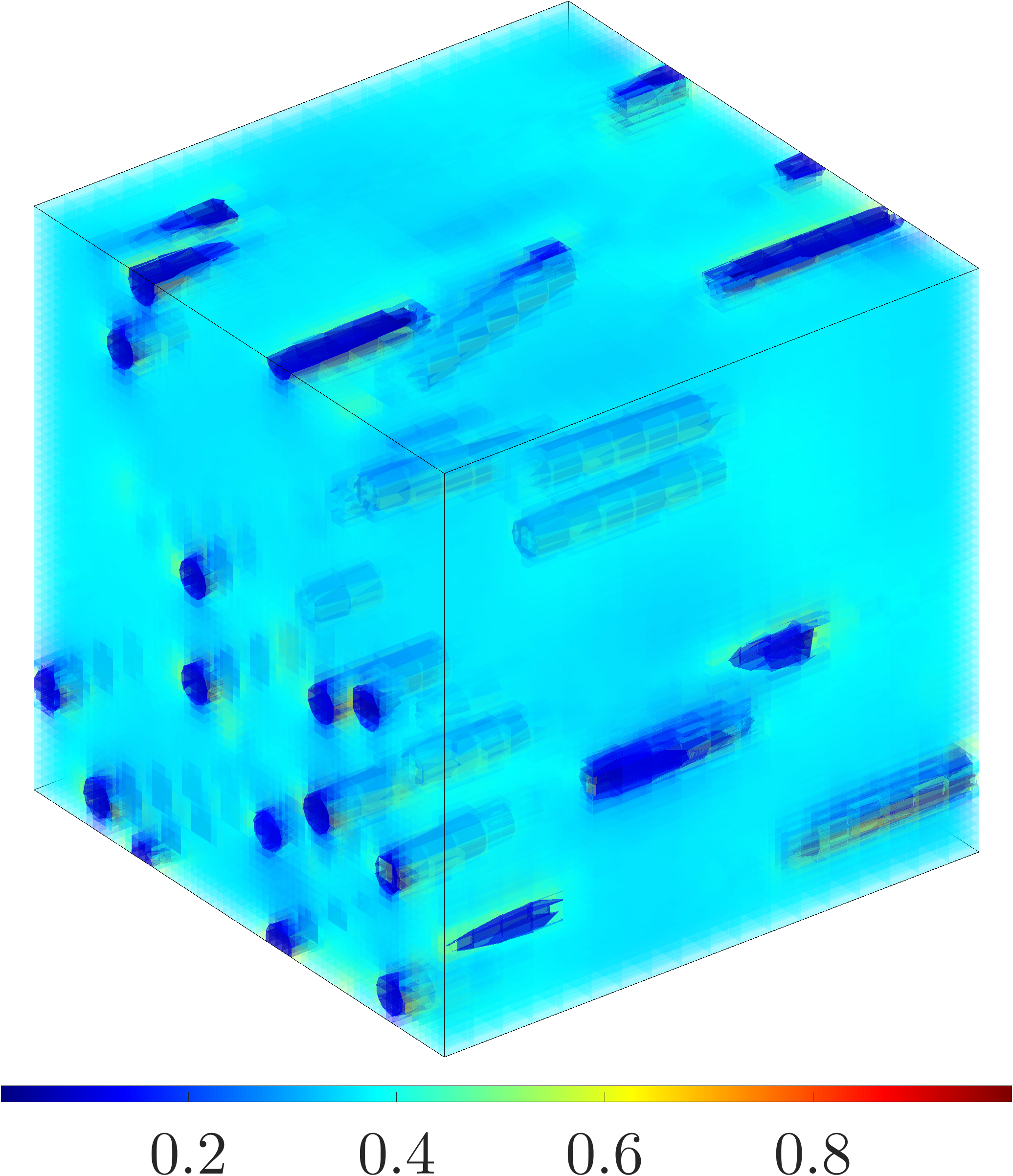}
         \caption{Green-Lagrange strain $E_{XY}$ - [FRP-1]}
         \label{fig:ex:frp1_strain}
    \end{subfigure}
    \hfill
    \begin{subfigure}[b]{0.45\textwidth}
         \centering
         \includegraphics[width=\textwidth]{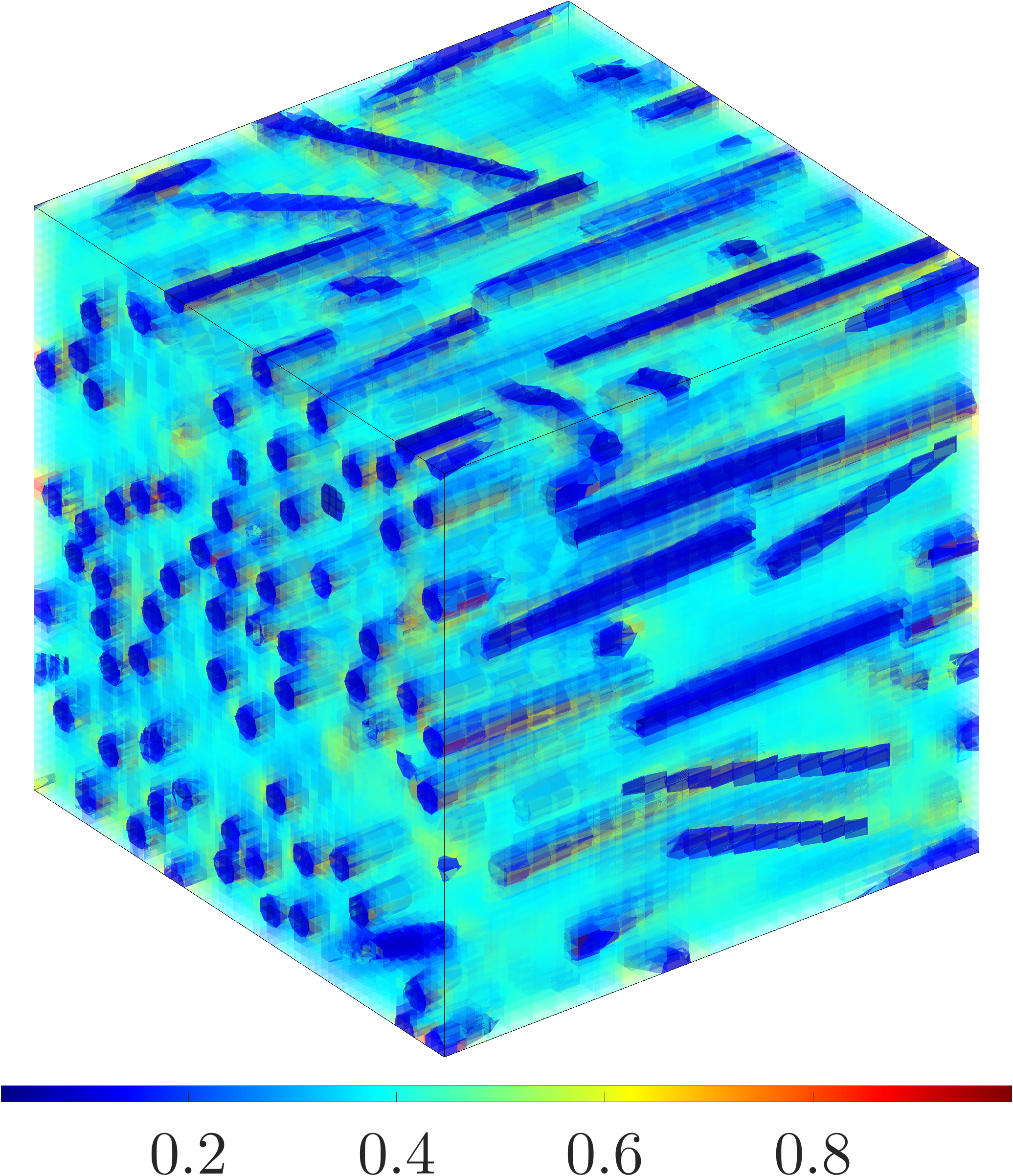}
         \caption{Green-Lagrange strain $E_{XY}$ - [FRP-2]}
         \label{fig:ex:frp2_strain}
    \end{subfigure}
    \caption{Fiber-reinforced composites: [FRP-1] $\sim 4\%$ fiber volume fraction (a) [FRP-2] $\sim 16\%$ fiber volume fraction (b); simulation results using FANS HEX8R with ComBo discretization 18$\times$63$\times$63 for [FRP-1] and 30$\times$60$\times$60 for [FRP-2] are shown in (c) and (d).}
    \label{fig:ex:FRP}
\end{figure*}

Fig.~\ref{fig:ex:frp1_strain} and \ref{fig:ex:frp2_strain} show that the ComBo discretization leads to reasonable local stress fields despite the use of massively anisotropic grids. Thereby, the resolution of the simulation can adapt to the aspect ratio of the fibers, allowing for distinct computational gains while, at the same time, the lateral resolution can remain sufficiently fine to gain (a) accurate orientation data (cf. Section~\ref{sec:normal}) and (b) to separate individual fibers. Both can be achieved without growing the number of degrees of freedom.
For instance, [FRP-1] is resolved $3.5$ times coarser along the fiber axis, yielding a speed-up and memory savings in the same order against an equiaxed grid. For [FRP-2] several ComBo discretizations are compared against a reference solution regarding $\ol{\fP}, \ol{\fP}_\cvpm$ in Table~\ref{tab:ex:FRP:P1}. The overall accuracy was better than 1\% with improvements independent of the boxel aspect ratio as the number of boxels grows. The accuracy within the inclusion phase is impressive, given that the fibers make up only $\sim 4 \%$ of the material. Similar trends are observed in Table~\ref{tab:ex:FRP:P2} for the [FRP-2] problem where the use of non-equiaxed boxels become vital in the need to resolve very small fiber distances.
Much coarser non-equiaxed resolutions still perform better than corresponding equiaxed resolution outlining the need to resolve the lateral plane of the fibers properly. It can also be observed that modest improvements are seen when the resolution along the fibers are increased while much more substantial improvements are observed when the fiber lateral plane is refined.

We also extract the Von-Mises Cauchy stress: overall and phase-wise histograms for the [FRP-2] problem as shown in Figure~\ref{fig:ex:FRP2-ST}, in order to judge the quality of the full field solution. The stress distributions for the equiaxed $48^3$ (Downscale : 125) and the non-equiaxed $30\times60\times60$ (Downscale: 128) is compared against a reference solution $240^3$. Although by the metric of vision we can observe a better agreement with the non-equiaxed resolution against the reference solution, the corresponding cumulative squared deviations of the distributions from the reference stress histogram are plotted which unequivocally states that the non-equiaxed resolution has lower error accross the board and thus captures the over-all field solution better.

\begin{figure}[h!]
\centering
\begin{subfigure}[b]{\columnwidth}
         \centering \captionsetup{justification=centering}
         \includegraphics[width=\textwidth]{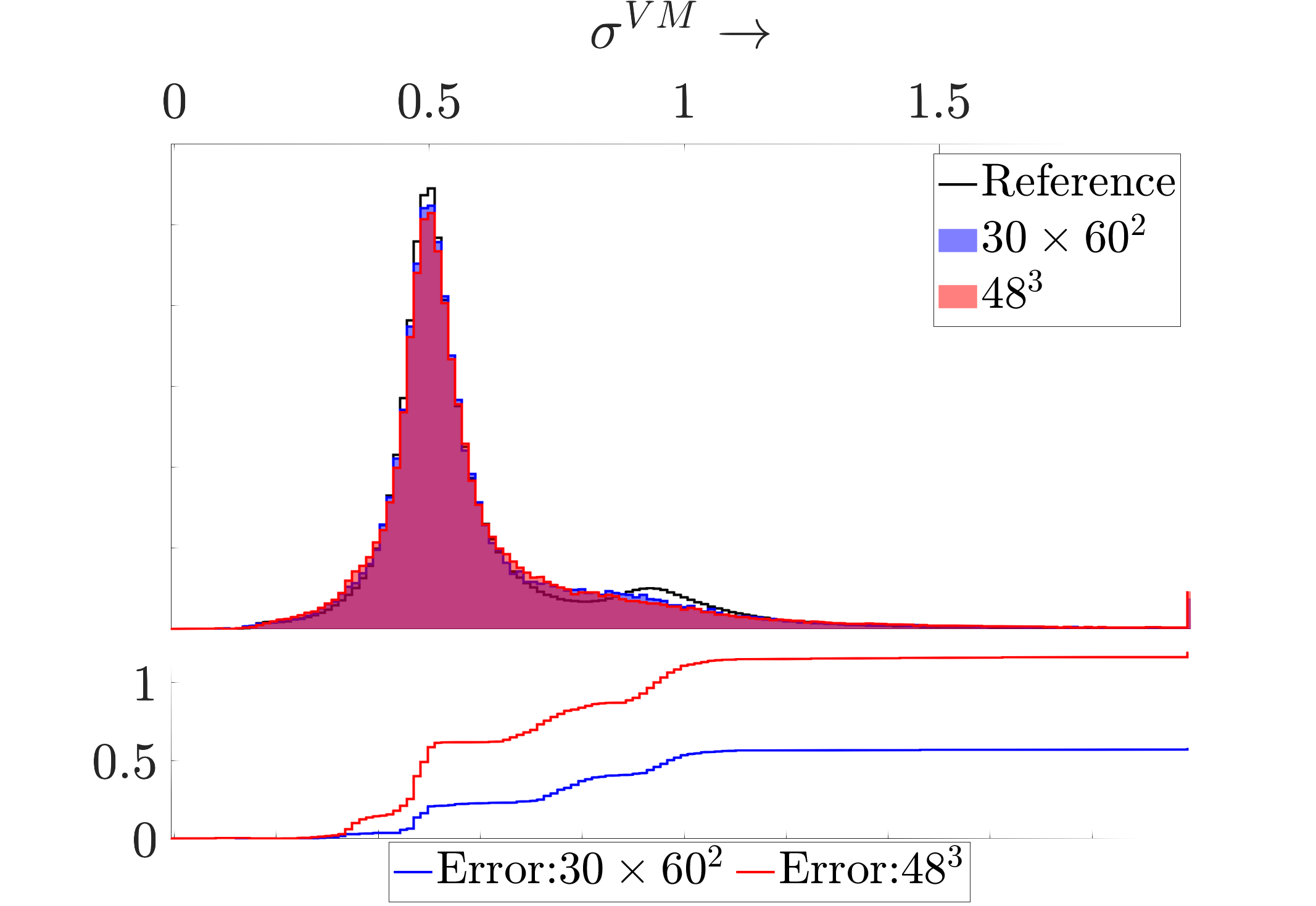}
        \label{fig:ex:frp2-STo}
    \end{subfigure}
    \begin{subfigure}[b]{\columnwidth}
         \centering \captionsetup{justification=centering}
         \includegraphics[width=\textwidth]{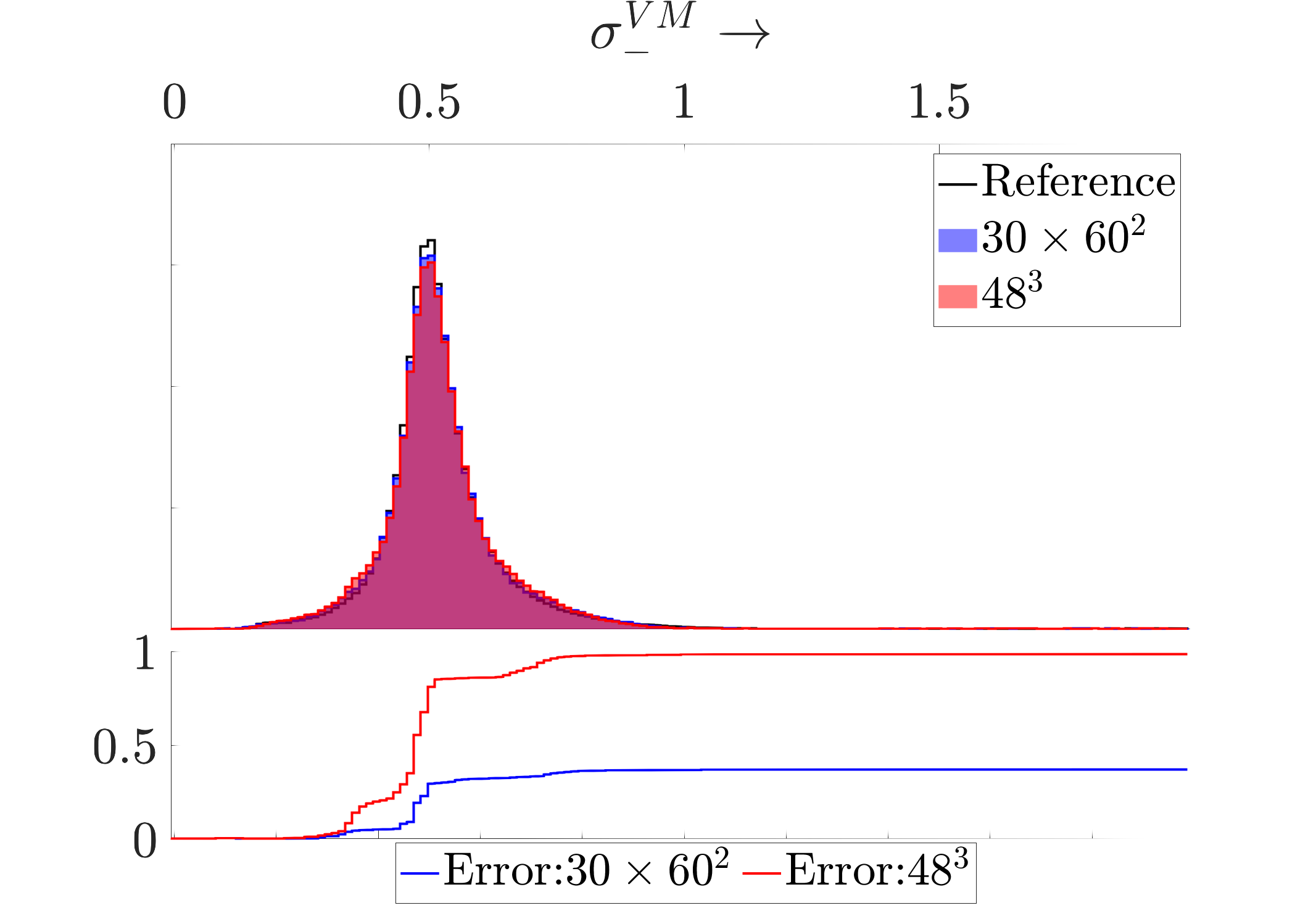}
         \label{fig:ex:frp2-ST1}
    \end{subfigure}
    \begin{subfigure}[b]{\columnwidth}
         \centering \captionsetup{justification=centering}
         \includegraphics[width=\textwidth]{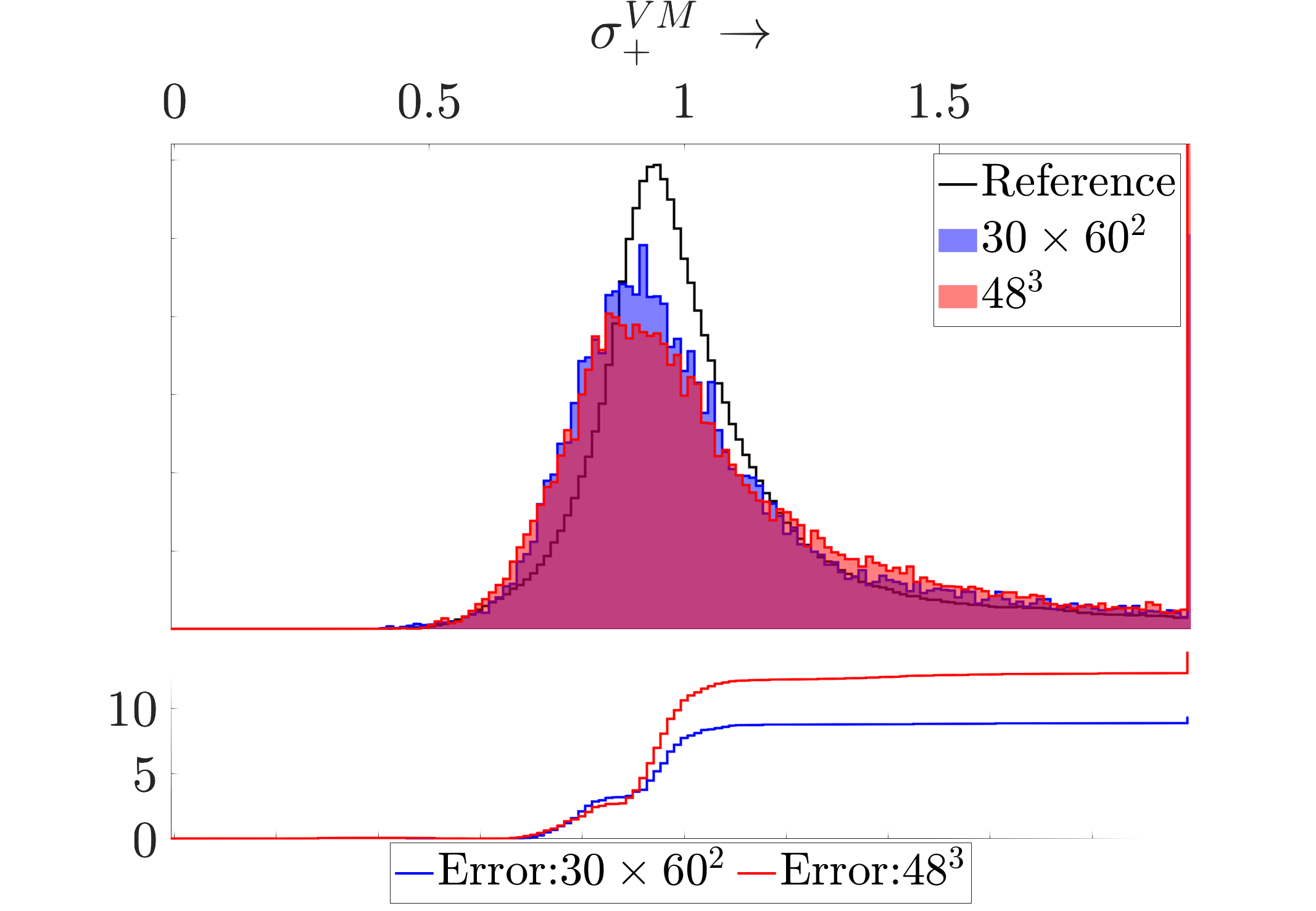}
         \label{fig:ex:frp2-ST2}
    \end{subfigure}
    \caption{Normalized weighted distribution of the Von-Mises cauchy stress (top) and the squared cumulative error (bottom) of the distribution w.r.t a reference ($240^3$) solution}
    \label{fig:ex:FRP2-ST}
\end{figure}

\begin{table*}[!h]
\begin{center}
\begin{minipage}{\textwidth}
\caption{\protect Comparing averaged 1\textsuperscript{st} Piola-Kirchoff stresses in the [FRP-1] problem cf. Section~\ref{subsec:complex}: different equiaxed and non-equiaxed ComBo resolutions, with errors against a reference solution $(504^3)$ based on \cite{Willot2015}; ComBo solutions are obtained using FANS with reduced integration \cite{Leuschner2018}; gray background highlights better result (equiaxed vs. boxels): Error defined as the relative Frobenius norm w.r.t reference solution
}\label{tab:ex:FRP:P1}
\scriptsize
\begin{tabular*}{\textwidth}{@{\extracolsep{\fill}}crrrr lrrrcl@{\extracolsep{\fill}}}
\toprule%
\multicolumn{5}{c}{Equiaxed voxels} & \multicolumn{5}{c}{non-equiaxed boxels}\\
\cmidrule(lr){1-5}
\cmidrule(lr){6-11} 
 &\multicolumn{3}{c}{Error (\%)} & & & \multicolumn{3}{c}{Error (\%)} & &\\
\cmidrule(lr){2-4}
\cmidrule(lr){7-9} 
Res &  {$\overline{\fP}$} & {$\overline{\fP}_\cvm$} & {$\overline{\fP}_\cvp$} & \multicolumn{2}{c}{Downscale} & {$\overline{\fP}$} & {$\overline{\fP}_\cvm$} & {$\overline{\fP}_\cvp$} & Res & Aspect ratio\\
\midrule
 $21^3$ & 0.894 & 0.148 & 13.916 & \bf 13824  & \bf 12348 & \cellcolor{uniSgray20} 0.654 &\cellcolor{uniSgray20}  0.110 &\cellcolor{uniSgray20}  10.205 &  $18\times 24^2$& $4:3:3$\\ 
 $24^3$ & 0.570 & 0.097 & 8.913  & \bf 9261   & \bf 9072  &\cellcolor{uniSgray20}  0.461 &\cellcolor{uniSgray20}  0.080 & \cellcolor{uniSgray20} 7.219  &  $18\times 28^2$& $14:9:9$\\ 
 $28^3$ & 0.351 & 0.061 & 5.503  & \bf 5832   & \bf 5488  & \cellcolor{uniSgray20} 0.298 & \cellcolor{uniSgray20} 0.051 &\cellcolor{uniSgray20}  4.662  &  $18\times 36^2$& $4:2:2$\\ 
 $36^3$ & 0.189 & 0.033 & 2.975  & \bf 2744   & \bf 3024  &\cellcolor{uniSgray20}  0.168 & \cellcolor{uniSgray20} 0.030 & \cellcolor{uniSgray20} 2.650  &  $24\times 42^2$& $7:4:4$\\ 
 $42^3$ & 0.119 & 0.021 & 1.871  & \bf 1728   & \bf 1792  &\cellcolor{uniSgray20}  0.075 &\cellcolor{uniSgray20}  0.017 &\cellcolor{uniSgray20}  1.222  &  $18\times 63^2$& $7:2:2$\\ 
 $63^3$ & \cellcolor{uniSgray20} 0.035 & 0.007 & \cellcolor{uniSgray20} 0.555  & \bf 512    & \bf 504   & 0.036 & \cellcolor{uniSgray20} 0.006 & 0.564  &  $36\times 84^2$& $7:3:3$\\ 
 $84^3$ &\cellcolor{uniSgray20}  0.037 & \cellcolor{uniSgray20} 0.006 & \cellcolor{uniSgray20} 0.579  & \bf 216    & \bf 224   & 0.049 & 0.007 & 0.744  &  $36\times 126^2$ & $7:2:2$\\
\botrule
\end{tabular*}
\label{table:1pk_comparision-FRP1}
\end{minipage}
\end{center}
\end{table*}

\begin{table*}[!h]
\begin{center}
\begin{minipage}{\textwidth}
\caption{\protect Comparing averaged 1\textsuperscript{st} Piola-Kirchoff stresses in the [FRP-2] problem cf. Section~\ref{subsec:complex}: different equiaxed and non-equiaxed ComBo resolutions, with errors against a reference solution $(240^3)$ based on \cite{Willot2015}; ComBo solutions are obtained using FANS with reduced integration \cite{Leuschner2018}; 
gray background highlights better result than corresponding equiaxed voxels resolution: 
Error defined as the relative Frobenius norm w.r.t reference solution
}\label{tab:ex:FRP:P2}
\scriptsize
\begin{tabular*}{\textwidth}{@{\extracolsep{\fill}}crrrr lrrrcl@{\extracolsep{\fill}}}
\toprule%
\multicolumn{5}{c}{Equiaxed voxels} & \multicolumn{5}{c}{non-equiaxed boxels}\\
\cmidrule(lr){1-5}
\cmidrule(lr){6-11} 
 &\multicolumn{3}{c}{Error (\%)} & & & \multicolumn{3}{c}{Error (\%)} & &\\
\cmidrule(lr){2-4}
\cmidrule(lr){7-9} 
Res &  {$\overline{\fP}$} & {$\overline{\fP}_\cvm$} & {$\overline{\fP}_\cvp$} & \multicolumn{2}{c}{Downscale} & {$\overline{\fP}$} & {$\overline{\fP}_\cvm$} & {$\overline{\fP}_\cvp$} & Res & Aspect ratio\\
\midrule
 &&&&& \bf 300 & \cellcolor{uniSgray20} 0.870 & \cellcolor{uniSgray20} 0.186 & \cellcolor{uniSgray20} 3.759 & $20 \times 48^2$ & $12:5:5$\\ 
 &&&&& \bf 250 & \cellcolor{uniSgray20} 0.783 & \cellcolor{uniSgray20} 0.169 & \cellcolor{uniSgray20} 3.392 & $24 \times 48^2$ & $2:1:1$\\ 
 $ 40^3$ & 1.110 & 0.234 & 4.806 & \bf216  & \bf 200  &  \cellcolor{uniSgray20} 0.726 &  \cellcolor{uniSgray20} 0.157 &  \cellcolor{uniSgray20} 3.149 &  $30\times 48^2$& $8:5:5$ \vspace{4pt}\\ 
 
  &&&&& \bf 192  & \cellcolor{uniSgray20} 0.375 & \cellcolor{uniSgray20} 0.081 & \cellcolor{uniSgray20} 1.586  &  $20 \times 60^2$ & $3:1:1$ \\
  &&&&& \bf 160  & \cellcolor{uniSgray20} 0.299 & \cellcolor{uniSgray20} 0.070 & \cellcolor{uniSgray20} 1.283  &  $24\times 60^2$& $5:2:2$\\ 
  $48^3$ & 0.688 & 0.146 & 2.982 & \bf 125 & \bf 128 & \cellcolor{uniSgray20} 0.238 & \cellcolor{uniSgray20} 0.059 & \cellcolor{uniSgray20} 1.041 &  $30 \times 60^2$ & $2:1:1$ \vspace{4pt} \\
  
  &&&&& \bf 108 & 0.254 & \cellcolor{uniSgray20} 0.039 & 0.973 & $20 \times 80^2$ & $4:1:1$ \\
  &&&&& \bf 96 & 0.228 & 0.055 & 1.004 & $40 \times 60^2$ & $3:2:2$ \\
  $60^3$ & 0.221 & 0.048 & 0.961 & \bf 64 & \bf 72 & \cellcolor{uniSgray20} 0.111 & \cellcolor{uniSgray20} 0.027 & \cellcolor{uniSgray20} 0.451 & $30 \times 80^2$ & $8:3:3$ \\

 
  
\botrule
\end{tabular*}
\label{table:1pk_comparision-FRP2}
\end{minipage}
\end{center}
\end{table*}




\section{Résumé}\label{sec:discussion} 

\subsection{Summary}
We present an extension of the composite voxel/boxel (anisotropic voxel) approach of \cite{Kabel2015} towards finite strain hyperelasticity for FFT-based homogenization schemes similar to \cite{Kabel2016}. The foundations of FFT-based homogenization are recalled in Section~\ref{sec:fft} and a detailed description of the doubly-fine material grid which can rule out some issues of the staggered grid approach cf. \cite{Schneider2016} regarding the local field accuracy is outlined in Appendix~\ref{app:staggeredGrid}.

The detailed algorithmic treatment of the composite voxels/boxels in Section~\ref{sec:combo} yields low-, i.e., $d$-dimensional nonlinear equations to be solved with explicit Hessians being provided for infinitesimal and finite strain problems; see also the cheat sheet in Appendix~\ref{app:mechanics}. The algorithmic tangent operator of the composite voxels is provided, too, and it has a sleek representation with a simplistic implementation.
 A crucial ingredient in composite boxel finite strain simulations is the back-projection scheme described in Algorithm~\ref{alg:backprojection}. It ensures admissibility of the deformation in either of the laminate phases at negligible computational overhead but much increased robustness.

When examining composite voxels and boxels, some issues with the normal detection cf. \cite{Kabel2015} were found. In Section~\ref{sec:normal}, a new algorithm for the normal identification is suggested, which leads to considerable improvements of the laminate orientation within the composite voxels. The new algorithm can also process composite boxels (ComBo) characterized by non-equiaxed coarsening which was impossible using the approach by \cite{Kabel2015}. The new procedure was shown to yield accurate normals for different microstructures. An open-source \texttt{python} implementation with examples can be accessed through the Github repository \cite{github:combo:normal}. It also features 3D tools for the visualization and a tutorial demonstrating the usage.

In Section~\ref{sec:examples} a variety of different microstructures are simulated using different FFT-based solvers, different normal detection procedures, and using different coarse-grained resolutions. The results demonstrate that the local fields using the ComBo discretization are closely matching full resolution solutions. FANS HEX8 and HEX8R \cite{Leuschner2018,Schneider2017} were found to yield the smoothest representation of the local fields. Despite the tendency of HEX8 to overestimate the stresses, this discretization has the smoothest stress fields. Moreover, the new normal detection was shown to provide notably improved accuracy on the overall stress response as well as for the phase-wise averaged stresses (see Tables~\ref{tab:ex:sphere:P1}-\ref{tab:ex:FRP:P2}). The improvement for actual boxels was even more notable. Further, the stress statistics for equiaxed and non-equiaxed resolutions with similar downscale factors hint at improved quality of the local stresses for the same downscale factor. Additionally, the new normals were shown to deliver convergence of the effective stress that depends mainly on the number of DOF, i.e., the overall amount of coarse-graining cf. Table~\ref{tab:ex:FRP:P1}-\ref{tab:ex:FRP:P2}. Computational savings of 2\,000 and beyond at errors around 1\% in the phase-averaged stresses are observed.

\subsection{Discussion}\label{subsec:discussion}
First up, the authors are thoroughly convinced that composite boxels have proven to be a valuable addition to many established FFT-based homogenization schemes. They allow for impressive computational savings in CPU time and memory (factor 2\,000 and beyond) at a modest -- if any -- sacrifice in accuracy. The accuracy was further improved by using the novel strategy for the normal detection from Section~\ref{sec:normal}. It leads to a reduction of approximately 30\% in the relative errors of the effective stress~$\ol{\fP}$ and its phase-wise counterparts $\ol{\fP}_\cvpm$ even for relatively smooth and simple microstructures. In the case of anisotropic boxels, more distinct improvements in the error were found. Surprisingly, in the presence of composite boxels, the number of DOF of the system seems to be the primary influence factor regarding the accuracy of the simulation even when pronounced boxel anisotropy is considered. A key advantage of non-equiaxed boxels is that pseudo-unidirectional fiber separation can be granted without growing the number of degrees of freedom of the problem while retaining accuracy.

We think that this can leverage the simulations in certain fields, e.g., for discontinuous short fiber composites with pronounced aspect ratios (e.g., 20 and beyond) and pseudo unidirectional fiber orientation. We are also convinced that making the source code for the normals freely available \cite{github:combo:normal} could help in rendering composite boxels an attractive choice in academia and industry. In the future, we are confident that more refined comparisons of the actual solution fields of ComBo and high resolution simulations will lead to further insights regarding accuracy and overall efficiency.

By the introduction of the doubly-fine material grid for the staggered grid discretization \cite{Schneider2016}, considerable improvements with respect to the quality of the local fields were observed. However, these come at the expense of a distinct rise in the number of constitutive evaluations and little gain regarding the overall homogenized response. Therefore, this method is probably best suited when local solution fields are sought-after. In this regard, FANS \cite{Leuschner2018}, or the equivalent FFT-$Q_1$ Hex \cite{Schneider2017} show the most confidence-inspiring results: local fields are smooth and match the reference solution closely; hour-glassing is less distinct (HEX8R) or absent (HEX8) than in the almost identical discretization of \cite{Willot2015}. It is important to state that the use of the ComBo discretization comes without issues.

The application of composite voxels/boxels in finite strain homogenization problems revealed that particular care should be taken in view of considering physical constraints: The selective  back-projection algorithm (Algorithm~\ref{alg:backprojection}) demonstrates that robustness can be gained and (sometimes unrecognized) physically questionable iterates might occur in practice. We are confident that the presented algorithm is a  leap forward. Despite this improvement, the realizable loading can still be limited, particularly when the phase volume fractions within the composite boxels tend towards 0 or 1 and the contrast in stiffness is pronounced. This can imply that the deformation gradient can approach critical states not just in the laminate but on the overall ComBo voxel (denoted $\fF_{\square}$). Finding further refinements to the simulation scheme could further boost robustness, giving rise to future research topics.

An important message is also given by demonstrating the usefulness of ComBo discretizations for a rich set of different discretizations: FANS HEX8(R)/FFT-$Q_1$ Hex \cite{Leuschner2018,Schneider2017}, DFMG, staggered grid \cite{Schneider2016} the rotated grid scheme of Willot \cite{Willot2015}, and the classical Moulinec-Suquet scheme~\cite{MoulinecSuquet1994,MoulinecSuquet1998} were all used with success and building on the same implementation. The authors would like to emphasize that the approach is, however \textit{not} limited to FFT-based schemes: regular Finite Element and Finite Difference schemes could use them, too, yielding potential benefits without the intrusiveness of, e.g., the extended finite element method (X-FEM, e.g., \cite{Loehnert:2011}). 

Related to the recent progress reported by \cite{Chen:2021}, the extension of our framework for interface mechanics in the small and finite strain setting is a promising route. Major benefits due to the improved normal orientations cf. Section~\ref{sec:normal} are expected: Both, the local solution fields as well as the (thereby influence) interfacial tractions are assumed to gain in accuracy.

Last, an equivalent to composite boxels for materials with more than two phases are urgently needed, e.g., in order to deal with polycrystals.

\backmatter

\bmhead{Supplementary information}
The normal detection algorithm is available from \cite{github:combo:normal}.

\bmhead{Acknowledgments}
Funded by Deutsche Forschungsgemeinschaft (DFG, German Research Foundation) under Germany's Excellence Strategy - EXC 2075 – 390740016. Contributions by Felix Fritzen are funded by Deutsche Forschungsgemeinschaft (DFG, German Research Foundation) within the Heisenberg program - DFG-FR2702/8 - 406068690. We acknowledge the support by the Stuttgart Center for Simulation Science (SimTech).

\section*{Declarations}
The authors declare no potential conflict of interest.

\begin{appendices}


\section{A finite difference discretization on a staggered grid} \label{app:staggeredGrid}
\label{subsec:finite_difference_staggered}
In the sequel, we describe the consistent formulation of the staggered grid discretization \cite{Schneider2016} for the geometrically nonlinear case \cite{Ospald2015}. Similar to the small strain case, the diagonal and off-diagonal components of the deformation gradient are located at different positions, compare Fig. \ref{fig:grids}. It is shown that ignoring this leads to unsatisfactory results which can be significantly improved by using a doubly-fine material grid (DFMG). In contrast to the linear elastic small strain regime, however, the DFMG needs an increased number of evaluations of the nonlinear material law.

Fix positive integers $n_1,n_2,n_3$ and consider a regular periodic grid consisting of $n = n_1 n_2 n_3$ cells, each a translate of $[0,\, h_1]\times[0,\, h_2]\times[0,\, h_3]$ with $h_j=\frac{l_j}{n_j}$ ($j=1,2,3$). Let $i^+,j^+$ and $k^+$ denote the location of the staggered grid coordinates, i.e., $a^+=a+\tfrac{1}{2}$ for $a=i,j,k$. For the $\RINDEX{i}{j}{k}$-cell the coordinates of the displacement vector $\RINDEX{U_1}{U_2}{U_3}$ are located at the cell face centers, i.e., $U_1\SINDEX{i}{j}{k}$ is located at the coordinate $\RINDEX{ih_1}{j^+ h_2}{k^+ h_3}$, $U_2\SINDEX{i}{j}{k}$ lives at $\RINDEX{i^+ h_1}{jh_2}{k^+ h_3}$ and $U_3\SINDEX{i}{j}{k}$ is associated to $\RINDEX{i^+ h_1}{j^+ h_2}{kh_3}$. The situation is displayed in Figure~\ref{fig:grids}~(a), where we have restricted ourselves to the $2$D case for clarity.

\begin{figure}
	\centering
	\begin{subfigure}[t]{.45\textwidth}
		\centering
		\includegraphics[height=.9\textwidth]{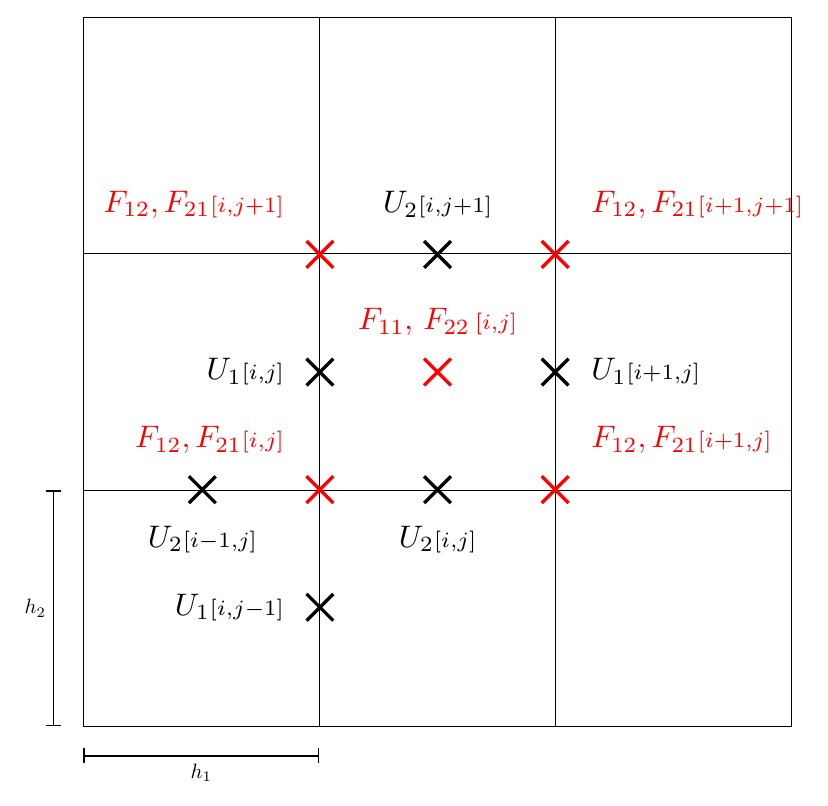}
		\caption{}
		\label{fig:variable_placement}
	\end{subfigure}%
	\hspace{0.05\textwidth}
	\begin{subfigure}[t]{.45\textwidth}
		\centering
		\includegraphics[height=.9\textwidth]{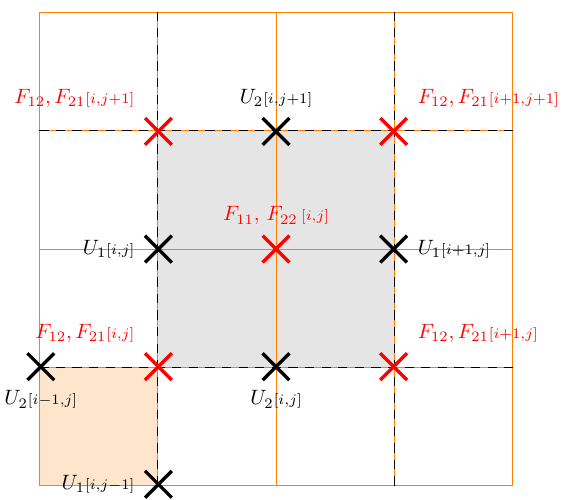}
		\caption{}
		\label{fig:material_grid}
	\end{subfigure}
	\caption{Placement of the strain and displacement variables in $2$D: (a) Location of strain and displacement variables. (b) The variables' grid (gray box) and the doubly fine material grid (orange box).}
	\label{fig:grids}
\end{figure}

\smallskip
The diagonal components $F_{aA}\SINDEX{i}{j}{k}$ ($a, A \in \{1,2,3\}$) of the deformation gradient $\fF$ are positioned on the cell centers $\RINDEX{i^+ h_1}{j^+ h_2}{k^+ h_3}$, whereas the off-diagonal strains $F_{23},F_{32}\SINDEX{i}{j}{k}$, $F_{13},F_{31}\SINDEX{i}{j}{k}$ and $F_{12},F_{21}\SINDEX{i}{j}{k}$, are located on the corresponding edge midpoints 
$\RINDEX{i^+ h_1}{jh_2}{kh_3}$, $\RINDEX{ih_1}{j^+ h_2}{kh_3}$ and $\RINDEX{ih_1}{j}{(k^+ h_3}$. For visualization, we again refer to Figure~\ref{fig:grids}~(a).

\smallskip
Displacements and gradients are connected by central difference formulae, where periodicity is understood implicitly. More precisely, introduce forward and backward difference operators on a scalar discrete field $\phi:V_n\rightarrow \ffR$ by the formulae
\begin{equation}\label{eq:forward_backward_difference}
	D_j^\pm\phi [I]=\pm\frac{\phi[I\pm e_j]-\phi [I]}{h_j}
\end{equation}
for $I\in V_n=\left\{0,1,\ldots,n_1-1\right\}\times\left\{0,1,\ldots,n_2-1\right\}\times\left\{0,1,\ldots,n_3-1\right\}$. Then we introduce the gradient operator 
\begin{equation}
	\label{eq:gradient_operator}
	\Grad \fU = \left[ 
	\begin{array}{ccc}
		D^+_1 U_1 & D^-_2 U_1 & D^-_3 U_1\\
		D^-_1 U_2 & D^+_2 U_2 & D^-_3 U_2\\
		D^-_1 U_3 & D^-_2 U_3 & D^+_3 U_3\\
	\end{array}
	\right]
\end{equation}
giving rise to the deformation gradient $\fF=I+\Grad \fU$ associated to a periodic displacement field
\[
\fU=(U_1,U_2,U_3):V_n\rightarrow \ffR^3,
\]
Similarly, there is a divergence operator, turning $\fP:V_n\rightarrow \ffR^{3\times 3}$ into
\begin{equation}
	\label{eq:div_operator}
	\mathrm{Div}\,\fP = \left[ 
	\begin{array}{c}
		D_1^- P_{11}+D_2^+ P_{21}+D_3^+ P_{31}\\
		D_1^+ P_{12}+D_2^- P_{22}+D_3^+ P_{32}\\
		D_1^+ P_{13}+D_2^+ P_{23}+D_3^- P_{33}
	\end{array}
	\right].
\end{equation}
The stress variable $P_{aB}$ is located on the same position as the corresponding $F_{aB}$.

The constitutive law $\fP(\fF)$ is defined piecewise on the regular voxel grid. For a typical material cell, shaded in gray in Figure~\ref{fig:grids}~(b), it is not clear how to apply the material law for off-diagonal strains, as the function $\fP(\fF)$ may be defined differently along the corresponding edge.

\begin{figure}
	\centering
	\begin{subfigure}[t]{.45\textwidth}
		\centering
		\includegraphics[height=.9\textwidth]{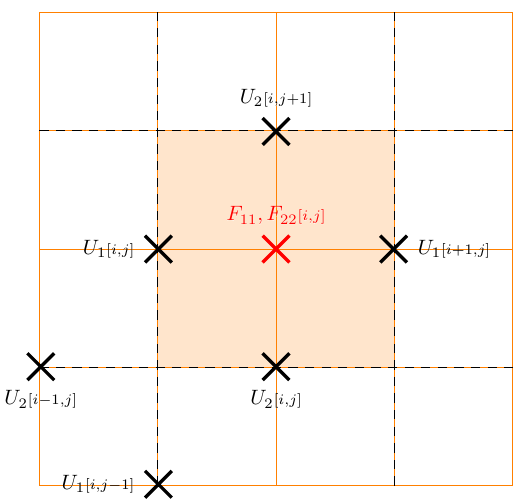}
		\caption{}
		\label{fig:F11_placement}
	\end{subfigure}%
	\hspace{0.05\textwidth}
	\begin{subfigure}[t]{.45\textwidth}
		\centering
		\includegraphics[height=.9\textwidth]{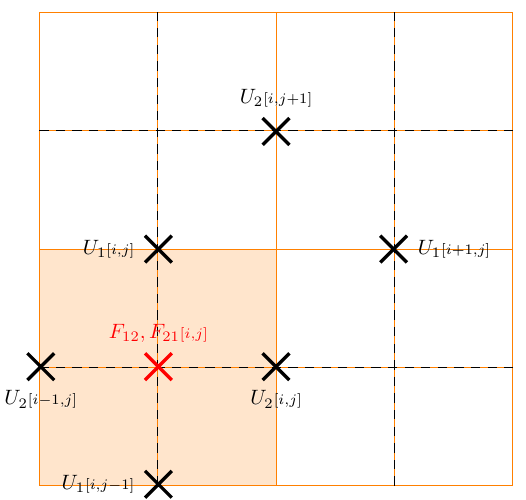}
		\caption{}
		\label{fig:F12_placement}
	\end{subfigure}
	\caption{Placement of deformation gradient variables within the doubly-fine material grid: (a) Location of $F_{11}$ (shaded in orange). (b) Location of $F_{12}$ (shaded in orange). The variables' grid is depicted by the dashed black lines.}
	\label{fig:deformationGradient_placement}
\end{figure}

\smallskip
We circumvent these problems by utilizing a doubly-fine grid, i.e., a grid with half the spacing of the original grid. Figure~\ref{fig:grids}~(b) illustrates this concept -- a typical doubly fine cell is shaded in orange. We interpret the deformation gradients and stresses as living on this doubly-fine grid. For every deformation gradient component $F_{iJ}$ and every doubly-fine cell there is precisely one $F_{iJ}$-value, as specified in Figure~\ref{fig:grids}~(a), located on the boundary of the cell. We associate this value to the doubly-fine cell. Thus, a particular value $F_{iJ}$ is distributed to the $4$ (in $2$D) or $8$ (in $3$D) adjacent doubly-fine cells, compare Figure~\ref{fig:deformationGradient_placement} for an illustration. The stress components are distributed similarly to the deformation gradients, i.e., the staggering of Figure~\ref{fig:deformationGradient_placement}, is also present for the stresses.

With this assignment, to any doubly-fine grid cell all $4$ (in $2$D) or $9$ (in $3$D) deformation gradient components are associated. The discretization just outlined directly carries over to three space dimensions easily. The variable placement in this case is shown in Figure~\ref{fig:variable_placement_3d}.
\begin{figure}
	\centering
	\includegraphics[height=.4\textwidth]{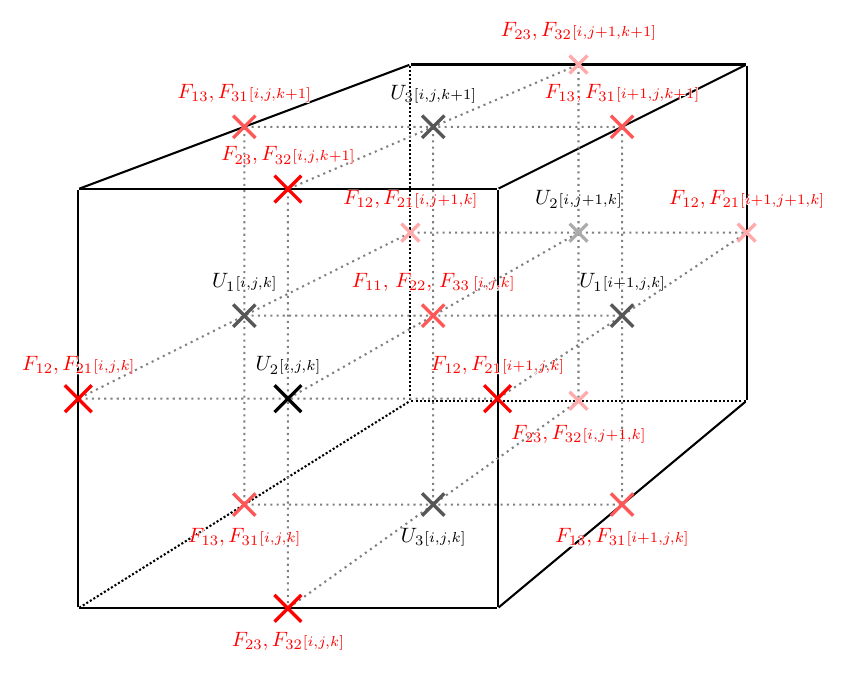}
	\caption{Placement of the deformation gradients and displacement variables in $3$D.}
	\label{fig:variable_placement_3d}
\end{figure}

\renewcommand{\fl}{\boldsymbol{l}}

We suppose that the constitutive material law is given on the original grid, i.e., each cell is associated with one material. We will elaborate on the implementation of the material law. By averaging over the combination of adjacent doubly-fine voxels, the material law can be written as
\begin{align}
	\label{eq:constitutive_law}
	P_{ab}[\fxi] &= \frac{1}{8} \sum_{\boldsymbol{l}\in\cL} 
	\lb\fP^{ab}_{\fl}[\fxi] \lb\fF^{ab}_{\fl}[\fxi]\rb\rb_{ab},
\end{align}
with 
\begin{align}
    \fxi &= [i,j,k] \\
    \cL &= \left\lbrace \fl = [l_1,l_2,l_3] \textrm{ with } l_1,l_2,l_3 \in\{ 0,1\} \right\rbrace
\end{align}
and
\begin{align*}
	\fP^{aa}_{\fl}[\fxi] &= \fP[\fxi], \, a\in\{ 1,2,3 \}, \\
	\fP^{ab}_{\fl}[\fxi] = \fP^{ba}_{\fl}[\fxi] &= \fP[\xi+\fl^{+a,+b}],\, a<b
\end{align*}
for
\begin{align*}
	\fl^{\pm1,\pm2} &= \fl^{\pm2,\pm1} = [\pm l_1,\pm l_2,0] \\
    \fl^{\pm1,\pm3} &= \fl^{\pm3,\pm1} = [\pm l_1,0,\pm l_3] \\
    \fl^{\pm2,\pm3} &= \fl^{\pm3,\pm2} = [0,\pm l_2,\pm l_3]
\end{align*}
as well as
\begin{align*}
	\fF^{aa}_{\fl}[\fxi] &= 
	\begin{pmatrix}
		F_{11}[\fxi] & F^{+1,+2}_{12,\fl}[\fxi] & F^{+1,+3}_{13,\fl}[\fxi] \\
		F^{+2,+1}_{21,\fl}[\fxi] & F_{22}[\fxi] & F^{+2,+3}_{23,\fl}[\fxi] \\
		F^{+3,+1}_{31,\fl}[\fxi] & F^{+3,+2}_{32,\fl}[\fxi] & F_{33}[\fxi] \\
	\end{pmatrix},\\
	\fF^{12}_{\fl}[\fxi] &= \begin{pmatrix}
		F_{11,\fl}^{-1,-2}[\fxi] & F_{12}[\fxi] & F_{13,\fl}^{-2,+3}[\fxi] \\
		F_{21}[\fxi] & F_{22,\fl}^{-1,-2}[\fxi]& F_{23,\fl}^{-1,+3}[\fxi] \\
		F_{31,\fl}^{-2,+3}[\fxi] & F_{32,\fl}^{-1,+3}[\fxi] & F_{33,\fl}^{-1,-2}[\fxi] \\
	\end{pmatrix},\\
	\fF^{21}_{\fl}[\fxi] &= \fF^{12}_{\fl}[\fxi], \\
	\fF^{13}_{\fl}[\fxi] &= 
	\begin{pmatrix}
		F_{11,\fl}^{-1,-3}[\fxi] & F_{12,\fl}^{+2,-3}[\fxi] & F_{13}[\fxi] \\
		F_{21,\fl}^{+2,-3}[\fxi] & F_{22,\fl}^{-1,-3}[\fxi]& F_{23,\fl}^{-1,+2}[\fxi] \\
		F_{31}[\fxi] & F_{32,\fl}^{-1,+2}[\fxi] & F_{33,\fl}^{-1,-3}[\fxi] \\
	\end{pmatrix},\\
	\fF^{31}_{\fl}[\fxi] &= \fF^{13}_{\fl}[\fxi], \\
	\fF^{23}_{\fl}[\fxi] &= 
	\begin{pmatrix}
		F_{11,\fl}^{-2,-3}[\fxi] & F_{12,\fl}^{+1,-3}[\fxi] & F_{13,\fl}^{+1,-2}[\fxi] \\
		F_{21,\fl}^{+1,-3}[\fxi] & F_{22,\fl}^{-2,-3}[\fxi]& F_{23}[\fxi] \\
		F_{31,\fl}^{+1,-2}[\fxi] & F_{32}[\fxi] & F_{33,\fl}^{-2,-3}[\fxi]
	\end{pmatrix},\\
	\fF^{32}_{\fl}[\fxi] &= \fF^{23}_{\fl}[\fxi],
\end{align*}
where
\begin{align*}
	F^{\pm a,\pm b}_{cd,\fl}[\fxi] &= F_{cd}[\fxi+\fl^{\pm a,\pm b}],\, a\neq b.
\end{align*}


\begin{sidewaystable*}[!h]
\sidewaystablefn \label{secA2}
\begin{center}
\begin{minipage}{\textheight}
\caption{Composite boxels for mechanics - Cheat sheet}
{\scriptsize
\begin{tabular*}{\textheight}{@{\extracolsep{\fill}}lcc@{\extracolsep{\fill}}}
\toprule%
 &\multicolumn{1}{c}{Finite strain theory}
&
\multicolumn{1}{c}{Infinitesimal strain theory} \\\hline \\[1ex] 

Hadamard jump & 
$\llbracket\fF\rrbracket_{\scrS^e} = \fF_\cvp - \fF_\cvm =  \lb \dfrac{1}{c_\cvp} + \dfrac{1}{c_\cvm}\rb \fa \otimes \fN$& 
$\llbracket\feps\rrbracket_{\scrS^e} = \feps_\cvp - \feps_\cvm =  \lb \dfrac{1}{c_\cvp} + \dfrac{1}{c_\cvm}\rb \fa \otimes^{\mathrm{s}} \fN$ \\[3ex]
Deformation & $\fF_\cvpm = \fF_{\square} \pm \dfrac{1}{c_\cvpm} (\fa \otimes \fN)$ &
$\feps_\cvpm = \feps_{\square} \pm \dfrac{1}{c_\cvpm} (\fa \otimes^\mathrm{s} \fN)$ \\[3ex]
Traction balance & $\llbracket\fP\rrbracket_{\scrS^e} \; \fN = \fzero$ &
$\llbracket\fsigma\rrbracket_{\scrS^e} \; \fN = \fzero$ \\[3ex]
Hessian for NR & $\fDelta_f = \fN \cdot \lb \dfrac{\ffA_\cvp}{c_\cvp} + \dfrac{\ffA_\cvm}{c_\cvm}\rb  \fN $  &
$\fDelta_f = \fN \cdot \lb \dfrac{\ffC_\cvp}{c_\cvp} + \dfrac{\ffC_\cvm}{c_\cvm}\rb  \fN $\\[3ex]
Update of $\fa$ & $\fa^{[k]} = \fa^{[k-1]} - \lb \fDelta_f^{[k-1]}\rb^{-1} f(\fa^{[k-1]})$ & $\fa =  	\fDelta_f^{-1}  \fN \cdot  \lb \ffC_\cvm - \ffC_\cvp \rb \feps_{\square} $ \\[3ex]
Stress & $\fP_\square = c_\cvp \fP_\cvp + c_\cvm \fP_\cvm$ & $\fsigma_{\square} = c_\cvp \fsigma_\cvp + c_\cvm \fsigma_\cvm$\\[3ex]
Tangent modulus & $\ffA_{\square} = \overline{\ffA}_{\mathrm{v}} 
	-  \delta \ffA \lb\fN \otimes	\fDelta_f^{-1}  \otimes \fN \rb \delta \ffA $ & $\ffC_{\square} = \overline{\ffC}_{\mathrm{v}} 
	-   \delta \ffC \lb\fN \otimes	\fDelta_f^{-1}  \otimes \fN \rb  \delta \ffC $\\[2ex] \hline \\

\textbf{Notation:} & & \\[1ex]
$\cvpm$ - material phases $\cvp$ and $\cvm$ in composite boxel &&\\
$c$ - volume fraction & & \\
$\fN$  - normal orientation of the laminate & & \\
$\fa$  - gradient jump vector $\fa \in \ffR^d$ &&\\
$\fF$  - deformation gradient &&\\
$\feps$ - infinitesimal strain &&\\
$\fP$ - first Piola-Kirchoff stress &&\\
$\fsigma$ - infinitesimal stress &&\\
$\fDelta_f$ - Newton-Raphson Jacobian $\fDelta \in \ffR^{d\times d}$ && \\
$\ffA$ - fourth order constitutive tangent $\dfrac{\partial \fP}{\partial \fF} = \dfrac{\partial^2 W}{\partial \fF \otimes \partial \fF}$ && \\
$\ffC$ - fourth order constitutive tangent $\dfrac{\partial \fsigma}{\partial \feps} = \dfrac{\partial^2 W}{\partial \feps \otimes \partial \feps}$ && \\

\botrule
\end{tabular*}}\label{app:mechanics}
\end{minipage}
\end{center}
\end{sidewaystable*}

\end{appendices}

\clearpage
\bibliography{bibliography}


\end{document}